\documentclass[a4paper, 12pt, final]{amsart}
\usepackage{nag}
\usepackage{setspace}
\usepackage{multicol}
\usepackage{multirow}
\usepackage[indentafter]{titlesec}
\usepackage{upgreek}
\usepackage[pdftex]{graphicx}
\usepackage{amsmath, amsfonts, amssymb, amsthm, float, bm}
\usepackage{scalerel}
\usepackage{mathtools}
\usepackage{mathrsfs}
\usepackage{boldline}
\usepackage{enumitem}
\usepackage[toc,page]{appendix}
\usepackage[a4paper, bottom=2cm]{geometry}
\usepackage{array}
\usepackage{xtab}
\usepackage{lmodern}
\usepackage{tcolorbox}
\usepackage{xcolor}
\usepackage{url}
\usepackage{tikz-cd}
\usepackage{bm}
\usepackage{tabularx}
\usepackage{ltablex} 
\usepackage{etoolbox}
\usepackage{tipa}
\usepackage[backend=biber, style=alphabetic]{biblatex}
\usepackage[urlcolor=black]{hyperref}
\hypersetup{colorlinks=true,linkcolor=black,citecolor=blue!80!black}
\usepackage[capitalize]{cleveref}
\renewbibmacro{in:}{}
\DeclareFieldFormat[article]{title}{#1}

\titleformat{name=\section}{}{\thetitle.}{0.8em}{\centering\scshape}
\titleformat{name=\subsection}[runin]{}{\thetitle.}{0.5em}{\bfseries}[.]
\titleformat{name=\subsubsection}[runin]{}{\thetitle.}{0.5em}{\itshape}[.]
\titleformat{name=\paragraph,numberless}[runin]{}{}{0em}{}[.]
\titlespacing{\paragraph}{0em}{0em}{0.5em}
\titleformat{name=\subparagraph,numberless}[runin]{}{}{0em}{}[.]
\titlespacing{\subparagraph}{0em}{0em}{0.5em}

\theoremstyle{plain}
\newtheorem{theorem}{Theorem}[section]
\newtheorem{proposition}[theorem]{Proposition}
\newtheorem{lemma}[theorem]{Lemma}
\newtheorem{corollary}[theorem]{Corollary}

\newtheorem{example}[theorem]{Example}
\newtheorem*{mtheorem}{Theorem A}

\newtheorem*{main}{Main Theorem}

\theoremstyle{definition}
\newtheorem{definition}[theorem]{Definition}

\theoremstyle{remark}

\AtBeginEnvironment{theorem}{\setlist[enumerate,1]{label=(\roman*),font=\upshape}}
\AtBeginEnvironment{mtheorem}{\setlist[enumerate,1]{label=(\roman*),font=\upshape}}
\AtBeginEnvironment{theorem*}{\setlist[enumerate,1]{label=(\roman*),font=\upshape}}
\AtBeginEnvironment{main}{\setlist[enumerate,1]{label=(\roman*),font=\upshape}}
\AtBeginEnvironment{example}{\setlist[enumerate,1]{label=(\roman*),font=\upshape}}
\AtBeginEnvironment{definition}{\setlist[enumerate,1]{label=(\roman*),font=\upshape}}
\AtBeginEnvironment{lemma}{\setlist[enumerate,1]{label=(\roman*),font=\upshape}}
\AtBeginEnvironment{proposition}{\setlist[enumerate,1]{label=(\roman*),font=\upshape}}
\AtBeginEnvironment{sublemma}{\setlist[enumerate,1]{label=(\roman*),font=\upshape}}
\AtBeginEnvironment{corollary}{\setlist[enumerate,1]{label=(\roman*),font=\upshape}}
\AtBeginEnvironment{theorem}{\setlist[enumerate,2]{label=(\alph*),font=\upshape}}
\AtBeginEnvironment{lemma}{\setlist[enumerate,2]{label=(\alph*),font=\upshape}}
\AtBeginEnvironment{proposition}{\setlist[enumerate,2]{label=(\alph*),font=\upshape}}
\AtBeginEnvironment{sublemma}{\setlist[enumerate,2]{label=(\alph*),font=\upshape}}

\makeatletter
\@namedef{subjclassname@2020}{MSC2020}
\makeatother

\makeatletter
\@namedef{keywordsname}{Keywords}
\makeatother

\renewcommand{\Gamma}{\varGamma}
\renewcommand{\epsilon}{\varepsilon}
\renewcommand{\bar}{\overline}
\renewcommand{\hat}{\widehat}
\renewcommand{\leq}{\leqslant}
\renewcommand{\geq}{\geqslant}

\newcommand{\normaleq}{\trianglelefteq}

\newcommand{\fs}{\mathcal{F}}

\newcommand{\N}{\mathbb{N}}

\newcommand{\SL}{\mathrm{SL}} 
\newcommand{\SU}{\mathrm{SU}} 
\newcommand{\GF}{\mathrm{GF}} 
\newcommand{\syl}{\mathrm{Syl}}
\newcommand{\GL}{\mathrm{GL}}
\newcommand{\Sp}{\mathrm{Sp}}

\newcommand{\PSL}{\mathrm{PSL}}
\newcommand{\PSp}{\mathrm{PSp}}
\newcommand{\PSU}{\mathrm{PSU}}

\newcommand{\Sz}{\mathrm{Sz}}

\newcommand{\Sym}{\mathrm{Sym}}

\newcommand{\Aut}{\mathrm{Aut}}
\newcommand{\Out}{\mathrm{Out}}
\newcommand{\Inn}{\mathrm{Inn}}

\newcommand{\Mor}{\mathrm{Mor}}
\newcommand{\Hom}{\mathrm{Hom}}
\newcommand{\Iso}{\mathrm{Iso}}
\newcommand{\Inj}{\mathrm{Inj}}
\newcommand{\Ob}{\mathrm{Ob}}

\def \wt {\widetilde}

\begin{document}

\title[Fusion Systems \& Rank $2$ Simple Groups of Lie type]{Saturated Fusion Systems on Sylow $\lowercase{p}$-subgroups of Rank $2$ Simple Groups of Lie type}
\author{Martin van Beek}
\thanks{The author is supported by the Heilbronn Institute for Mathematical Research. Part of this work was carried out at the Isaac Newton Institute for Mathematical Sciences during the programme GRA2, supported by the EPSRC (EP/K032208/1).}
\address{Alan Turing Building, University of Manchester, UK, M13 9PL}
\email{\href{mailto:martin.vanbeek@manchester.ac.uk}{martin.vanbeek@manchester.ac.uk}}

\keywords{Fusion systems; Groups of Lie type; Finite Simple Groups}
\subjclass[2020]{20D20; 20D05; 20E42}

\begin{abstract}
For any prime $p$ and $S$ a $p$-group isomorphic to a Sylow $p$-subgroup of a rank $2$ simple group of Lie type in characteristic $p$, we determine all saturated fusion systems supported on $S$ up to isomorphism.
\end{abstract}

\maketitle

\section{Introduction}\label{Intro}

We complete the classification of all saturated fusion systems supported on Sylow $p$-subgroups of rank $2$ groups of Lie type in characteristic $p$, unifying several results already in the literature. The rationale behind this undertaking comes from earlier results concerning fusion systems supported on Sylow $p$-subgroups of $\PSL_3(p)$, $\mathrm{G}_2(p)$ and $\mathrm{Sp}_4(p^n)$ which exposed exotic fusion systems which were previously unknown. Our main result below indicates these are the only instances when a Sylow $p$-subgroup of a rank $2$ group of Lie type can support an exotic fusion system.

\begin{main}\hypertarget{MainTheorem}{}
Suppose that $\fs$ is a saturated fusion system on a $p$-group $S$, where $S$ is isomorphic to a Sylow $p$-subgroup of some rank $2$ simple group of Lie type. Then $\fs$ is known. Moreover, if $\fs$ is exotic then $O_p(\fs)=\{1\}$ and one of the following holds:
\begin{enumerate}
    \item $S\cong 7^{1+2}_+$ and $\fs$ is described in \cite{RV1+2};
    \item $S$ is isomorphic to a Sylow $7$-subgroup of $\mathrm{G}_2(7)$ and $\fs$ is described in \cite{parkersem}; or
    \item $S$ is isomorphic to a Sylow $p$-subgroup of $\Sp_4(p^n)$ for $p$ an odd prime and $n\in\N$.
\end{enumerate}
\end{main}

We include the ``non-simple" versions of the rank $2$ groups of Lie type, as well as their derived subgroups i.e. $\PSp_4(2), \mathrm{G}_2(2), {}^2\mathrm{F}_4(2)$ and their derived subgroups. We note that our work provides evidence in support of a conjecture made by Parker and Semeraro \cite[Conjecture 2]{Comp1} on the rigidity of saturated fusion systems supported on Sylow $p$-subgroups of groups of Lie type in defining characteristic. 

The \hyperlink{MainTheorem}{Main Theorem} is a culmination of various results published over several years. The origins of these endeavors are in a paper of Ruiz and Viruel \cite{RV1+2}, which classified all saturated fusion systems supported on an extraspecial group of order $p^3$ and exponent $p$ (a Sylow $p$-subgroup of $\PSL_3(p)$). As eluded to above, they uncovered three exotic fusion systems supported on $7^{1+2}_+$. The generic case of fusion systems supported on a Sylow $p$-subgroup of $\PSL_3(p^n)$ for $n\in \N$ was examined in the thesis of Clelland \cite{Clelland}, yielding no further exotic examples. 

Next came the paper of Parker and Semeraro \cite{parkersem} on fusion systems supported on a Sylow $p$-subgroup of $\mathrm{G}_2(p)$, again for $p$ odd. They uncovered twenty-seven new exotic fusion systems, all supported on a Sylow $7$-subgroup of $\mathrm{G}_2(7)$. The next entry is the effort of Henke and Shpectorov in classifying fusion systems on a Sylow $p$-subgroup of $\Sp_4(p^n)$. The outcomes of this work are, as of yet, unpublished. By work of Parker and Stroth \cite{parkerstroth}, it is already known that there are infinitely many exotic fusion systems arising here, although they may all be suitably organized and understood.

The case of Sylow $p$-subgroups of $\PSU_4(p)$, for $p$ an odd prime, was handled disjointly in two papers: the case $p=3$ dealt with by Baccanelli, Franchi and Mainardis  \cite{Baccanelli}, and the cases $p\geq 5$ analyzed by Moragues-Moncho \cite{Raul}. No exotic fusion systems were found here. The final entry, up to the present paper, was work of the author classifying saturated fusion systems on Sylow $p$-subgroups of $\mathrm{G}_2(p^n)$ and $\PSU_4(p^n)$ for any prime $p$ and any $n\in\N$, subsuming several of the previously described cases. The methodology here was slightly different than what was employed in previous works, and more explicitly relied on certain uniqueness results pertaining to amalgams and weak BN-pairs arising in the context of fusion systems.

We complete this program with the determination of the remaining cases: saturated fusion systems on Sylow $p$-subgroups of ${}^2\mathrm{F}_4(2^n)$, ${}^3\mathrm{D}_4(p^n)$ or $\PSU_5(p^n)$.

\begin{mtheorem}\hypertarget{thm2}{}
Suppose that $\fs$ is a saturated fusion system on a $p$-group $S$, where $S$ is isomorphic to a Sylow $p$-subgroup of ${}^2\mathrm{F}_4(2^n)$, ${}^3\mathrm{D}_4(p^n)$ or $\PSU_5(p^n)$ and $p=2$ in the first case. Then $\fs$ is realized by a known finite group. Moreover, if $O_p(\fs)=\{1\}$ then $\fs$ is isomorphic to the $p$-fusion category of an almost simple group $G$ with $F^*(G)\cong {}^2\mathrm{F}_4(2^n)$, ${}^3\mathrm{D}_4(p^n)$ or $\PSU_5(p^n)$ respectively.
\end{mtheorem}

This is proved as \cref{F42}, \cref{F42n}, \cref{3D4} and \cref{PSU5}.

Given a fixed $p$-group, or a fixed family of $p$-groups in which one wants to classify all saturated fusion systems supported, the approach generally breaks down into two steps: determine all the potential essential subgroups of $\fs$ as well as their automizers, and then show that $\fs$ is a known fusion system using certain uniqueness results. This is the approach adopted in this paper. For our work, the uniqueness results mainly stem from the uniqueness of weak BN-pairs of rank $2$, as shown in \cite{Greenbook}, and recognizing certain groups with strongly $p$-embedded subgroups and their associated $\GF(p)$-modules. Consequently, the majority of the effort in proving \hyperlink{thm2}{Theorem A} is in determining a complete list of potential essential subgroups. None of the techniques employed for this are especially new, and for the most part are adaptations of arguments from standard references. We do provide several structural results about Sylow $p$-subgroups of rank $2$ simple groups of Lie type which may find use elsewhere.

We use this opportunity to indicate exactly where, in this paper and across all the works needed to prove the \hyperlink{MainTheorem}{Main Theorem}, there is any reliance on the classification of the finite simple groups. As a starting point, at the time of writing there is no known generic way to distinguish an exotic fusion system without invoking the classification, and so whenever a system is shown to be exotic, the classification is used. For the determination of automizers in fusion systems on a Sylow $p$-subgroup of $\PSU_4(p^n)$ in \cite{G2pPaper}, the classification is used to identify the group $\PSL_2(p^{2n})$, acting on an essential subgroup $E$ as a \emph{natural $\Omega_4^-(p^n)$-module}. In determining the possible fusion systems supported on a Sylow $p$-subgroup of $\Sp_4(p^n)$, whenever $J(S)$ is an essential subgroup of the fusion system, $\Aut_{\fs}(J(S))$ acts irreducibly on $J(S)$ which is elementary abelian of order $p^{3n}$. The classification is used here to determine $\Aut_{\fs}(J(S))$ from this information. In the current work, we rely on the classification to determine the automizer of a certain essential subgroup of a Sylow $p$-subgroup of ${}^3\mathrm{D}_4(p^n)$. The boils down to determining the group $\SL_2(p^{3n})$ (embedded in $\Sp_8(p^n)$) acting on a \emph{triality module}, for $p$ an odd prime.

\section{Preliminaries: Groups and Fusion Systems}\label{PrelimSec}

In this first section, we set the scene for the proof of the \hyperlink{MainTheorem}{Main Theorem}, assembling conventions and providing standard results regarding fusion systems. For this reason, there is a large overlap with various other works e.g. \cite{ako} and \cite{G2pPaper}.

\begin{definition}
Let $S$ be a finite $p$-group. A fusion system $\fs$ on $S$ is a category with object set $\Ob(\fs):=\{Q: Q\le S\}$ and whose morphism set satisfies the following properties for $P, Q\le S$:
\begin{itemize}
\item $\Hom_S(P, Q)\subseteq \Mor_{\fs}(P,Q)\subseteq \Inj(P,Q)$; and
\item each $\phi\in\Mor_{\fs}(P,Q)$ is the composite of an $\fs$-isomorphism followed by an inclusion,
\end{itemize}
where $\Inj(P,Q)$ denotes injective homomorphisms between $P$ and $Q$. To motivate the group analogy, we write $\Hom_{\fs}(P,Q):=\Mor_{\fs}(P,Q)$ and $\Aut_{\fs}(P):=\Hom_{\fs}(P,P)$.
\end{definition}

\begin{example}\label{realize}
Let $G$ be a finite group with $S\in\syl_p(G)$. The \emph{$p$-fusion category} of $G$ on $S$, written $\fs_S(G)$, is the category with object set $\Ob(\fs_S(G)):= \{Q: Q\le S\}$ and for $P,Q\le S$, $\Mor_{\fs_S(G)}(P,Q):=\Hom_G(P,Q)$, where $\Hom_G(P,Q)$ denotes maps induced by conjugation by elements of $G$. That is, all morphisms in the category are induced by conjugation by elements of $G$.
\end{example}

We will often appeal to \cite{ako} for standard results and terminology regarding fusion systems. However, we use this section to emphasize certain definitions and properties which are pivotal in the communication of this work. Henceforth, we generically let $\fs$ be a fusion system on a finite $p$-group $S$. 

Say that $\mathcal{H}$ is a \emph{subsystem} of $\fs$, written $\mathcal{H}\le \fs$, on a $p$-group $T$ if $T\le S$, $\mathcal{H}\subseteq \fs$ as categories and $\mathcal{H}$ is itself a fusion system. Two subgroups of $S$ are said to be \emph{$\fs$-conjugate} if they are isomorphic as objects in $\fs$. We write $Q^{\fs}$ for the set of all $\fs$-conjugates of $Q$. 

We say a fusion system is \emph{realizable} if there exists a finite group $G$ with $S\in\syl_p(G)$ and $\fs=\fs_S(G)$, as in \cref{realize}. Otherwise, the fusion system is said to be \emph{exotic}.

Now for some conventions more concentrated on the objects of $\fs$: the subgroups of $S$.

\begin{definition}
Let $\fs$ be a fusion system on a $p$-group $S$ and let $Q\le S$. Say that $Q$ is
\begin{itemize}
\item \emph{fully $\fs$-normalized} if $|N_S(Q)|\ge |N_S(P)|$ for all $P\in Q^{\fs}$;
\item \emph{fully $\fs$-centralized} if $|C_S(Q)|\ge |C_S(P)|$ for all $P\in Q^{\fs}$;
\item \emph{fully $\fs$-automized} if $\Aut_S(Q)\in\syl_p(\Aut_{\fs}(Q))$;
\item \emph{receptive} in $\fs$ if for each $P\le S$ and each $\phi\in\Iso_{\fs}(P,Q)$, setting \[N_{\phi}=\{g\in N_S(P) : {}^{\phi}c_g\in\Aut_S(Q)\},\] there is $\bar{\phi}\in\Hom_{\fs}(N_{\phi}, S)$ such that $\bar{\phi}|_P = \phi$;
\item \emph{$S$-centric} if $C_S(Q)=Z(Q)$ and \emph{$\fs$-centric} if $P$ is $S$-centric for all $P\in Q^{\fs}$;
\item \emph{$S$-radical} if $O_p(\Out(Q))\cap \Out_S(Q)=\{1\}$; or
\item \emph{$\fs$-radical} if $O_p(\Out_{\fs}(Q))=\{1\}$.
\end{itemize}
\end{definition}

\begin{definition}
Let $\fs$ be a fusion system on a $p$-group $S$. Then $\fs$ is \emph{saturated} if the following conditions hold:
\begin{enumerate}
\item Every fully $\fs$-normalized subgroup is also fully $\fs$-centralized and fully $\fs$-automized.
\item Every fully $\fs$-centralized subgroup is receptive in $\fs$.
\end{enumerate}
By a theorem of Puig \cite{Puig1}, the fusion category of a finite group $\fs_S(G)$ is a saturated fusion system. Throughout this work, we will focus only on saturated fusion systems, although several of the results we provide should be applicable more generally.
\end{definition}

\begin{definition}
Let $\fs$ be a fusion system on a $p$-group $S$ and $Q\le S$. Say that $Q$ is \emph{normal} in $\fs$ if $Q\normaleq S$ and for all $P,R\le S$ and $\phi\in\Hom_{\fs}(P,R)$, $\phi$ extends to a morphism $\bar{\phi}\in\Hom_{\fs}(PQ,RQ)$ such that $\bar{\phi}(Q)=Q$.

It may be checked that the product of normal subgroups is itself normal. Thus, we may talk about the largest normal subgroup of $\fs$ which we denote $O_p(\fs)$ (and occasionally refer to as the $p$-core of $\fs$).  Further, it follows immediately from the saturation axioms that any subgroup normal in $S$ is fully normalized and fully centralized.
\end{definition}

\begin{definition}
Let $\fs$ be a fusion system on a $p$-group $S$ and let $Q$ be a subgroup. The \emph{normalizer fusion subsystem} of $Q$, denoted $N_{\fs}(Q)$, is the largest subsystem of $\fs$, supported on $N_S(Q)$, in which $Q$ is normal. 
\end{definition}

\begin{theorem}
Let $\fs$ be a saturated fusion system on a $p$-group $S$. If $Q\le S$ is fully $\fs$-normalized then $N_{\fs}(Q)$ is saturated.
\end{theorem}
\begin{proof}
See \cite[Theorem I.5.5]{ako}.
\end{proof}

\begin{theorem}[Model Theorem]\label{model}
Let $\fs$ be a saturated fusion system on a $p$-group $S$. Fix $Q\le S$ which is $\fs$-centric and normal in $\fs$. Then the following hold:
\begin{enumerate} 
\item There are models for $\fs$.
\item If $G_1$ and $G_2$ are two models for $\fs$, then there is an isomorphism $\phi: G_1\to G_2$ such that $\phi|_S = \mathrm{Id}_S$.
\item For any finite group $G$ containing $S$ as a Sylow $p$-subgroup such that $Q\normaleq G$, $C_G(Q) \le Q$ and $Aut_G(Q) = Aut_{\fs}(Q)$, there is $\beta\in\Aut(S)$ such that $\beta|_Q = \mathrm{Id}_Q$ and $\fs_S(G) = {}^{\beta}\fs$. Thus, there is a model for $\fs$ which is isomorphic to $G$.
\end{enumerate}
\end{theorem}
\begin{proof}
See \cite[Theorem I.4.9]{ako}.
\end{proof}

Introductions out of the way, we now come to the strategy to prove the main theorem. As remarked in the introduction, the boils down to two steps: determining the \emph{essential subgroups} of $\fs$ and the morphism attached to them, and then determining $\fs$ from certain uniqueness results.

\begin{definition}
$E$ is $\fs$-essential if $E$ is an $\fs$-centric, fully $\fs$-normalized subgroup of $S$ such that $\Out_{\fs}(E)$ contains a strongly $p$-embedded subgroup. We write $\mathcal{E}(\fs)$ for the set of essential subgroups of a fusion system $\fs$.
\end{definition}

Directly from the definition, it is clear that $N_{\fs}(E)$ is a saturated fusion system with $O_p(N_{\fs}(E))=E$ an $N_{\fs}(E)$-centric subgroup. Hence, by \cref{model} there is a finite group $G$ with $E=O_p(G)$, $C_G(E)\le E$ and $N_{\fs}(E)=\fs_S(G)$.

The following theorem justifies the first step of our strategy of proof of \hyperlink{thm2}{Theorem A}.

\begin{theorem}[Alperin -- Goldschmidt Fusion Theorem]
Let $\fs$ be a saturated fusion system on a $p$-group $S$. Then \[\fs=\langle \Aut_{\fs}(Q) \mid Q\,\, \text{is essential or}\,\, Q=S \rangle.\]
\end{theorem}
\begin{proof}
See \cite[Theorem I.3.5]{ako}.
\end{proof}

In this work, we operate in a situation where $S$ is prescribed but $\fs$ is not necessarily known. Thus, we need techniques for detecting whether subgroups of $S$ can be essential subgroups of an arbitrary saturated fusion systems supported on $S$.

\begin{lemma}\label{burnside}
Let $S$ be a finite $p$-group. Then $C_{\Aut(S)}(S/\Phi(S))$ is a normal $p$-subgroup of $\Aut(S)$.
\end{lemma}
\begin{proof}
This is due to Burnside, see \cite[Theorem 5.1.4]{gor}.
\end{proof}

Note that for subgroup $P<Q\le N_S(P)$ with $Q$ centralizing $P/\Phi(P)$, $P$ is not $S$-radical so not $\fs$-radical and hence, never essential in any saturated fusion systems supported on $S$. Actually, this result will often not be enough and we will require a variation on this result more suited to our application. For this, we need to consider certain \emph{$\fs$-characteristic chains} of $E$.

\begin{definition}
Let $\fs$ be a fusion system on a $p$-group $S$ and $P\le Q\le S$. Say that $P$ is \emph{$\fs$-characteristic} in $Q$ if $\Aut_{\fs}(Q)\le N_{\Aut(Q)}(P)$.
\end{definition}

\begin{lemma}\label{Chain}
Let $\fs$ be a saturated fusion system on a $p$-group $S$, $E\in\mathcal{E}(\fs)$ and $A\le \Aut_{\fs}(E)$. Set $\Phi(E)=E_0 \normaleq E_1  \normaleq E_2 \normaleq \dots \normaleq E_m=E$ such that, for all $0 \le i \le m$, $E_i\alpha = E_i$ for each $\alpha \in A$.  If $Q\le S$ is such that $[Q, E_i]\le E_{i-1}$ for all $1 \le i \le m$ then $Q\le E$.
\end{lemma}
\begin{proof}
Let $\fs$, $S$, $E$ and $Q$ be as described in the lemma. Then $N_{QE}(E)$ embeds in $\Aut(E)$ as $\Aut_{QE}(E)$ and we can arrange that $N_{QE}(E)$, and hence $\Aut_{QE}(E)$, centralizes (an $A$-invariant refinement of) $E_0 \normaleq E_1  \normaleq E_2 \normaleq \dots \normaleq E_m$. Setting $A:=\langle \Aut_{QE}^{\Aut_{\fs}(E)}\rangle$, we have by \cite[{(I.5.3.3)}]{gor} that $A$ is a $p$-group. Since $A\normaleq \Aut_{\fs}(E)$, $A\le O_p(\Aut_{\fs}(E))=\Inn(E)$, since $E$ is $\fs$-radical. Hence, $\Aut_{QE}(E)\le \Inn(E)$ and since $E$ is $\fs$-centric, we infer that $N_{QE}(E)\le E$ and $N_{QE}(E)=E$. Then $E\normaleq N_{QE}(N_{QE}(E))$ and the only possibility is that $E\normaleq QE$ so that $E=QE$ and $Q\le E$, as required.
\end{proof}

In order to fully apply \cref{Chain}, we need some characteristic subgroups to work with. In our application, these tend to be members of the lower or upper central series of a candidate essential subgroup $E$, and their centralizers in $E$. However, there are other fairly natural characteristic subgroups which will appear later in our analysis which we document below. We start with a subgroup inspired by one defined in \cite{GlaSol}.

\begin{definition}
Let $S$ be a finite $p$-group and let $\mathbb{W}(S)$ be the set of elementary abelian normal subgroups $A$ of $S$ satisfying the following following property: whenever $[A, s, s]=\{1\}$ we have that $[A, s]=\{1\}$ for any $s\in S$.
Set $W(S)=\langle \mathbb{W}(S)\rangle$.
\end{definition}

\begin{lemma}\label{WS}
Let $S$ be a finite $p$-group. Then
\begin{enumerate}
    \item $W(S)$ is a characteristic subgroup of $S$;
    \item $W(S)$ is the unique largest member of the set $\mathbb{W}(S)$ with respect to inclusion;
    \item $\Omega(Z(S))\le W(S)$; and
    \item if $W(S)\le T\le S$ then $W(S)\le W(T)$.
\end{enumerate}
\end{lemma}
\begin{proof}
Let $A\in \mathbb{W}(S)$ and let $\phi$ an automorphism of $S$. Then $A\phi$ is an elementary abelian normal subgroup of $S$. For any $s\in S$ with $[A\phi, s, s]=\{1\}$, we have that $[A\phi, s, s]=[A, s\phi^{-1}, s\phi^{-1}]\phi=\{1\}$. Since $A\in\mathbb{W}(S)$, we deduce that $[A, s\phi^{-1}]=\{1\}$ from which it follows that $[A\phi, s]=\{1\}$ and so $A\phi\in\mathbb{W}(S)$ and $W(S)$ is characteristic in $S$, proving (i).

For $A, B\in \mathbb{W}(S)$, we have that $[A, B, B]\le [A\cap B, B]=\{1\}$ so that $[A, B]=\{1\}$. Hence, every member of $\mathbb{W}(S)$ is abelian and commutes with every other member of $\mathbb{W}(S)$ and we infer that $W(S)$ is abelian. Then $W(S)$ is generated by elements of order $p$, and so $W(S)$ is an elementary abelian normal subgroup of $S$. Now, for $s\in S$ with $[W(S), s, s]=\{1\}$ we have that $[A, s, s]=[A, s]=\{1\}$ for all $A\in\mathbb{W}(S)$. Since $W(S)=\langle \mathbb{W}(S)\rangle$, we must have that $[W(S), s]=\{1\}$ and $W(S)\in\mathbb{W}(S)$. Hence, (ii) holds.

Since $\Omega(Z(S))$ is an elementary abelian normal subgroup of $S$, with $[\Omega(Z(S)),s]=\{1\}$ for all $s\in S$, $\Omega(Z(S))\in\mathbb{W}(S)$, and (iii) holds. Let $T\le S$ with $W(S)\le T\le S$. Then $W(S)$ is an elementary abelian normal subgroup of $T$ and for any $t\in T$, $t\in S$ and so $[W(S), t, t]=\{1\}$ implies that $[W(S), t]=\{1\}$ and $W(S)\in\mathbb{W}(T)$. Since $\mathbb{W}(T)$ has a unique maximal element, $W(S)\le W(T)$ and (iv) holds.
\end{proof}

Throughout, we use the ``elementary" version of the Thompson subgroup.

\begin{definition}
Let $S$ be a finite $p$-group. Set $\mathcal{A}(S)$ to be the set of all elementary abelian subgroups of $S$ of maximal rank. Then the \emph{Thompson subgroup} of $S$ is defined as $J(S):=\langle A \mid A\in\mathcal{A}(S)\rangle$. 
\end{definition}

Related to the Thompson subgroup, for a $p$-group $S$ we define $\mathcal{A}_{\normaleq}(S):=\{A\in\mathcal{A}(S) \mid A\normaleq S\}$ and $J_{\normaleq}(S):=\langle A \mid A\in\mathcal{A}_{\normaleq}(S)\rangle$. Of course, there is the question of whether the set $\mathcal{A}_{\normaleq}(S)$ is non-empty, and whether the group $J_{\normaleq}(S)$ is well defined, the answer to which is generally no. However, in our application, whenever we exploit arguments involving $J_{\normaleq}(S)$, we will have already demonstrated that $\mathcal{A}_{\normaleq}(S)$ is non-empty. It is clear that $J_{\normaleq}(S)$ is a characteristic subgroup of $S$.

The following lemma is specific only to the case $p=2$ and relies on a nice property of involutions. In this work, it is used only in the analysis of a Sylow $2$-subgroup of ${}^2\mathrm{F}_4(2^n)$.

\begin{lemma}\label{elementary2}
Let $S$ be a $2$-group and suppose that there is $A,B\in\mathcal{A}(S)$ with $AB=S$. If $C_B(a)=A\cap B$ for every $a\in A\setminus (A\cap B)$, then $\mathcal{A}(S)=\{A, B\}$ and every involution of $S$ lies in $A\cup B$.
\end{lemma}
\begin{proof}
Aiming for a contradiction, let $C\in\mathcal{A}(S)$ with $A\ne C\ne B$ and choose $c$ an involution in $S$. Then $c=ab$ for some $a\in A$ and $b\in B$ and $1=c^2=(ab)^2=[a, b]$. Hence, either $a\in A\cap B$ so that $c\in B$, or $a\not\in B$ and $b\in C_B(a)=A\cap B$ so that $c=ab\in A$. Thus, $c\in A\cup B$ and we now have that $C=(C\cap A)(C\cap B)$. 

Let $c\in C\cap A$ with $c\not\in B$. Then $c\in A\setminus (A\cap B)$ so that $C_B(c)=A\cap B$. But $C\cap B\le C_B(c)$ so that $C\cap B\le A$ and $C\le A$. Since $A,C\in\mathcal{A}(S)$, $C=A$, a contradiction. Hence, no such $C$ exists and $\mathcal{A}(S)=\{A, B\}$.
\end{proof}

The Thompson subgroup is intimately related with the notion of \emph{failure to factorize} modules for finite groups.

\begin{definition}
Let $G$ be a finite group and $V$ a $\mathrm{GF}(p)$-module. If there exists $\{1\}\ne A\le G$ such that
\begin{enumerate}
\item $A/C_A(V)$ is an elementary abelian $p$-group;
\item $[V,A]\ne\{1\}$; and 
\item $|V/C_V(A)|\leq |A/C_A(V)|$
\end{enumerate}
then $V$ is a failure to factorize module (abbrev. FF-module) for $G$ and $A$ is an \emph{offender} on $V$. 
\end{definition}

Now follows a selection of results used to identify automizers and the structure of $\Aut_{\fs}(E)$-chief factors of $E$, viewed as $\GF(p)\Out_{\fs}(E)$-modules. We begin with some modules for $\SL_2(p^n)$ and some ways in which we may recognize them from the actions of $N_S(E)$ on $E$. These results will often be used in the remainder of this work without explicit references.

\begin{definition}
\sloppy{A \emph{natural $\SL_2(p^n)$-module} is any irreducible $2$-dimensional $\GF(p^n)\SL_2(p^n)$-module regarded as a $2n$-dimension module for $\GF(p)\SL_2(p^n)$ by restriction.}
\end{definition}

\begin{lemma}\label{sl2p-mod}
Suppose $G\cong\SL_2(p^n)$, $S\in\syl_p(G)$ and $V$ is natural $\SL_2(p^n)$-module. Then the following holds:
\begin{enumerate}
\item $[V, S,S]=\{1\}$;
\item $|V|=p^{2n}$ and $|C_V(S)|=p^n$, and so $V$ is an FF-module;
\item $C_V(s)=C_V(S)=[V,S]=[V,s]=[v, S]$ for all $v\in V\setminus C_V(S)$ and $1\ne s\in S$; and
\item $V/C_V(S)$ and $C_V(S)$ are irreducible $\mathrm{GF}(p)N_G(S)$-modules upon restriction.
\end{enumerate}
\end{lemma}
\begin{proof}
See \cite[Lemma 4.6]{parkerBN} and \cite[Lemma 3.13]{parkerSymp}.
\end{proof}

\begin{theorem}\label{SEFF}
Suppose that $E$ is an essential subgroup of a saturated fusion system $\fs$ on a $p$-group $S$, and assume that there is an $\Aut_{\fs}(E)$-invariant subgroups $U\le V\le E$ such that $E=C_S(V/U)$ and $V/U$ is an FF-module for $G:=\Out_{\fs}(E)$. Then, writing $L:=O^{p'}(G)$ and $W:=V/U$, we have that $L/C_L(W)\cong \SL_2(p^n)$, $C_L(W)$ is a $p'$-group and $W/C_W(O^p(L))$ is a natural $\SL_2(p^n)$-module.
\end{theorem}
\begin{proof}
This follows from \cite[Theorem 5.6]{henkesl2} (c.f. \cite[Theorem 1]{henkesl2}).
\end{proof}

\begin{definition}
Let $V$ be a natural $\SL_2(p^{2n})$-module for $G\cong \SL_2(p^{2n})$. A \emph{natural $\Omega_4^-(p^n)$-module} for $G$ is any non-trivial irreducible submodule of $(V\otimes_k V^\tau)_{\GF(p^n)G}$ regarded as a $\GF(p)G$-module by restriction, where $\tau$ is an involutary automorphism of $\GF(p^{2n})$.
\end{definition}

\begin{lemma}\label{Omega4}
Let $G\cong \SL_2(p^{2n})$, $S\in\syl_p(G)$ and $V$ a natural $\Omega_4^-(p^n)$-module for $G$. Then the following holds:
\begin{enumerate}
\item $C_G(V)=Z(G)$;
\item $[V, S, S, S]=\{1\}$;
\item $|V|=p^{4n}$ and $|V/[V,S]|=|C_V(S)|=p^n$; and
\item $V/[V,S]$ and $C_V(S)$ are irreducible $\mathrm{GF}(p)N_G(S)$-modules upon restriction.
\end{enumerate}
Moreover, for $\{1\}\ne F\le S$, one of the following occurs:
\begin{enumerate}[label=(\alph*)]
\item $[V, F]=[V, S]$ and $C_{V}(F)=C_{V}(S)$;
\item $p=2$, $[V, F]=C_{V}(F)$ has order $2^{2n}$, $F$ is quadratic on $V$ and $|F|\leq 2^n$; or
\item $p$ is odd, $|[V, F]|=|C_{V}(F)|=p^{2n}$, $[V, S]=[V, F]C_{V}(F)$, $C_V(S)=C_{[V, F]}(F)$ and $|F|\leq p^n$.
\end{enumerate}
\end{lemma}
\begin{proof}
See \cite[Lemma 4.8]{parkerBN} and \cite[Lemmas 3.12, 3.15]{parkerSymp}.
\end{proof}

\begin{definition}
Let $V$ be a natural $\SL_2(p^{3n})$-module for $G\cong \SL_2(p^{3n})$. A \emph{triality module} for $G$ is any non-trivial irreducible submodule of $(V\otimes V^{\tau}\otimes V^{\tau^2})|_{\mathrm{GF}(p^n)G}$ regarded as a $\mathrm{GF}(p)G$-module by restriction, where $\tau$ is an automorphism of $\GF(p^{3n})$ of order $3$.
\end{definition}

\begin{lemma}\label{trialitydescription}
Suppose that $G\cong\SL_2(p^{3n})$, $S\in \syl_p(G)$ and $V$ is a triality module for $G$. Then the following hold:
\begin{enumerate}
\item $[V, S, S, S, S]=\{1\}$;
\item $|V|=p^{8n}$, $|V/[V,S]|=|C_V(S)|=|[V, S, S, S]|=p^n$ and $|[V,S,S]|=p^{4n}$;
\item if $p$ is odd then $|V/C_V(s)|=p^{5n}$, while if $p=2$ then $|V/C_V(s)|=p^{4n}$, for all $1\ne s\in S$; and 
\item $V/[V,S]$ and $C_V(S)$ are irreducible $\mathrm{GF}(p)N_G(S)$-modules upon restriction.
\end{enumerate}
\end{lemma}
\begin{proof}
See \cite[Lemma 4.10]{parkerBN} and \cite[Lemmas 3.12, 3.14, 3.16]{parkerSymp}.
\end{proof}

Aside from $\SL_2(p^n)$, the only other automizers of essential subgroups which will appear later in this work are the other rank $1$ groups of Lie type in characteristic $p$. We will never have the deal with the Ree groups ${}^2\mathrm{G}_2(3^n)$, and so we only describe some modules associated to $\SU_3(p^n)$ and $\Sz(2^n)$.

\begin{definition}
A \emph{natural $\SU_3(p^n)$-module} is the restriction of the natural module for $\SL_3(p^{2n})$ regarded as a $\mathrm{GF}(p)\SU_3(p^n)$-module by restriction.
\end{definition}

\begin{lemma}\label{SUMod}
Suppose $G\cong\SU_3(p^n)$, $S\in\syl_p(G)$ and $V$ is a natural module.  Then the following hold:
\begin{enumerate}
\item $C_V(S)=[V, Z(S)]=[V, S,S]$ is of order $p^{2n}$;
\item $C_V(Z(S))=[V,S]$ is of order $p^{4n}$; and
\item $V/[V, S]$, $[V, S]/C_V(S)$ and $C_V(S)$ are irreducible $\mathrm{GF}(p)N_G(S)$-modules upon restriction.
\end{enumerate}
\end{lemma}
\begin{proof}
This may be calculated in $\SL_3(p^{2n})$. 
\end{proof}

\begin{definition}
A \emph{natural $\Sz(2^n)$-module} is the restriction of the natural module for $\Sp_4(2^n)$ regarded as a $\mathrm{GF}(2)\Sz(2^n)$-module by restriction.
\end{definition}

\begin{lemma}\label{SzMod}
Suppose $G\cong\Sz(2^n)$, $S\in\syl_2(G)$ and $V$ is the natural module. Then the following hold:
\begin{enumerate}
\item $[V, S]$ has order $2^{3n}$;
\item $[V,\Omega(S)]=C_V(\Omega(S))=[V,S,S]$ has order $2^{2n}$;
\item $C_V(S)=[V,S,\Omega(S)]=[V, \Omega(S), S]=[V,S,S,S]$ has order $2^n$; and
\item $V/[V,S]$, $[V,S]/C_V(\Omega(S))$, $C_V(\Omega(S))/C_V(S)$ and $C_V(S)$ are all irreducible $\mathrm{GF}(p)N_G(S)$-modules upon restriction.
\end{enumerate}
\end{lemma}
\begin{proof}
This may be calculated in $\Sp_4(2^n)$. 
\end{proof}

With the relevant groups and modules described, we now provide results which detect them. Although the previously quoted \cref{SEFF} is the most prevalent result throughout this work, some others will be required in various places.

\begin{lemma}\label{SL2ModRecog}
Let $G$ be a $p'$-central extension of $\PSL_2(p^n)$, $S\in\syl_p(G)$ and $V$ a faithful irreducible $\GF(p)$-module. If $|V|<p^{3n}$ then either
\begin{enumerate}
\item $V$ is a natural $\SL_2(p^n)$-module for $G\cong \SL_2(p^n)$;
\item $V$ is a natural $\Omega_4^-(p^{n/2})$-module, $n$ is even, $S$ does not act quadratically on $V$ and $Z(G)$ acts trivially on $V$; or
\item $V$ is a triality module, $n$ is a multiple of $3$ and $S$ does not act quadratically on $V$.
\end{enumerate}
\end{lemma}
\begin{proof}
See \cite[Lemma 2.6]{ChermakJ}.
\end{proof}

Often, upon examining some candidate subgroup $E$ of a saturated fusion system $\fs$ with the aim of showing that it is not essential, we can assume that $E$ is not contained in any other essential subgroup of $\fs$. We term essentials subgroups which are contained in no other essential subgroups \emph{maximally essential}.

\begin{definition}
Suppose that $\fs$ is a saturated fusion system on a $p$-group $S$. Then $E\le S$ is \emph{maximally essential} in $\fs$ if $E$ is essential and, if $F\le S$ is essential in $\fs$ and $E\le F$, then $E=F$.
\end{definition}

\begin{lemma}\label{MaxEssen}
Suppose that $E$ is a maximally essential subgroup of a saturated fusion system $\fs$ which is supported on a $p$-group $S$. If $m_p(\Out_S(E))\geq 2$ then $O_{p'}(O^{p'}(\Out_{\fs}(E)))\le Z(O^{p'}(\Out_{\fs}(E)))$.
\end{lemma}
\begin{proof}
Let $T\le N_S(E)$ with $E<T$. Now, since $E$ is receptive, for all $\alpha\in N_{\Aut_{\fs}(E)}(\Aut_T(E))$, $\alpha$ lifts to a morphism $\hat{\alpha}\in\Hom_{\fs}(N_\alpha, S)$ with $N_\alpha\ge T>E$. Since $E$ is maximally essential, applying the Alperin--Goldschmidt theorem, $\hat{\alpha}$ is the restriction to $N_\alpha$ of a morphism $\bar{\alpha}\in\Aut_{\fs}(S)$. But then, upon restriction, $\alpha$ normalizes $\Aut_S(E)$ and so $N_{\Aut_{\fs}(E)}(\Aut_T(E))\le N_{\Aut_{\fs}(E)}(\Aut_S(E))$. This induces the inclusion $N_{\Out_{\fs}(E)}(\Out_T(E))\le N_{\Out_{\fs}(E)}(\Out_S(E))$. Since this holds for all $T\le N_S(E)$ with $E<T$, we infer that $N_{\Out_{\fs}(E)}(\Out_S(E))$ is strongly $p$-embedded in $\Out_{\fs}(E)$.

Since $m_p(\Out_S(E))\geq 2$, by coprime action, \[O_{p'}(O^{p'}(\Out_{\fs}(E)))=\langle C_{O_{p'}(O^{p'}(\Out_{\fs}(E)))}(a) \mid a\in \Out_S(E)^\# \rangle\] so that $O_{p'}(O^{p'}(\Out_{\fs}(E)))$ is contained in the strongly $p$-embedded subgroup of $\Out_{\fs}(E)$. Hence, $[\Out_S(E), O_{p'}(O^{p'}(\Out_{\fs}(E)))]=\{1\}$ and since $O^{p'}(\Out_{\fs}(E))$ is the normal closure in $\Out_{\fs}(E)$ of $\Out_S(E)$, we deduce that $O_{p'}(O^{p'}(\Out_{\fs}(E)))$ lies in the center of $O^{p'}(\Out_{\fs}(E))$, as required.
\end{proof}

The following results are independent of the classification of the finite simple groups and provide generic ways in which to identify groups with a strongly $p$-embedded subgroup.

\begin{proposition}\label{MaxEssenEven}
Suppose that $\fs$ is a saturated fusion system on a $2$-group $S$ and $E\le S$ is maximally essential. If $m_2(\Out_S(E))\geq 2$ then $O^{2'}(\Out_{\fs}(E))\cong \PSL_2(2^n)$, $\mathrm{(P)SU}_3(2^n)$ or $\Sz(2^n)$ for some $n>1$.
\end{proposition}
\begin{proof}
We apply a result of Bender \cite{Bender} so that, as $m_2(\Out_S(E))>1$, we have $O^{2'}(\Out_{\fs}(E))/O_{2'}(O^{2'}(\Out_{\fs}(E)))\cong \PSL_2(2^n)$, $\PSU_3(2^n)$ or $\Sz(2^n)$ for some $n>1$. Since $O_{2'}(O^{2'}(\Out_{\fs}(E)))\le Z(O^{2'}(\Out_{\fs}(E)))$ by \cref{MaxEssen}, using information on Schur multipliers as can be found in \cite{GLS3}, we have that $O^{2'}(\Out_{\fs}(E))\cong \PSL_2(2^n)$, $\mathrm{(P)SU}_3(2^n)$ or $\Sz(2^n)$, as required.
\end{proof}

\begin{proposition}\label{MaxEssenOdd}
Suppose that $\fs$ is a saturated fusion system on a $p$-group $S$, $p$ is odd, $E\le S$ is maximally essential and $\Out_S(E)$ is a TI-set for $\Out_{\fs}(E)$. If $m_p(\Out_S(E))\geq 2$ and there is $x\in \Out_S(E)$ with $[E, x, x]\le \Phi(E)$ then $O^{p'}(\Out_{\fs}(E))\cong \SL_2(p^{n+1})$ or $\mathrm{(P)SU}_3(p^n)$ for $n\geq 1$.
\end{proposition}
\begin{proof}
\sloppy{We apply the main result of \cite{HoTI} from which we deduce that $O^{p'}(\Out_{\fs}(E))/O_{p'}(O^{p'}(\Out_{\fs}(E)))\cong \PSL_2(p^{n+1})$ or $\PSU_3(p^n)$ for $n\in\N$. As in \cref{MaxEssenEven}, \cref{MaxEssen} yields that $O_{p'}(O^{p'}(\Out_{\fs}(E)))\le Z(O^{p'}(\Out_{\fs}(E)))$ and using information about the Schur multipliers of $\PSL_2(p^{n+1})$ and $\PSU_3(p^n)$ as can be found in \cite{GLS3}, we deduce that $O^{p'}(\Out_{\fs}(E))\cong \mathrm{(P)SL}_2(p^{n+1})$ or $\mathrm{(P)SU}_3(p^n)$ for $n\geq 1$. Since $\PSL_2(p^{n+1})$ has abelian or dihedral Sylow $2$-subgroups, \cite[(I.3.8.4)]{gor} provides a contradiction in this case. Hence, the result holds.}
\end{proof}

A feature of the parabolic subgroups of several small rank groups of Lie type in characteristic $p$ is the appearance of \emph{semi-extraspecial groups}.

\begin{definition}
A $p$-group $Q$ is \emph{semi-extraspecial} if $\Phi(Q)=[Q, Q]=Z(Q)$ and $Q/Z$ is extraspecial for $Z$ any maximal subgroup of $Z(Q)$. $Q$ is ultraspecial if $Q$ is semi-extraspecial and $|Z(Q)|^2=|Q/Z(Q)|$.
\end{definition}

\begin{lemma}\label{Ultraspecial}
Suppose that $Q$ is a semi-extraspecial $p$-group. Then the following hold:
\begin{enumerate}
    \item $|Q/Z(Q)|=p^{2n}$ for some $n\in \N$.
    \item For $Z$ a maximal subgroup of $Z(Q)$, and $n$ such that $|Q/Z(Q)|=p^{2n}$ if $A/Z$ is an abelian subgroup of $Q/Z$, then $|A/Z|\leq p^{1+n}$. Consequently, if $A\le Q$ is abelian then $|A|\leq |Z(Q)|p^n$. 
    \item If $A\le Q$ with $Z(Q)\le A$ and $|A|=|Z(Q)|p^n$ then $[x, A]=Z(S)$ for any $x\in Q\setminus A$.
\end{enumerate}
\begin{proof}
See \cite[Proposition 2.66]{parkerSymp}.
\end{proof}
\end{lemma}

This slew of results will be enough to determine the local actions in  saturated fusion system $\fs$ in our setup, so we now turn our attention to determining $\fs$ from this local information. This is the second step in the strategy to determine $\fs$, and concerns itself with the uniqueness arguments for $\fs$. Fortunately, the hard work for these arguments has been completed elsewhere (\cite[Main Theorem]{MainThm} using \cite{Greenbook}).

\begin{theorem}\label{MainThm}
Let $\fs$ be a saturated fusion system with exactly two essential subgroups $E_1, E_2\normaleq S$ such that $\Aut_{\fs}(E_i)$ is a $\mathcal{K}$-group for $i\in\{1,2\}$. If $O_p(\fs)=\{1\}$ then $\fs$ is known explicitly.
\end{theorem}

Recall that a $\mathcal{K}$-group is a finite group in which every simple section is assumed to be a known finite simple group. For the most part, this $\mathcal{K}$-group hypothesis is used to reduce the list of groups with a strongly $p$-embedded subgroups in the search for automizers of essential subgroups of a given fusion system. For this work, \cref{SEFF}, \cref{MaxEssenEven} and \cref{MaxEssenOdd} will generally be enough to determine the form of the automizer and then \cite{Greenbook} determines the induced \emph{amalgam} uniquely up to isomorphism. By \cite{MainThm}, this is enough to determine the fusion system up to isomorphism. 

The only place where we require this $\mathcal{K}$-group hypothesis is in the determination of the automizer of a particular essential subgroup of the $p$-fusion category of ${}^3\mathrm{D}_4(p^n)$ for $p$ an odd prime. Indeed, here we need to determine subgroups of $\Sp_8(p^n)$ acting on an $8$-dimensional $\GF(p^n)$-module $V$, with a Sylow $p$-subgroup elementary abelian of order $p^{3n}$ and acting on $V$ as a Sylow $p$-subgroup of $\SL_2(p^{3n})$ acts on a triality module, and with the normalizer of this Sylow $p$-subgroup strongly $p$-embedded. As far as the author is aware, there is no result in the literature which deduces $\SL_2(p^{3n})$ from this hypothesis without using the classification.

Rather than provide a full list of outcomes in \cref{MainThm}, we instead only list the relevant fusion systems for this work in the following three corollaries.

\begin{corollary}\label{2F4Cor}
Suppose the hypothesis of \cref{MainThm} and assume that $S$ is isomorphic to a Sylow $2$-subgroup of ${}^2\mathrm{F}_4(2^n)$ for some $n\in\N$. Then $\fs=\fs_S(G)$, where $F^*(G)=O^{2'}(G)\cong{}^2\mathrm{F}_4(2^n)$.
\end{corollary}

\begin{corollary}\label{3D4Cor}
Suppose the hypothesis of \cref{MainThm} and assume that $S$ is isomorphic to a Sylow $p$-subgroup of ${}^3\mathrm{D}_4(p^n)$ for some $n\in\N$. Then $\fs=\fs_S(G)$, where $F^*(G)=O^{p'}(G)\cong{}^3\mathrm{D}_4(p^n)$.
\end{corollary}

\begin{corollary}\label{PSUCor}
Suppose the hypothesis of \cref{MainThm} and assume that $S$ is isomorphic to a Sylow $p$-subgroup of $\PSU_5(p^n)$ for some $n\in\N$. Then $\fs=\fs_S(G)$, where $F^*(G)=O^{p'}(G)\cong\mathrm{PSU}_5(p^n)$.
\end{corollary}

\section[Fusion Systems on a Sylow \texorpdfstring{$2$}{2}-subgroup of \texorpdfstring{${}^2\mathrm{F}_4(2^n)$}{2F4(2n)}]{Fusion Systems on a Sylow $2$-subgroup of ${}^2\mathrm{F}_4(2^n)$}

We first deal with the case where $S$ is isomorphic to a Sylow $2$-subgroup of ${}^2\mathrm{F}_4(2)$ or ${}^2\mathrm{F}_4(2)'$. In the case where $S\in\syl_2({}^2\mathrm{F}_4(2))$, we set $Q_1:=C_S(Z_2(S))$ and $Q_2:=C_S(Z_3(\Omega(S))/Z(S))$. The following result is proved using the fusion systems package in MAGMA \cite{Comp1}.

\begin{proposition}\label{F42}
Let $G={}^2\mathrm{F}_4(2)$ and $S\in\syl_2(G)$. If $\fs$ is a saturated fusion system supported on $S$ then either:
\begin{enumerate}
\item $\fs=\fs_S(S)$;
\item $\fs=\fs_S(N_G(Q_1))$ where $\Out_{\fs}(Q_1)\cong \Sym(3)$;
\item $\fs=\fs_S(N_G(Q_2))$ where $\Out_{\fs}(Q_2)\cong \Sz(2)$; or
\item $\fs=\fs_S(G)$.
\end{enumerate}
Moreover, for each of the above fusion systems, writing $\fs=\fs_S(H)$, we have that $\fs_{T}(H\cap O^2(G))$ is a saturated fusion on $T:=S\cap O^2(G)$, and every saturated fusion system on $T$ arises in this fashion.
\end{proposition}

We remark that $\Out_{\fs_S(N_G(Q_i))}(Q_i)\cong \Out_{\fs_T(N_{O^2(G)}(Q_i\cap T))}(Q_i\cap T)$ and $Q_i\cap T$ has index $2$ in $Q_i$ for $i\in\{1,2\}$.

We continue now under the assumption that $S$ is isomorphic to a Sylow $2$-subgroup of ${}^2\mathrm{F}_4(2^n)$ where $n>1$, and write $q=2^n$. The reader is referred to \cite{parrott} for a description of $S$ in terms of Chevalley commutator formulas. Indeed, in the work which follows, we could have used this in our approach of the analysis of the internal structure of $S$, but we have elected instead to describe $S$ in terms of its embedding in the two maximal subgroups of ${}^2\mathrm{F}_4(q)$ which contain it.

For $G:={}^2\mathrm{F}_4(q)$, we identify $S\in\syl_2(G)$ and set $G_1:=N_G(C_S(Z_2(S)))$, $G_2:=N_G(C_S(Z_3(\Omega(S))/Z(S)))$ and $L_i:=O^{2'}(G_i)$ for $i\in\{1,2\}$ so that $L_1$ is of shape $q^{2+1+2+4+2}:\SL_2(q)$ and $L_2$ is of shape $q^{1+4+1+4}:\Sz(q)$. Note that $G_1$ and $G_2$ are the unique maximal parabolic subgroups of $G$ containing $N_G(S)$. We write $Q_i:=O_2(L_i)=O_2(G_i)$ so that $Q_1=C_S(Z_2(S))$ and $Q_2=C_S(Z_3(\Omega(S))/Z(S))$, characteristic subgroups of $S$. Then set $V_1:=\Omega(Q_1)$, $U_1:=[V_1, Q_1]$ and $V_2:=[Q_2, Q_2]$. 

We record the following structural properties of $L_1, L_2$ and $S$. That the subgroups of $S, Q_1$ and $Q_2$ are as claimed may be calculated using the commutator formulas in \cite{parrott}, while the chief factor structure of $L_1$ and $L_2$ may be extracted from \cite{Seitz}, \cite{2F4Par}, \cite[Chapter 12]{Greenbook} or argued using some of the techniques developed in \cref{PrelimSec}.

\begin{proposition}\label{F4Basic}
The following hold:
\begin{enumerate}
\item $|S|=q^{12}$;
\item $S'=\Phi(S)=V_1(Q_1\cap Q_2)$;
\item $m_p(S)=5n$;
\item $Z(S)=Z(Q_2)$ has order $q$;
\item $Z_2(S)=Z(Q_1)$ has order $q^2$;
\item $V_2=Z_3(\Omega(S))$ is elementary abelian of order $q^5$;
\item $\Omega(U_1C_{Q_2}(V_2))=U_1V_2$; and
\item $\Omega(S)=V_1Q_2$ and if $x\in S$ is such that $x^2=1$, then $x\in \Omega(Q_1) \cup Q_2$.
\end{enumerate}
\end{proposition}

\begin{proposition}
The following hold:
\begin{enumerate}
    \item $|Q_1|=q^{11}$;
    \item $|V_1|=q^9$ and $Q_1/V_1$ is a natural module for $L_1/Q_1\cong \SL_2(q)$;
    \item $|U_1|=q^5$, $U_1\in\mathcal{A}(Q_1)\subset \mathcal{A}(S)$ and $V_1/U_1$ is an irreducible module of order $q^4$ for $L_1/Q_1$ such that $C_{V_1/U_1}(S)=U_1V_2/U_1$;
    \item $|Z(V_1)|=q^3$ and $U_1/Z(V_1)$ is a natural module for $L_1/Q_1$;
    \item $|Z(Q_1)|=q^2$, $Z(V_1)/Z(Q_1)$ is centralized by $L_1$ and $Z(Q_1)$ is a natural module for $L_1/Q_1$; and
    \item $V_1=[Q_1, Q_1]$, $Z(V_1)=[V_1, V_1]=[U_1, Q_1]$ and $Z(Q_1)=[Z(V_1), Q_1]=[U_1, V_1]$.
\end{enumerate}
\end{proposition}

\begin{proposition}
The following hold:
\begin{enumerate}
    \item $|Q_2|=q^{10}$;
    \item $\Omega(S/Q_2)=Z(S/Q_2)=\Phi(S/Q_2)=V_1Q_2/Q_2$;
    \item $|C_{Q_2}(V_2)|=q^6$ and $Q_2/C_{Q_2}(V_2)$ is a natural module for $L_2/Q_2\cong \Sz(q)$;
    \item $|V_2|=q^5$, $V_2\in\mathcal{A}(Q_2)\subset \mathcal{A}(S)$ and $C_{Q_2}(V_2)/V_2$ is centralized by $L_2$;
    \item $|Z(Q_2)|=q$, $L_2$ centralizes $Z(Q_2)$ and $V_2/Z(Q_2)$ is a natural module for $L_2/Q_2$; and
    \item $V_2=\Omega(C_{Q_2}(V_2))=[Q_2, C_{Q_2}(V_2)]$ and $Z(Q_2)=[Q_2, V_2]=[C_{Q_2}(V_2), C_{Q_2}(V_2)]$.
\end{enumerate}
\end{proposition}

The majority of the work in this section is in reducing the list of candidate essential subgroups of any saturated fusion system supported on $S$, ultimately showing in \cref{F4Essen} that the only possible essentials are $Q_1$ and $Q_2$. This comes as a consequence of a number of lemmas and propositions which follow.

\begin{proposition}
Let $\fs$ be a saturated fusion system supported on $S$ and $E\in\mathcal{E}(\fs)$. Then $Z(V_1)\le E$.
\end{proposition}
\begin{proof}
Aiming for a contradiction, assume that $E\in\mathcal{E}(\fs)$ with $Z(V_1)\not\le E$. We have that $\Omega(E)\le E\cap \Omega(S)=E\cap V_1Q_2$ so that $[Z(V_1), \Omega(E)]\le Z(S)\le Z(Q_1)$. Now, unless $E\cap Z(Q_1)=Z(S)$, we have that $C_S(E\cap Z(Q_1))=Q_1$ so that $Z(E)\le Q_1$ and $\Omega(Z(E))\le V_1$. In this case, $Z(Q_1)$ centralizes the chain $\{1\}\normaleq Z(E)\normaleq E$ so that $Z(Q_1)\le E$. Then $Z(V_1)\le N_S(E)$ and $Z(V_1)$ now normalizes the chain $\{1\}\normaleq \Omega(Z(E))\normaleq \Omega(E)\normaleq E$ from which we infer that $Z(V_1)\le E$.

Therefore, to complete the argument, we may assume that $E\cap Z(Q_1)=Z(S)$. Then $Z(Q_1)$ centralizes $E/\Omega(Z(E))$ and applying \cref{SEFF}, we conclude that $S=EQ_1=\Omega(Z(E))Q_1$ and $N_S(E)=EZ(Q_1)$. But $\Omega(Z(E))\le V_1Q_2$ by \cref{F4Basic} and we have that $Q_2=(Q_1\cap Q_2)\Omega(Z(E))$. Since $Q_2/C_{Q_2}(V_2)$ has the structure of a natural $\Sz(q)$-module, we conclude that $E\le Q_2$.

Now, $E\le Q_2$ so that $V_2\le N_S(E)$. Since $N_S(E)=EZ(Q_1)$, we deduce that $V_2\cap E$ is elementary of order $q^4$. In addition, \cref{SEFF} also yields that $\Phi(E)\cap Z(S)=\{1\}$ and so $[E\cap V_2, E]=\{1\}$ and $E\cap V_2\le \Omega(Z(E))$. Hence, $E\le C_{Q_2}(E\cap V_2)$ from which we deduce that $EC_{Q_2}(V_2)/C_{Q_2}(V_2)=\Omega(Z(E))C_{Q_2}(V_2)/C_{Q_2}(V_2)$ has order $q$. Furthermore, $\{1\}=\mho^1(E)\cap Z(S)\ge \mho^1(E\cap C_{Q_2}(V_2))$ from which we deduce that $E\cap C_{Q_2}(V_2)=E\cap V_2\le \Omega(Z(E))$. All in, we have that $E$ is elementary abelian of order $q^5$. Now, applying \cref{elementary2} we see that that $E$ and $V_2=(E\cap Q_1)Z(Q_1)$ are the only elementary abelian subgroups of $N_S(E)$ of order $q^5$. But now, $V_2\le N_S(E)\le Q_2$ so that $Q_2\le N_S(N_S(E))$. Indeed, $Q_2$ acts on the set $\{E, V_2\}$ and since $Q_2$ normalizes $V_2$ we must also have that $Q_2$ normalizes $E$, a contradiction. Hence, $Z(V_1)\le E$.
\end{proof}

\begin{proposition}\label{Q_1le1}
Let $\fs$ be a saturated fusion system supported on $S$ and $E\in\mathcal{E}(\fs)$. Then $E\not\le V_1$. 
\end{proposition}
\begin{proof}
Aiming for a contradiction, suppose that $E\le V_1$ so that $Z(V_1)\le \Omega(Z(E))$. Since $[V_1, V_1]=Z(V_1)\le E$, we have that $V_1\le N_S(E)$. Indeed, $[V_1, E]\le Z(V_1)\le \Omega(Z(E))$ and writing $A_E:=\langle \Aut_{V_1}(E)^{\Aut_{\fs}(E)}\rangle$, $A_E$ centralizes $E/\Omega(Z(E))$. If $E<V_1$, then $A_E$ contains elements of odd order. 

\sloppy{Assume that $|V_2\cap E|>q^4$. Then $\Omega(Z(E))\le \Omega(C_{E}(V_2\cap E))=V_2\cap E$. Then $V_2$ centralizes the chain $\{1\}\normaleq \Omega(Z(E))\le E$ so that $V_2\le E$ by \cref{Chain}. Hence, $\Omega(Z(E))\le V_2$. Now, either $E\le Q_2$ or $Z(V_1)\le \Omega(Z(E))\le C_{V_2}(E)=Z(V_1)$. In the former case, since now $|V_1/E|>2$, applying \cref{MaxEssenEven} we have that $O^2(O^{2'}(\Aut_{\fs}(E)))\le A_E$ centralizes $E/\Omega(Z(E))$ and so normalizes $V_2$. Moreover, $Q_2\le N_S(E)$, $|Q_2/E|\geq q^2$ and $[Q_2, V_2]=Z(S)$ has order $q$ so that applying \cref{SEFF} and \cref{MaxEssenEven}, using that $O^2(O^{2'}(\Aut_{\fs}(E)))\le \langle \Aut_{Q_2}(E)^{\Aut_{\fs}(E)}\rangle$, $O^2(O^{2'}(\Aut_{\fs}(E)))$ centralizes $V_2$, a contradiction for then $O^2(O^{2'}(\Aut_{\fs}(E)))$ centralizes $E$. In the latter case, $V_1$ centralizes the chain $\{1\}\normaleq Z(V_1)\normaleq E$, and we deduce that $V_1=E$ so that $E\normaleq S$. Indeed, by \cite{Bender}, $O^{2'}(\Out_{\fs}(E))O_{2'}(\Out_{\fs}(E))/O_{2'}(\Out_{\fs}(E))\cong \SU_3(q)$. Since $Z(Q_1)$ has index $q$ in $Z(V_1)$ and $|Q_1/E|=q^2$, $O^{2'}(\Aut_{\fs}(E))=O^2(O^{2'}(\Aut_{\fs}(E)))\Inn(E)$ centralizes $Z(V_1)$, a contradiction since $[\Aut_S(E), Z(V_1)]=[S, Z(V_1)]=Z(Q_1)$. Hence, $|V_2\cap E|\leq q^4$.}

If $V_2\cap E=Z(V_1)$ then $|V_2E/E|=q^2$ and since $m_2(E)\leq m_2(S)=5n$, we have that $Z(V_1)$ has index at most $q^2$ in $\Omega(Z(E))$ and is centralized by $V_2$. By \cref{SEFF}, $\Omega(Z(E))=\Omega(E)\in\mathcal{A}(E)\subseteq \mathcal{A}(S)$, $N_S(E)=EV_2$ and $\Omega(Z(E))/C_{\Omega(Z(E))}(O^2(\Aut_{\fs}(E)))$ is a natural $\SL_2(q^2)$-module. In particular, for any $x\in N_S(E)\setminus E$, $C_{\Omega(Z(E))}(x)=Z(V_1)$. Since $U_1\not\le E$, else $E\cap V_2\ge U_1\cap V_2>Z(V_1)$, we must have that $U_1\cap E=Z(V_1)$ and $N_S(E)=EU_1=EV_2$. Then 
\begin{align*}
[V_2, \Omega(Z(E))]&=[N_S(E), \Omega(Z(E))]=[U_1E, \Omega(Z(E))]=[U_1, \Omega(Z(E))]\\
&\le [U_1, V_1]=Z(Q_1).
\end{align*}
Since $V_2/Z(S)$ has the structure of a natural $\Sz(q)$-module for $L_2/Q_2$, we deduce that $\Omega(Z(E))\le Q_2$ and $[\Omega(Z(E)), V_2]\le Z(S)$ has order at most $q$, contradicting the module structure of $\Omega(Z(E))$.

Hence, $V_2\cap E>Z(V_1)$. Again, since $V_2/Z(S)$ has the structure of a natural $\Sz(q)$-module, we have that $\Omega(Z(E))\le Q_2$. Then $[V_2, \Omega(Z(E))]\le Z(S)$ has order at most $q$ and \cref{SEFF} gives $|V_2\cap E|=q^4$. Indeed, $V_1\le N_S(E)=V_2E\le V_1$. Then $Z(S)=[V_2\cap E, E\cap Q_2]\le E'$ so that $V_2$ centralizes the chain $\{1\}\normaleq E'\normaleq E'\Omega(Z(E))\normaleq E$, a contradiction by \cref{Chain}. Hence, $E\not\le V_1$.
\end{proof}

\begin{lemma}\label{Q_1le2}
Let $\fs$ be a saturated fusion system supported on $S$ and $E\in\mathcal{E}(\fs)$. If $E\le V_1(Q_1\cap Q_2)$, then $V_2\le E$.
\end{lemma}
\begin{proof}
Aiming for a contradiction, suppose that $E\le V_1(Q_1\cap Q_2)$ and $V_2\not\le E$. By \cref{Q_1le1}, we have that $E\not\le V_1$. Note that $\Omega(E)\le E\cap \Omega(Q_1)=E\cap V_1$, $Z(V_1)\le \Omega(Z(\Omega(E)))$ and $\Omega(Z(E))\cap Z(V_1)=Z(Q_1)$. Set $B_E:=\langle \Aut_{V_2}(E)^{\Aut_{\fs}(E)}\rangle$. Note that if $|V_2\cap E|>q^4$, then as $V_2\cap E\le \Omega(E)$, $Z(V_1)\le \Omega(Z(\Omega(E))\le \Omega(C_E(V_2\cap E))=V_2$. Hence, $V_2$ centralizes the chain $\{1\}\normaleq \Omega(Z(\Omega(E)))\normaleq E$, a contradiction. Hence, $|V_2\cap E|\leq q^4$ and by \cref{MaxEssenEven}, since $q>2$ we have that $O^2(O^{2'}(\Aut_{\fs}(E)))\le B_E$.

If $O^2(B_E)$ centralizes $E/\Omega(Z(E))$ then it normalizes $Z(V_1)\Omega(Z(E))$ and so normalizes $[Z(V_1), E]=Z(S)$. Indeed, if $\Omega(Z(E))\le Q_2$, then as $[V_2, Q_2]=Z(S)$, we have that $O^2(B_E)$ centralizes the chain $\{1\}\normaleq Z(S)\normaleq \Omega(Z(E))\normaleq E$, a contradiction. Then $\Omega(Z(E))\not\le Q_2$ so that $V_2\cap E=Z(V_1)$. But now $V_2$ centralizes $(\Omega(Z(E))\cap Q_2)/Z(S)$ and since $|V_2E/E|=q^2$, \cref{SEFF} implies that $O^2(B_E)$ centralizes $\Omega(Z(E))/Z(S)$, and we arrive at the same contradiction as before. Hence, $O^2(B_E)$ is non-trivial on $E/\Omega(Z(E))$.

Now, $E\cap Q_2$ has index at most $q$ in $E$ and since $|V_2E/E|\geq q$, $Z(S)\le \Omega(Z(E))$ and $[V_2, Q_2]=Z(S)$, applying \cref{SEFF} we have that $|V_2E/E|=q$, $N_S(E)=EV_2$ and $V_1Q_2=EQ_2$. Unless $V_2\cap E=[V_2, Q_1]$, we have that $Z(V_1)=Z(Q_1)[V_2\cap E, E]$ so that $C_E(E')=C_E(E'Z(Q_1))\le C_E(Z(V_1))=E\cap V_1$. Since $[U_1, E]\le [U_1, Q_1]=Z(V_1)\le C_E(E')\le V_1$ and $[U_1, V_1]=Z(Q_1)\le \Omega(Z(E))$, we have that $U_1$ centralizes the chain $\Omega(Z(E))\normaleq C_E(E')\normaleq E$. Since $O^2(B_E)$ is non-trivial on $E/\Omega(Z(E))$ and applying \cref{MaxEssenEven} since $|V_2E/E|=q$, we infer that $U_1\le E$. But then $[V_2, Q_1]\le U_1\le E$ and $|V_2\cap E|>q^4$, a contradiction.

Therefore, if $V_2\not\le E$ then $V_2\cap E=[V_2, Q_1]$ and $\Omega(Z(\Omega(E)))\le \Omega(C_S(V_2\cap E))=\Omega(U_1C_{Q_2}(V_2))=U_1V_2$. Since $Z(V_1)\le \Omega(Z(\Omega(E)))\cap V_2$, if $\Omega(Z(\Omega(E)))\le V_2$, then $V_2$ centralizes the chain $\{1\}\normaleq \Omega(Z(\Omega(E)))\normaleq E$, a contradiction since $V_2\not\le E$. It follows by \cref{elementary2} that $\Omega(Z(\Omega(E)))\le U_1$, $U_1$ centralizes a chain and $U_1\le \Omega(E)$. But then $[V_2, E]\le Z(V_1)\le \Omega(Z(\Omega(E)))\le U_1$ and $[V_2, \Omega(Z(\Omega(E)))]\le [V_2, U_1]=Z(S)\le \Omega(Z(E))$ and finally, we have that $O^2(B_E)$ centralizes $E/\Omega(Z(E))$, a contradiction. Hence, $V_2\le E$.
\end{proof}

\begin{proposition}\label{Q_1le3}
Let $\fs$ be a saturated fusion system supported on $S$ and $E\in\mathcal{E}(\fs)$. Then $E\not\le V_1(Q_1\cap Q_2)$. 
\end{proposition}
\begin{proof}
We assume throughout that $E\le V_1(Q_1\cap Q_2)$, so that by \cref{Q_1le2}, $V_2\le E$. In particular, $\Omega(Z(E))\le \Omega(C_S(V_2))=V_2$. If $E\le Q_2$, then $Q_2\le N_S(E)$ and $[Q_2, \Omega(Z(E))]=Z(S)$. Applying \cref{SEFF} and \cref{MaxEssenEven}, either $E=Q_1\cap Q_2$ or $\Aut_S(E)\le O^2(\langle \Aut_{Q_2}(E)^{\langle \Aut_{\fs}(E)\rangle})\Inn(E)$ centralizes $\Omega(Z(E))$. In the latter case, we get that $Z(S)=[\Aut_{Q_2}(E), Z(Q_1)]\le [\Aut_{Q_2}(E), \Omega(Z(E))]=\{1\}$, absurd. Hence, if $E\le Q_2$ then $E=Q_1\cap Q_2\normaleq S$. Then $S/E$ is a group of exponent strictly greater than $2$, and has $V_1Q_2/E\le Z(S/Q_2)$. However, by \cref{MaxEssenEven}, $S/E$ is isomorphic to a Sylow $2$-subgroup of $\SL_2(q)$, $\Sz(q)$ or $\SU_3(q)$, a contradiction. 

Hence, if $E\le V_1(Q_1\cap Q_2)$ then $V_2\le E$, $E\not\le Q_2$ and $\Omega(Z(E))\le C_{V_2}(E)\le Z(V_1)$. Moreover, $Z(V_1)\le \Omega(Z(\Omega(E)))\le V_2$ and since $[U_1, E]\le [U_1, Q_1]=Z(V_1)$,  $U_1$ centralizes the chain $\{1\}\normaleq \Omega(Z(E))\normaleq \Omega(Z(\Omega(E)))\normaleq E$ and $U_1\le E$. Since $[E, V_1]\le [Q_1, V_1]=U_1\le E$, we have that $V_1\le N_S(E)$. Then $U_1V_2\le E$ and we further have that $Z(V_1)\le \Omega(Z(\Omega(E)))\le [V_2, Q_1]$ and $[V_2, Q_1, V_1]=Z(Q_1)\le \Omega(Z(E))\le Z(V_1)$. Hence, $V_1$ centralizes the chain $\{1\}\normaleq \Omega(Z(E))\normaleq \Omega(Z(\Omega(E)))\normaleq \Omega(E)\normaleq E$ so that $V_1\le E$. Since $E\not\le V_1$, we have that $\Omega(Z(E))=Z(Q_1)$ and $Z(V_1)=\Omega(Z(\Omega(E)))$. Then $U_1$ is the preimage in $V_1$ of $Z(V_1/Z(Q_1))$ and since $V_1=\Omega(E)$ is characteristic in $E$, so too is $U_1$. But now $Q_1$ centralizes the chain $\{1\}\normaleq Z(Q_1)\normaleq Z(V_1)\normaleq U_1\normaleq V_1\normaleq E$, a contradiction since $Q_1\not\le E$. 
\end{proof}

\begin{proposition}\label{Q_1le}
Let $\fs$ be a saturated fusion system supported on $S$ and $E\in\mathcal{E}(\fs)$. If $E\le Q_1$, then $E=Q_1$.
\end{proposition}
\begin{proof}
By \cref{Q_1le3}, any essential subgroup $E$ of $\fs$ which is contained in $Q_1$ is not contained in $V_1(Q_1\cap Q_2)$. Since $Q_2/C_{Q_2}(V_2)$ has the structure of a natural $\Sz(q)$-module, we deduce that $\Omega(Z(E))\le \Omega(U_1C_{Q_2}(V_2))=U_1V_2$. Hence, $[[V_2, Q_1], \Omega(Z(E))]=\{1\}$. But now, $[[V_2, Q_1], E]=Z(V_1)\le \Omega(E)\le V_1$, $[[V_2, Q_1], \Omega(E)]\le [[V_2, Q_1], V_1]=Z(Q_1)\le \Omega(Z(E))$ and \cref{Chain} implies that $[V_2, Q_1]\le E$.

Assume that $Z(Q_1)\not\le E'$. Since $E\not\le V_1Q_2$, we have that $Z(Q_1)=[[V_2, Q_1], E, E]Z(S)$ and so we must have that $Z(S)\not\le E'$. Hence, $E\cap Q_2$ centralizes $[V_2, Q_1]$ so that $E\cap Q_2\le U_1C_{Q_2}(V_2)$. If $V_2\cap E>[V_2, Q_1]$ then $\Omega(Z(E))\le \Omega(Z(\Omega(E)))\le \Omega(C_S(V_2\cap E))=V_2$. Since $[V_2, \Omega(E)]\le [V_2, V_1]=Z(V_1)\le \Omega(Z(\Omega(E)))$, $V_2$ centralizes the chain $\{1\}\normaleq \Omega(Z(\Omega(E)))\normaleq \Omega(E)\normaleq E$ and $V_2\le E$. Indeed, since $E\not\le V_1Q_2$, $Z(Q_1)=\Omega(Z(E))$. Since $Z(S)\not\le E'$, we have that $[E\cap Q_2, V_2]=\{1\}$ so that $E\cap Q_2\le C_{Q_2}(V_2)$. In particular, $U_1\cap E=[V_2, Q_1]$. But $U_1$ centralizes the chain $\{1\}\normaleq \Omega(Z(E))\normaleq \Omega(E)\normaleq E$, a contradiction. 

Hence, if $Z(Q_1)\not\le E'$ then $[V_2, Q_1]=V_2\cap E$. Indeed, $V_2\cap E\le \Omega(E)$ and $Z(V_1)\le \Omega(Z(\Omega(E)))\le Q_2$. In particular, $V_2$ centralizes the chain $\Omega(Z(E))\normaleq \Omega(Z(\Omega(E)))\normaleq \Omega(E)\normaleq E$ and so $C_{V_2}(\Omega(Z(E)))=V_2\cap E$ by \cref{Chain}. Since $[\Omega(Z(E)), V_2]\le Z(S)$ and $|V_2/V_2\cap E|=q$, applying \cref{SEFF} we deduce that $[\Omega(Z(E)), V_2]=Z(S)$, $N_S(E)=V_2E$ and $\Omega(Z(E))\cap V_2$ has index $q$ in $\Omega(Z(E))$. Since $\Omega(Z(E))\le \Omega(U_1C_{Q_2}(V_2))=U_1V_2$ and $\Omega(Z(E))(V_2\cap E)$ is elementary abelian of order $q^5$, $\Omega(Z(E))\in\mathcal{A}(U_1V_2)$. By \cref{elementary2}, $U_1=\Omega(Z(E))(V_2\cap E)\le E$. Since $[E, V_1]\le [Q_1, V_1]=U_1\le E$, $V_1\le N_S(E)$. But $V_1\cap Q_2\le N_S(E)\cap Q_2=V_2(E\cap Q_2)\le U_1C_{Q_2}(V_2)<V_1\cap Q_2$, and we have a contradiction.

Therefore, $Z(Q_1)\le E'$ and since $[E, U_1]\le [Q_1, U_1]=Z(V_1)$ and $[\Omega(E), U_1]\le Z(Q_1)$, we have by \cref{Chain} that $U_1\le E$. In particular, $U_1\le \Omega(E)$ and $Z(V_1)\le\Omega(Z(\Omega(E)))\le U_1$. Since $[\Omega(E), V_2]\le [V_1, V_2]=Z(V_1)\le \Omega(Z(\Omega(E)))$ and $[\Omega(Z(\Omega(E))), V_2]\le [U_1, V_2]=Z(S)\le E'$, we deduce by \cref{Chain} that $V_2\le E$. As before, $V_2\le \Omega(E)$ and $\Omega(Z(\Omega(E)))\le V_2\cap U_1=[V_2, Q_1]$. Moreover, $\Omega(Z(E))\le C_{V_2}(E)$ and since $E\not\le V_1Q_2$, $\Omega(Z(E))=Z(Q_1)$. But now, $[E, V_1]\le [Q_1, V_1]=U_1\le \Omega(E)$, $[\Omega(E), V_1]\le V_1'=Z(V_1)\le \Omega(Z(\Omega(E)))$, $[\Omega(Z(\Omega(E))), V_1]\le [[V_2, Q_1], V_1]=Z(Q_1)=\Omega(Z(E))$ and so \cref{Chain} implies that $V_1\le E$. In fact, $V_1=\Omega(E)$ and $Z(V_1)=Z(\Omega(E))$. Then $U_1$ is the preimage in $V_1$ of $Z(V_1/Z(Q_1))$ and as $V_1$ and $Z(Q_1)$ are characteristic in $E$, so too is $U_1$. But then $Q_1$ centralizes the chain $\{1\}\normaleq Z(Q_1)\normaleq Z(V_1)\normaleq U_1\normaleq V_1\normaleq E$ and \cref{Chain} implies that $E=Q_1$, as desired.
\end{proof}

\begin{proposition}\label{Q_2le}
Suppose that $E\in\mathcal{E}(\fs)$ and $E\le Q_2$. Then $E=Q_2$.
\end{proposition}
\begin{proof}
Suppose first that $V_2\not\le E$. Note that if $E\cap V_2\not\le Z(E)$ then $Z(S)=[E, E\cap V_2]\le \Phi(E)$ so that $[V_2, E]\le \Phi(E)$, a contradiction by \cref{Chain}. Thus, $Z(V_1)\le E\cap V_2\le Z(E)$ and we have that $E\le V_1$. By \cref{Q_1le1}, this is a contradiction. Hence, $V_2\le E$. Indeed, $\Omega(Z(E))\le \Omega(C_E(V_2))=V_2$ and as $[Q_2, E]\le [Q_2, Q_2]=V_2\le E$, $Q_2\le N_S(E)$. 

We aim to show that $Z(S)$ is normalized by $O^{2'}(\Aut_{\fs}(E))$. Firstly, if $|Q_2/E|<q$ then it is clear that $Z(S)=Z(E)$. Hence, $m_2(\Out_S(E))\geq 2$ and applying \cref{MaxEssenEven}, we have that $O^{2'}(\Aut_{\fs}(E))=\langle \Aut_{Q_2}(E)^{\Aut_{\fs}(E)}\rangle$. Now, $[Q_2, \Omega(Z(E))]\le Z(S)$ and we conclude by \cref{SEFF} that either $|Q_2/E|=q$, or $O^{2'}(\Aut_{\fs}(E))$ is trivial on $\Omega(Z(E))$. In the latter case, it is clear that $Z(S)$ is normalized by $O^{2'}(\Aut_{\fs}(E))$. In the former case, we must have that $|Q_2/E|=q$ and $C_{Q_2}(V_2)\le E$, else $\Omega(Z(E))=Z(S)$ and $O^{2'}(\Out_{\fs}(E))$ is trivial on $\Omega(Z(E))$. It follows that $|\Omega(Z(E))|=q^2$, $E=C_{Q_2}(\Omega(Z(E)))$ and \cref{SEFF} gives that $O^{2'}(\Out_{\fs}(E))\cong \SL_2(q)$ and $N_S(E)=Q_2$. Note that if $E(V_1\cap Q_2)<Q_2$ then $\Omega(Z(E))\le C_{V_2}(E\cap V_1)\le Z(V_1)$ and as $E=C_{Q_2}(\Omega(Z(E)))$, we have that $V_1\cap Q_2\le E$. But then $V_1\le N_S(E)$, a contradiction. Hence, $Q_2=E(V_1\cap Q_2)$ so that $V_2=[Q_2, C_{Q_2}(V_2)]=[E, C_{Q_2}(V_2)][V_1\cap Q_2, C_{Q_2}(V_2)]=E'Z(V_1)$. Since $|\Omega(Z(E))|=q^2$, we deduce that $Z(S)=[E, E']$ is characteristic in $E$ and normalized by $O^{2'}(\Aut_{\fs}(E))$, and $O^{2'}(\Aut_{\fs}(E))$ acts trivially on $Z(S)$.

Hence, $O^{2'}(\Aut_{\fs}(E))$ is trivial on $Z(S)$ in all cases. Since $Q_2\le N_S(E)$ and $\Omega(Z(E))\le V_2$, we deduce that $\Omega(Z(E))\le C_{V_2}(Q_2)$ and $\Omega(Z(E))=Z(S)$. Since $L_2$ acts trivially on $C_{Q_2}(V_2)/V_2$, for any subgroup $A$ with $V_2\le A\le C_{Q_2}(V_2)$, $C_{Q_2}(A/Z(S))\normaleq L_2$. Since $C_{Q_2}(V_2)\le C_{Q_2}(A/Z(S))$ and $V_2=Z(Q_2/Z(S))$, if $V_2<A$ then we have that $C_{Q_2}(A/Z(S))=C_{Q_2}(V_2)$. Hence, either $V_2$ is the preimage in $E$ of $Z(E/Z(S))$ or $E\cap C_{Q_2}(V_2)=V_2$. Even in this latter case, since $Q_2$ centralizes $V_2/Z(S)$, we must have by \cref{SEFF} that $|Q_2/E|\leq |E/V_2|$ and since $|Q_2/V_2|=q^5$, we certainly have that $|E/V_2|\geq q^2$. One can calculate that $V_2$ is the preimage in $E$ of $Z(E/Z(S))$ in this case also. Hence, $V_2$ is $O^{2'}(\Aut_{\fs}(E))$-invariant and $Q_2$ centralizes the chain $\{1\}\normaleq Z(S)\normaleq V_2\normaleq E$ so that $E=Q_2$, as desired.
\end{proof}

For the remainder of this section, ultimately aiming for a contradiction, we let $E$ be an essential subgroup of $\fs$ distinct from $Q_1$ and $Q_2$, and chosen maximally with respect to this condition. Hence, $E$ is contained in no other essential subgroups of $\fs$ and so, applying the Alperin--Goldschmidt theorem and using that $E$ is receptive in $\fs$, for any $\alpha\in\Aut_{\fs}(E)$ which lifts to a subgroup of $S$ strictly larger than $E$, $\alpha$ may be lifted to some $\hat{\alpha}\in\Aut_{\fs}(S)$. 

\begin{lemma}\label{VBin}
$V_2\le E$.
\end{lemma}
\begin{proof}
Aiming for a contradiction we assume that $V_2\not\le E$. Since $E\not\le Q_1$, we have that $Z(S)=[Z(Q_1), E]\le E'$. Suppose first that $E\not\le V_1Q_2$. Then $Z(Q_1)=Z(S)[Z(V_1), E]\le E'$ since $V_2/Z(S)$ has the structure of a natural $\Sz(q)$-module. Since $\Omega(E)E'\le V_1Q_2$ and $[V_2, Q_1, V_1Q_2]=Z(Q_1)$, we have that $[V_2, Q_1]$ centralizes the chain $E'\normaleq \Omega(E)E'\normaleq E$ and $[V_2, Q_1]\le E$. Then $Z(V_1)=[V_2, Q_1, E]\le E'$ and now $V_2$ centralizes the chain $E'\le \Omega(E)E'\normaleq E$, a contradiction. Hence, if $V_2\not\le E$, then $E\le V_1Q_2$. 

Now, $[V_2, E]\le Z(V_1)\le E$ so that $V_2\le N_S(E)$. Since $E\not\le Q_1$, $Z(S)\le [E, Z(Q_1)]\le E'$. Then $E\cap Q_2$ has index at most $q$ in $E$ and as $[V_2, E\cap Q_2]\le Z(S)\le E'$, we deduce by \cref{SEFF} that $|V_2\cap E|\geq q^4$. Set $A_E:=\langle \Aut_{V_2}(E)^{\Aut_{\fs}(E)}\rangle$. 

Suppose that $|V_2\cap E|=q^4$ so that by \cref{SEFF}, we have $EQ_2=V_1Q_2$, $N_S(E)=EV_2$ and $O^{2'}(\Out_{\fs}(E))\cong \SL_2(q)$. Assume, in addition, that $\Omega(Z(E))$ contains an $\Out_{\fs}(E)$-chief factor. Then $|\Omega(Z(E))C_{Q_2}(V_2)/C_{Q_2}(V_2)|\geq q$ and since $\Omega(Z(E))(E\cap V_2)$ is elementary abelian, $|\Omega(Z(E))C_{Q_2}(V_2)/C_{Q_2}(V_2)|=q$ and $\Omega(Z(E))(E\cap V_2)\in\mathcal{A}(S)$. Hence, $\mathcal{A}_{\normaleq}(E)\ne \emptyset$. Suppose that $A\in\mathcal{A}_{\normaleq}(E)$ with $A\not\le Q_1$. Then $(A\cap Q_1)Z(Q_1)\in\mathcal{A}(S)$ and $A\cap Z(Q_1)=Z(S)$ from which we deduce that $S=AQ_1$. But then $C_{V_1/U_1}(A)=U_1V_2/U_1$ and since $\Omega(Z(E))\le V_1$, we have that $\Omega(Z(E))\le U_1V_2$. In particular, $\Omega(Z(E))(V_2\cap E)\le U_1V_2$ and by \cref{elementary2} we have that $\Omega(Z(E))(V_2\cap E)=U_1$ and $V_2\cap E=[V_2, Q_1]$. But $E\not\le Q_1$ and so $[V_2, Q_1]=Z(V_1)[E, U_1]=Z(V_1)[E, \Omega(Z(E))(V_2\cap E)]\le Z(V_1)$, a contradiction. Hence, $J_{\normaleq}(E)$ is characteristic in $E$ and contained in $\Omega(Q_1)=V_1$. Then $[V_2, E]\le Z(V_1)\le Z(J_{\normaleq}(E))$. Since $\Omega(Z(E))(V_2\cap E)\in\mathcal{A}_\normaleq(E)$, $Z(J_{\normaleq}(E))$ centralizes $V_2\cap E$ and so is contained in $Q_2$. But then $[V_2, Z(J_{\normaleq}(E))]\le Z(S)\le E'$ and so by \cref{Chain}, we have a contradiction.

Therefore, if $|V_2\cap E|=q^4$ then $\Omega(Z(E))$ contains no non-central $\Out_{\fs}(E)$-chief factors. Since $E$ is maximally essential, $\Aut_{\fs}(E)$ normalizes $Z(S)$. On the other hand, if $|V_2\cap E|>q^4$, then $\Omega(Z(E))\le \Omega(C_{Q_2}(V_2\cap E))=V_2$ and $A_E$ centralizes $\Omega(Z(E))$. Hence, $Z(S)$ is normalized by $A_E$ in all circumstances. Set $A$ to be the preimage in $E$ of $Z(E/Z(S))$ so that $Z(V_1)\le A$. In particular, $C_E(A)\le E\cap V_1$. Since $|V_2\cap E|\geq q^4$, we have that $A\le Q_2$ so that $V_2$ centralizes $A/Z(S)$ and $Z(S)$. Hence, there are elements of $A_E$ of odd order which centralize $A$. By the three subgroups lemma, such automorphism also centralize $E/C_E(A)$. But now, $[V_2, C_E(A)]\le [V_2, V_1]=Z(V_1)\le A$ and so there are odd order elements which centralize $E/C_E(A)A$, $C_E(A)A/A$, $A/Z(S)$ and $Z(S)$, a contradiction.
\end{proof}

Since $V_2\le E$, $V_2\in\mathcal{A}(S)$ and $V_2\normaleq S$, we have that $\emptyset\ne \mathcal{A}_{\normaleq}(E)\subset \mathcal{A}(S)$. In particular, the group $J_{\normaleq}(E)$ is well defined.

\begin{lemma}\label{EleOmega}
$E\le V_1Q_2$.
\end{lemma}
\begin{proof}
We suppose throughout that $E\not\le V_1Q_2$. Since $V_2\le E$, we have that $\Omega(Z(E))\le \Omega(C_S(V_2))=V_2$. Moreover,  $V_2/Z(S)$ has the structure of a natural $\Sz(q)$-module, so that $\Omega(Z(E))\le Z(Q_1)$. Finally, $E\not\le Q_1$ so that $\Omega(Z(E))=Z(S)$. Let $A\in\mathcal{A}_{\normaleq}(E)$. If $A\not\le Q_1$, then $(A\cap Q_1)Z(Q_1)$ is also elementary abelian and $A\cap Z(Q_1)=Z(S)$. We deduce that $S=AQ_1$. Indeed, we have that $Q_2=A(Q_1\cap Q_2)$. Since there is $x\in E\setminus V_1Q_2$ and $Q_2/C_{Q_2}(V_2)$ has the structure of a natural $\Sz(q)$-module, we deduce that $Q_2=AC_{Q_2}(V_2)$. But then $[A, C_{Q_2}(V_2)]\le A\cap C_{Q_2}(V_2)=Z(S)$ and $[Q_2, C_{Q_2}(V_2)]=Z(S)$, a contradiction. Hence, if $A\in\mathcal{A}_\normaleq(E)$, then $A\le \Omega(Q_1)=V_1$. Then $Z(V_1)\le \Omega(Z(J_{\normaleq}(E)))\le \Omega(C_E(V_2))=V_2$. Since $E\not\le V_1Q_2$ and $V_2/Z(S)$ has the structure of a natural $\Sz(q)$-module, we have that $Z(Q_1)$ is the intersection of $\Omega(Z(J_{\normaleq}(E)))$ and the preimage in $E$ of $Z(E/\Omega(Z(E))$ and so, $Z(Q_1)$ is characteristic in $E$.

Now, $C_E(Z(Q_1))=E\cap Q_1$ is characteristic in $E$. Then $[U_1, E]\le E\cap Q_1$, $[U_1, E\cap Q_1]\le Z(V_1)\le \Omega(E\cap Q_1)\le E\cap V_1$, $[U_1, \Omega(E\cap Q_1)]\le Z(Q_1)$ and $[U_1, Z(Q_1)]=\{1\}$ so that $U_1\le E$. Now, $U_1V_2\le J_{\normaleq}(E)\le V_1$ and we have that $Z(V_1)\le Z(J_{\normaleq}(E))\le [V_2, Q_1]$. Then $V_1$ centralizes the chain $\{1\}\normaleq Z(Q_1)\normaleq Z(J_{\normaleq}(E))\normaleq J_{\normaleq}(E)\normaleq E\cap Q_1$. If $V_1\not\le E$, then $N_{EV_1}(E)>E$ and $N_{EV_1}(E)$ normalizes the chain $\{1\}\normaleq \dots\normaleq E\cap Q_1\normaleq E$, a contradiction. Hence, $V_1\le E$. In particular, $V_1=\Omega(E\cap Q_1)$ is characteristic in $E$. Now, $U_1$ is the preimage in $V_1$ of $Z(V_1/Z(Q_1))$ and since $V_1$ and $Z(Q_1)$ are characteristic in $E$, so too is $U_1$. Then $Q_1$ centralizes the chain $\{1\}\normaleq Z(Q_1)\normaleq Z(V_1)\normaleq U_1\normaleq V_1\normaleq E\cap Q_1$ and $Q_1\le E$. Hence, $V_1=\Phi(Q_1)\le \Phi(E)\le \Phi(S)=[Q_1, x]\Phi(Q_1)$ for all $x\in S\setminus Q_1$. Since $E\not\le Q_1$, we have that $\Phi(E)=\Phi(S)$ and $S$ centralizes $E/\Phi(E)$, a final contradiction.
\end{proof}

\begin{lemma}\label{F4Essen1}
$V_1\not\le E$.
\end{lemma}
\begin{proof}
Aiming for a contradiction, let $V_1\le E$. Since $E\ne Q_1$, we have that $E\not\le Q_1$ by \cref{Q_1le}. Let $A\in\mathcal{A}_{\normaleq}(E)$ and $A\not\le V_1$. Then $A\not\le Q_1$ and since $A$ is abelian, $A\cap Z(Q_1)=Z(S)$. Hence, $(A\cap Q_1)Z(Q_1)\in\mathcal{A}(Q_1)$ and we deduce that $S=AQ_1$. Moreover, since $(A\cap Q_1)Z(Q_1)\in\mathcal{A}(V_1)$, we have that $Z(V_1)\le AZ(Q_1)$ so that $|A\cap Z(V_1)|=q^2$. Since $U_1/Z(V_1)$ has the structure of a natural $\SL_2(q)$-module for $L_1/Q_1$, we have that $|A\cap U_1|\geq |[A, U_1](A\cap Z(V_1))|\geq q^3$. Finally, since $|[A, V_1]U_1/U_1|>q$ and since $V_1\le E\le N_S(A)$, we have that $|A|>q^5$, a contradiction since $m_2(S)=5n$. Thus, $J_{\normaleq}(E)$ is a characteristic subgroup of $E$ contained in $V_1$. Then $Z(V_1)\le Z(J_{\normaleq} (E))\le [V_2, Q_1]$ since $Z(U_1V_2)=[V_2, Q_2]$ and $U_1V_2\le J_{\normaleq}(E)$. 

Assume that $Z(V_1)<Z(J_{\normaleq}(E))$. Then $C_S(Z(J_{\normaleq}(E)))=C_E(Z(J_{\normaleq}(E)))=U_1C_{Q_2}(V_2)$ is characteristic in $E$ so that $U_1V_2=\Omega(C_{Q_2}(V_2))$ is characteristic in $E$. Then by \cref{elementary2}, $U_1$ and $V_2$ are the only maximal elementary abelian subgroups of $U_1V_2$ and since $U_1, V_2\normaleq S$, we deduce that both $U_1$ and $V_2$ are $\Aut_{\fs}(E)$-invariant. Then $Z(S)=[U_1, V_2]$ is $\Aut_{\fs}(E)$-invariant and $E\cap Q_2=C_E(V_2/Z(S))$ is also $\Aut_{\fs}(E)$-invariant. Then $Q_2$ centralizes the chain $\{1\}\normaleq Z(S)\normaleq V_2\normaleq E\cap Q_2$ so that $Q_2\le E$ and $E=V_1Q_2$. Then $Z(S)<[Z(J_{\normaleq}(E)), E]\le Z(Q_1)$ and we get that $V_1(Q_1\cap Q_2)=C_E([Z(J_{\normaleq}(E)), E])$ is $\Aut_{\fs}(E)$-invariant. Then $[Q_1, E]\le V_1(Q_1\cap Q_2)$ and $[Q_1, V_1(Q_1\cap Q_2)]= V_1\cap Q_2\le \Phi(E)$ and \cref{Chain} gives a contradiction.

Assume now that $Z(V_1)=Z(J_{\normaleq}(E))$ so that $Z(V_1)$ is characteristic in $E$. Then, as $V_1<E$, $Z(S)=[Z(V_1), E]$ and $V_1=C_{E}(Z(V_1))$ are also characteristic subgroups of $E$. Suppose that $E<V_1Q_2$. Then $V_1Q_2$ centralizes $Z(V_1)/Z(S)$ and so $B_E:=\langle \Aut_{V_1Q_2}(E)^{\Aut_{\fs}(E)}\rangle$ centralizes $Z(V_1)/Z(S)$. Indeed, $B_E$ normalizes $Z(Q_1)$ and so normalizes $E\cap Q_1$. Furthermore, $U_1$ is the preimage in $V_1$ of $Z(V_1/Z(Q_1))$ so is normalized by $B_E$, and then as $V_1<E\not\le Q_1$ by \cref{Q_1le}, $U_1\cap V_2$ is the preimage in $U_1$ of $C_{U_1/Z(V_1)}(E)$, also normalized by $B_E$. Then $V_1Q_2$ centralizes the chain $\{1\}\normaleq Z(S)\normaleq Z(Q_1)\normaleq U_1\cap V_2\normaleq U_1$ and we deduce that $O^2(B_E)$ centralizes $U_1$. By the three subgroups lemma, $O^2(B_E)$ centralizes $E/C_E(U_1)$ and since $U_1$ is self-centralizing in $S$, we have that $O^2(B_E)$ acts trivially on $E$, a contradiction. Hence, $E=V_1Q_2$ by \cref{EleOmega}.

Now, $U_1V_2$ is the preimage in $V_1$ of $C_{V_1/U_1}(S)=C_{V_1/U_1}(E)$ and so $C_{V_1/Z(V_1)}(E)=C_{U_1V_2/Z(V_1)}(Q_2)=V_2/Z(V_1)$ so that $V_2$ is characteristic in $E$. Furthermore, $Q_2=C_E(V_2/Z(S))$ is also characteristic in $E$. Since $S/E$ is elementary abelian of order $q>2$ and $E$ is maximally essential, \cref{MaxEssenEven} yields that $O^{2'}(\Out_{\fs}(E))\cong \SL_2(q)$. Let $G_E$ a model for $E$ and $L_E:=O^{2'}(G_E)$ so that $L_E/E\cong \SL_2(q)$. Then $S$ centralizes $E/Q_2$ so that $L_E/Q_2$ is a central extension of $L_E/E\cong \SL_2(q)$ by $E/Q_2$ which is elementary abelian of order $q$. Since $S/Q_2\cong T\in\syl_2(\Sz(q))$, this is a non-split extension. However, the Schur multiplier of $\SL_2(q)$ when $q=2^n>4$ is trivial and so we have a contradiction. Hence, $V_1\not\le E$. 
\end{proof}

We now demonstrate that the proposed $E$ does not exist and the only possibilities for the essential subgroups of any saturated fusion system on $S$ are $Q_1$ and $Q_2$. By \cref{Q_1le} and \cref{Q_2le}, we may assume that that $E\not\le Q_1$ and $E\not\le Q_2$ and we let $E$ be maximally essential in $\fs$ because of this. Since $V_2\le E$ by \cref{VBin} and $\Omega(C_S(V_2))=V_2$, we have that $\Omega(Z(E))\le V_2$. Since $E\not\le Q_2$, we have that $\Omega(Z(E))\le Z(V_1)$. 

\begin{proposition}\label{F4Essen}
Let $\fs$ be a saturated fusion system supported on $S$. Then $\mathcal{E}(\fs)\subseteq \{Q_1, Q_2\}$.
\end{proposition}
\begin{proof}
By \cref{Q_1le} and \cref{Q_2le}, aiming for a contradiction, it suffices to let $E$ be an essential subgroup of $\fs$ chosen maximally such that $E\not\le Q_1$ and $E\not\le Q_2$. By \cref{F4Essen1}, $V_1\not\le E$. Setting $A_E:=\langle \Aut_{V_1}(E)^{\Aut_{\fs}(E)}\rangle$, $A_E$ contains elements of odd order, and since $\Omega(Z(E))\le Z(V_1)$, $A_E$ acts trivially on $Z(S)$. Hence, $A_E$ normalizes $P_E$, the preimage in $E$ of $Z(E/Z(S))$. Moreover, since $V_2\le E\not\le Q_2$, we have that $Z(V_1)\le P_E\le V_1\cap Q_2$ and $P_E\cap V_2=Z(V_1)$. For $C_E:=C_{A_E}(P_E)$, by the three subgroups lemma, $[C_E, E]\le C_E(P_E)\le V_1$. Then $[V_1, C_E(P_E)]\le Z(V_1)\le P_E$ from which it follows that $[C_E, C_E(P_E)]\le P_E$ and so $O^2(C_E)$ centralizes $E$ and $C_E$ is a $2$-group. Hence, $A_E/\Inn(E)$ acts faithfully on $P_E$.

Assume first that $E\cap C_{Q_2}(V_2)=V_2$, noting that $m_2(E)=m_2(S)$. If $J_{\normaleq}(E)\le Q_1$, then $J_{\normaleq}(E)\le V_1$. Since $V_2\le J_{\normaleq}(E)$, we have that $Z(V_1)\le \Omega(Z(J_{\normaleq}(E))\le V_2$. Then $C_{Q_2}(V_2)$ centralizes the chain $\{1\}\normaleq \Omega(Z(J_{\normaleq}(E)))\normaleq J_{\normaleq}(E)\normaleq E$, a contradiction. Hence, there is $A\in\mathcal{A}(E)$ with $A\normaleq E$ and $A\not\le Q_1$. Since $(A\cap Q_1)Z(Q_1)$ is also elementary abelian and $A\cap Z(Q_1)=Z(S)$, we must have that $S=AQ_1$. Indeed, $C_{V_1/U_1}(A)=U_1V_2/U_1$. Then $P_EU_1/U_1\le U_1V_2$ so that $P_E\le U_1V_2$. By \cref{elementary2}, $\Omega(P_E)\le (P_E\cap V_2)(P_E\cap U_1)$. Since $E\not\le Q_1$, $P_E\cap U_1\le [V_2, Q_1]$ so that $\Omega(P_E)\le P_E\cap V_2=Z(V_1)$. Hence, $\Omega(P_E)=Z(V_1)$ is centralized by $A_E$ so that $C_E(\Omega(P_E))=V_1\cap E$ is also normalized by $A_E$. Then $[V_1, E]\le V_1\cap E$, $[V_1, V_1\cap E]\le V_1'=\Omega(P_E)$ and so $O^2(A_E)$ is trivial on $E$, a contradiction. 

Assume now that $V_2<E\cap Q_2\le C_{Q_2}(V_2)$. Since $\Omega(E)\le (E\cap V_1)(E\cap Q_2)$, $\Omega(E)\le V_1$. Since $V_2\le \Omega(E)$ and $E\not\le Q_2$, we have that $\Omega(Z(\Omega(E)))=Z(V_1)$. But then $U_1$ centralizes the chain $\{1\}\normaleq Z(V_1)\normaleq \Omega(E)\le E$, a contradiction.

Hence, $V_2<C_{E}(V_2)\normaleq L_2$ and $E\cap Q_2\not\le C_{Q_2}(V_2)$. Then $P_E\le C_{Q_2}(C_E(V_2)/Z(S))\normaleq L_2$ and since $V_2/Z(S)=Z(Q_2/Z(S))$, we have that $P_E\le C_{Q_2}(V_2)$. Since there are elements in $E\cap Q_2\setminus C_E(V_2)$, we conclude that $P_E\le V_2$. Thus, $P_E=P_E\cap V_2=Z(V_1)$. Indeed, $E\cap V_1=C_E(P_E)$ is normalized by $A_E$. But $[V_1, E]\le V_1\cap E$, $[V_1, V_1\cap E]\le V_1'=P_E$ and so $O^2(A_E)$ is trivial on $E$, a final contradiction. Therefore, $\mathcal{E}(\fs)\subseteq \{Q_1, Q_2\}$.
\end{proof}

Finally, we classify all saturated fusion systems supported on $S$. With the help of \cref{MaxEssenEven}, we are able to completely determine the local actions in $\fs$ without any reliance on the classification of finite simple groups. We employ \cref{2F4Cor} which in this usage only relies on \cite{Greenbook}, again independent of the classification.

\begin{theorem}\label{F42n}
Let $\fs$ be a saturated fusion system on a Sylow $2$-subgroup of ${}^2\mathrm{F}_4(2^n)$ for $n>1$. Then either:
\begin{enumerate}
\item $\fs=\fs_S(S: \Out_{\fs}(S))$;
\item $\fs=\fs_S(Q_1: \Out_{\fs}(Q_1))$ where $O^{2'}(\Out_{\fs}(Q_1))\cong \SL_2(2^n)$;
\item $\fs=\fs_S(Q_2: \Out_{\fs}(Q_2))$ where $O^{2'}(\Out_{\fs}(Q_2))\cong \Sz(2^n)$; or 
\item $\fs=\fs_S(G)$ where $F^*(G)=O^{2'}(G)={}^2\mathrm{F}_4(2^n)$.
\end{enumerate}
\end{theorem}
\begin{proof}
By \cref{F4Essen}, $\mathcal{E}(\fs)\subseteq \{Q_1, Q_2\}$. If $\fs$ has no essential subgroups then the Alperin--Goldschmidt theorem implies that $S\normaleq \fs$ and outcome (i) holds. 

Suppose that $Q_1\in\mathcal{E}(\fs)$. Then, as $V_1=\Phi(Q_1)$ and $S$ centralizes $(Q_1\cap Q_2)\Phi(Q_1)/\Phi(Q_1)$, an application of \cref{SEFF} yields that $O^{2'}(\Out_{\fs}(Q_1))\cong \SL_2(q)$. Since $\Out_S(E)$ is non-trivial on $Z(Q_1)$, and acts trivially on $Z(S)$, we infer by \cref{SEFF} that $Z(Q_1)$ is a natural module for $O^{2'}(\Out_{\fs}(Q_1))$. If $Q_2$ is not essential then the Alperin--Goldschmidt theorem yields outcome (ii). 

Suppose that $Q_2\in\mathcal{E}(\fs)$. Then $S/Q_2$ is isomorphic to a Sylow $2$-subgroup of $\Sz(q)$ and so we deduce using \cref{MaxEssenEven} that $O^{2'}(\Out_{\fs}(E))\cong \Sz(q)$. Indeed, since $|Q_2/C_{Q_2}(V_2)|=|V_2/Z(S)|=q^4$, it follows that $Q_2/C_{Q_2}(V_2)$ and $V_2/Z(S)$ are natural $\Sz(q)$-modules for $O^{2'}(\Out_{\fs}(Q_2))$. If $Q_1$ is not essential then outcome (iii) holds. 

Finally, we claim that whenever $\mathcal{E}(\fs)=\{Q_1, Q_2\}$, we have that $O_2(\fs)=\{1\}$. By \cite[Proposition I.4.5]{ako}, we have that $O_2(\fs)\le Q_1\cap Q_2$ and since $O^{2'}(\Out_{\fs}(Q_2))$ is irreducible on $Q_2/C_{Q_2}(V_2)$, we have that $O_2(\fs)\le C_{Q_2}(V_2)$. Then $\Phi(O_2(\fs))$ is normal subgroup of $\fs$ contained in $Z(S)$ and since $O^{2'}(\Out_{\fs}(Q_1))$ is irreducible on $Z(Q_1)$, we deduce that $\Phi(O_2(\fs))=\{1\}$ and $O_2(\fs)$ is elementary abelian. Hence, $O_2(\fs)\le \Omega(C_{Q_2}(V_2))=V_2$. Since $O^{2'}(\Out_{\fs}(Q_2))$ is irreducible on $V_2/Z(S)$ and $O^{2'}(\Out_{\fs}(Q_1))$ is irreducible on $Z(Q_1)$, if $O_2(\fs)\ne\{1\}$, then $O_2(\fs)=V_2$. But then $C_{Q_2}(V_2)=C_{Q_1}(V_2)$ would be normal in $\fs$ by \cite[Proposition I.4.5]{ako}, a contradiction. Hence, $O_2(\fs)=\{1\}$. Since $Q_1$ and $Q_2$ are characteristic in $S$, if $\mathcal{E}(\fs)=\{Q_1, Q_2\}$ then $\fs$ satisfies the hypothesis of \cref{2F4Cor} and outcome (iv) holds.
\end{proof}

\section[Fusion Systems on a Sylow \texorpdfstring{$p$}{p}-subgroup of \texorpdfstring{$\PSU_5(p^n)$}{PSU5(pn)}]{Fusion Systems on a Sylow $p$-subgroup of $\PSU_5(p^n)$}

In this section, we classify saturated fusion systems supported on a Sylow $p$-subgroup of $\PSU_5(p^n)$. Throughout we fix $S$ isomorphic to a Sylow $p$-subgroup of $\PSU_5(p^n)$ and write $q=p^n$. 

For $G:=\PSU_5(q)$, we identify $S\in\syl_p(G)$ and set $G_1:=N_G(C_S(Z_2(S)))$ and $G_2:=N_G(Q)$, where $Q$ is the preimage in $S$ of $J(S/Z(S))$ of order $q^7$. Let $L_i:=O^{p'}(G_i)$ for $i\in\{1,2\}$ so that $L_1$ is of shape $q^{4+4}:\SL_2(q^2)$ and $L_2$ is of shape $q^{1+6}:\SU_3(q)$. Note that $G_1$ and $G_2$ are the unique maximal parabolic subgroups of $G$ containing $N_G(S)$. We write $Q_i:=O_p(L_i)=O_p(G_i)$ so that $Q_1=C_S(Z_2(S))$ and $Q=Q_2$, characteristic subgroups of $S$.

Once again, we require details about $S$ and the structure of chief series of $L_1$ and $L_2$. We garner information this time from \cite{Seitz} and \cite{PSU5Par}, and when combined with the results and techniques in \cref{PrelimSec}, we are able to deduce the following well known facts.

\begin{proposition}\label{U5Basic}
The following hold:
\begin{enumerate}
\item $|S|=q^{10}$;
\item $S'=\Phi(S)=Z(Q_1)(Q_1\cap Q_2)$;
\item $m_p(S)=4n, 5n$ for $p=2, odd$ respectively;
\item $Z(S)=Z(Q_2)$; and
\item $Z_2(S)=Z(Q_1)\cap Q_2$.
\end{enumerate}
\end{proposition}

\begin{proposition}
The following hold:
\begin{enumerate}
    \item $|Q_1|=q^8$;
    \item $|Z(Q_1)|=q^4$ and $Q_1/Z(Q_1)$ is a natural module for $L_1/Q_1\cong \SL_2(q^2)$;
    \item $Z(Q_1)\in\mathcal{A}(S)$ if and only if $p=2$;
    \item $Z(Q_1)=\Phi(Q_1)=[Q_1, Q_1]$ is a natural module $\Omega_4^-(q)$-module for $L_1/Q_1\cong \SL_2(q^2)$;
    \item if $A\le Z(Q_1)$ with $|A|>q^2$ then $C_S(A)=Q_1$;
    \item $Q_2$ is a semi-extraspecial group of order $q^7$; and
    \item $Q_2/Z(S)$ is natural module for $L_2/Q_2\cong \SU_3(q)$.
\end{enumerate}
\end{proposition}

The previous propositions contain all the structural information required to determine a complete list of essentials subgroups of any saturated fusion system $\fs$ supported on $S$. We establish this list over a series of lemmas and propositions, culminating in \cref{PSU5Essen}.

\begin{proposition}\label{Q1Not}
Suppose that $E\in\mathcal{E}(\fs)$ and $E\le Q_1$. Then $E=Q_1$.
\end{proposition}
\begin{proof}
Since $E\le Q_1$ and $E$ is $\fs$-centric, we have that $Z(Q_1)\le E$. Furthermore, as $Z(Q_1)=\Phi(Q_1)$, we have that $Q_1\le N_S(E)$. If $E\le Z(Q_1)(Q_1\cap Q_2)$ then $[S, E]\le Z(Q_1)\le Z(E)$ so that $E\normaleq S$. But now, $Z(Q_1)$ has index at most $q^2$ in $E$ and $Q_1/E$ has order at least $q^2$ and \cref{SEFF} yields $E=Z(Q_1)(Q_1\cap Q_2)$ and $Q_1=N_S(E)$, a contradiction. Hence, $E\not\le Z(Q_1)(Q_1\cap Q_2)$. 

Since $Z(Q_1)\le E$ we have that $Z(Q_1)Q_2<EQ_2$ and applying \cref{SUMod}, we infer that $Z(E)\cap Q_2=Z(Q_1)\cap Q_2$. If $Z(E)>Z(Q_1)$ then again applying \cref{SUMod}, we have that $E\cap Q_2=Z(Q_1)\cap Q_2$ so that $|Q_1/E|\geq q^2$ and $|E/Z(Q_1)|\leq q^2$. In fact, $|E/Z(E)|<q^2$ and applying \cref{SEFF}, we must have that $E=Z(E)$, $|Q_1/E|=q^2=|E/Z(Q_1)|$ and $Q_1=E(Q_1\cap Q_2)$. But now, it follows that $S=EQ_2$, a contradiction for then $\{1\}=[E, E]\not\le Q_2$. Thus, $Z(E)=Z(Q_1)$ so that $Q_1$ centralizes the chain $\{1\}\normaleq Z(E)\normaleq E$ and \cref{Chain} yields $E=Q_1$.
\end{proof}

\begin{lemma}\label{NotCenter}
Suppose that $E\in\mathcal{E}(\fs)$ with $E\ne Q_1$. Then either 
\begin{enumerate}
    \item $\Omega(Z(E))\cap Z(Q_1)=Z(S)$; or
    \item $|\Omega(Z(E))\cap Z(Q_1)|=q^2$.
\end{enumerate}
Moreover, if $\Omega(Z(E))\not\le Q_1$ then $E$ is elementary abelian and $N_S(E)=E(Z(Q_1)\cap Q_2)$.
\end{lemma}
\begin{proof}
Since $Q_1<EQ_1\le C_S(\Omega(Z(E))\cap Z(Q_1))$, applying \cref{Omega4}, we either have that $Z(C_S(\Omega(Z(E))\cap Z(Q_1)))=Z(S)$, or $|Z(C_S(\Omega(Z(E))\cap Z(Q_1)))|=q^2$ and $|EQ_1/Q_1|\leq q$. Since $Z(C_S(\Omega(Z(E))\cap Z(Q_1)))$ centralizes $E$ which is $\fs$-centric, we have that $Z(C_S(\Omega(Z(E))\cap Z(Q_1)))=\Omega(Z(E))\cap Z(Q_1)$.

Assume that $\Omega(Z(E))\not\le Q_1$. Then by \cref{Omega4}, $E\cap Z(Q_1)\le C_{Z(Q_1)}(Z(E))$ has order at most $q^2$. Since $Z(Q_1)\cap Q_2\le N_S(E)$ and $[Z(Q_1)\cap Q_2, E\cap Q_1]=\{1\}$, \cref{Omega4} and \cref{SEFF} yield $|E\cap Z(Q_1)|=q^2$, $N_S(E)=E(Z(Q_1)\cap Q_2)$, $|N_S(E)/E|=q$, $EQ_1/Q_1=\Omega(Z(E))Q_1/Q_1$ has order $q$ and $\Phi(E)\cap Z(S)=\{1\}$.

Now, $E\cap Z(Q_1)$ is centralized by $\Omega(Z(E))Q_1\ge E$ so that $E\cap Z(Q_1)=\Omega(Z(E))\cap Z(Q_1)$. Furthermore, since $\Omega(Z(E))\not\le Q_1$, we deduce that $E\cap Q_1\le Z(Q_1)(Q_1\cap Q_2)$ so that $E=\Omega(Z(E))(E\cap Z(Q_1)(Q_1\cap Q_2))$ from which it follows that $[E, E]\le [Q_1\cap Q_2, Q_1\cap Q_2]\cap \Phi(E)=Z(S)\cap \Phi(E)=\{1\}$ and $E$ is abelian. Furthermore, note that $\mho^1(E)=\mho^1(E\cap Z(Q_1)(Q_1\cap Q_2))\le Z(S)\cap \Phi(E)=\{1\}$. Hence, $\Phi(E)=\{1\}$ and $E$ is elementary abelian.
\end{proof}

With the ultimate aim of demonstrating that $\mathcal{E}(\fs)\subseteq \{Q_1, Q_2\}$ we assume that whenever $E\ne Q_1$ and $E\not\le Q_2$, $E$ is maximally essential. Thus, we are free to use \cref{MaxEssen} - \cref{MaxEssenOdd} throughout.

\begin{proposition}
Suppose that $E\in\mathcal{E}(\fs)$ with $E\ne Q_1$. Then $\Omega(Z(E))\le Q_1$, $Z(Q_1)\cap Q_2\le E$ and $Z(Q_1)(Q_1\cap Q_2)\le N_S(E)$.
\end{proposition}
\begin{proof}
Aiming for a contradiction, assume that $\Omega(Z(E))\not\le Q_1$ so that by \cref{NotCenter}, $E$ is elementary abelian and $N_S(E)=E(Z(Q_1)\cap Q_2)$. If $E\le Q_2$, then as $E$ is elementary abelian and $Q_2$ is semi-extraspecial, $|E|\leq q^4$ and since $[Q_2, E]\le [Q_2, Q_2]=Z(S)\le E$, $Q_2\le N_S(E)$ and $q=|N_S(E)/E|\geq q^3$, a contradiction. Hence, $E$ is maximally essential in $\fs$. If $E\le Z(Q_1)Q_2$ then $[Q_1\cap Q_2, E]\le [Q_1\cap Q_2, Z(Q_1)Q_2]=Z(S)\le E$ so that $Q_1\cap Q_2\le N_S(E)$. Indeed, since $|N_S(E)/E|=q$, we must have that $|E\cap Q_1|= q^4$. But then $(E\cap Q_1)Z(Q_1)=q^6$ and $(E\cap Q_1)Z(Q_1)$ is elementary abelian, a contradiction since $m_p(S)\leq 5n$. Hence, $E\not\le Z(Q_1)Q_2$ and since $E$ is elementary abelian and $\Omega(S/Q_2)=Z(Q_1)Q_2/Q_2$ when $p=2$, we deduce that $p$ is odd. In the application of \cref{SEFF}, and using that $E$ is abelian, we have that $O^{p'}(\Aut_{\fs}(E))\cong \SL_2(q)$ and $E=[E, O^{p'}(\Aut_{\fs}(E))]\times C_{E}(O^{p'}(\Aut_{\fs}(E)))$ where $[E, O^{p'}(\Aut_{\fs}(E))]$ is a natural $\SL_2(q)$-module.

Let $\hat{t}$ be a non-trivial involution in $Z(O^{p'}(\Aut_{\fs}(E)))$. Then, as $E$ is receptive, $\hat{t}$ lifts to subgroup of $S$ strictly containing $E$, and by the Alperin--Goldschmidt theorem and using that $E$ is maximally essential, we deduce that $\hat{t}$ is the restriction to $E$ of some element $t\in\Aut_{\fs}(S)$. In particular, $t$ normalizes $Q_1$ and $Q_2$. Then $t$ acts isomorphically on $EQ_1/Q_1\cong E/E\cap Q_1=E/C_E(\Aut_S(E))$. So $t$ inverts $EQ_1/Q_1$. By similar arguments, $t$ inverts $Z(S)$ and acts trivially on $(Z(Q_1)\cap Q_2)/Z(S)$. Let $Z_t=C_{Z(Q_1)}(t)$. Then $[t, Z_t, E]=\{1\}$ and $[E, Z_t, t]\le [Z(Q_1)\cap Q_2, t]=Z(S)$. Applying the three subgroups lemma, we have that $[E, t, Z_t]\le Z(S)$ and since $[E, t]\not\le Q_1$, we deduce that $Z_t\le Z(Q_1)\cap Q_2$ and $t$ inverts $Z(Q_1)/(Z(Q_1)\cap Q_2)$. 

Replacing $Z_t$ by $C_t$ in the previous argument, where $C_t$ is the preimage in $Q_1\cap Q_2$ of $C_{(Q_1\cap Q_2)/Z(S)}(t)$, since $t$ inverts $Z(Q_1)Q_2/Q_2=\Phi(S/Q_2)$, $t$ must act non-trivially on $S/Z(Q_1)Q_2$ and it follows that $t$ inverts $(Q_1\cap Q_2)/Z(Q_1)\cap Q_2$. Furthermore, $t$ acts on $S/Z(Q_1)Q_2$ as it acts on $Q_1/Z(Q_1)(Q_1\cap Q_2)$. Hence, $[Q_1/Z(Q_1), t]>(Q_1\cap Q_2)/Z(Q_1)$. Then for $D$ the preimage in $S/Z(Q_1)$ of $[S/Z(Q_1)(Q_1\cap Q_2), t]$, $t$ inverts $D$ so that $D$ is abelian. Since $t$ inverts $EQ_1/Q_1$, $D\not\le Q_1/Z(Q_1)$. But $C_{Q_1/Z(Q_1)}(s)=Z(Q_1)(Q_1\cap Q_2)/Z(Q_1)$ for any $s\in S\setminus Q_1$ since $Q_1/Z(Q_1)$ is a natural $\SL_2(q^2)$-module for $\Out_{\PSU_5(q)}(Q_1)\cong \SL_2(q^2)$, and since $D$ is abelian we deduce that $(D\cap Q_1)/Z(Q_1)=Z(Q_1)(Q_1\cap Q_2)/Z(Q_1)$ and we have a contradiction.

Hence, $\Omega(Z(E))\le Q_1$. But now, $[Z(Q_1)\cap Q_2, E]\le [Z(Q_1)\cap Q_2, S]\le Z(S)\le \Omega(Z(E))$ so that $Z(Q_1)\cap Q_2$ centralizes the chain $\{1\}\normaleq \Omega(Z(E))\normaleq E$ and by \cref{Chain}, $Z(Q_1)\cap Q_2\le E$. In particular, $Z(Q_1)(Q_1\cap Q_2)\le N_S(E)$.
\end{proof}

For the following proposition, we recall the definition of $W(S)$ and its properties, as in \cref{WS}.

\begin{proposition}\label{CenterContain}
Suppose that $E\in\mathcal{E}(\fs)$. Then $\Omega(Z(E))\le Z(Q_1)$.
\end{proposition}
\begin{proof}
Aiming for a contradiction, assume that $\Omega(Z(E))\not\le Z(Q_1)$. In particular, $E\ne Q_1$ so that $E\not\le Q_1$ by \cref{Q1Not}. Since $\Omega(Z(E))\le Q_1$ and $E\not\le Q_1$, we must have that $\Omega(Z(E))\le Z(Q_1)(Q_1\cap Q_2)$. Furthermore, $\Omega(Z(E))Z(Q_1)$ is elementary abelian. By \cite[Theorem 3.3.3]{GLS3}, when $p=2$ we have that $m_2(S)=4n$ and when $p$ is odd, $m_p(S)=5n$. Hence, we have a contradiction when $p=2$ and when $p$ is odd, we deduce that $|\Omega(Z(E))Z(Q_1)|\leq q^5$. We let $p$ be an odd prime for the remainder of the proof. Note that $C_S(Z(Q_1)(Q_1\cap Q_2))=Z(Q_1)$ and since $\Omega(Z(E))\le Q_1$, we must have that $Q_1\cap Q_2\not\le E$. 

Assume that $Z(Q_1)\cap E>Z(Q_1)\cap Q_2$. Then for any $x\in E$ such that $[Z(Q_1)\cap E, x, x]=\{1\}$, we must have that $x\in E\cap Q_1$ and $[Z(Q_1)\cap E, x]=\{1\}$. Hence, $Z(Q_1)\cap E\le W(E)$. Since $W(E)$ is itself elementary abelian, it acts quadratically on $Z(Q_1)\cap E$ and so we deduce that $W(E)\le Q_1$. Hence, $Z(Q_1)$ centralizes the chain $\{1\}\normaleq W(E)\normaleq E$ and by \cref{Chain}, $Z(Q_1)\le E$. Note that $[W(E), E\cap Q_1]\le Z(Q_1)$ so that $[W(E), E\cap Q_1, E\cap Q_1]=\{1\}$ and $W(E)$ centralizes $E\cap Q_1$. If $W(E)\le Z(Q_1)(Q_1\cap Q_2)$, then as $\Omega(Z(E))\le W(E)$, we have that $Z(Q_1)<Z(Q_1)\Omega(Z(E))\le W(E)\le Z(Q_1)(Q_1\cap Q_2)$. Moreover, $[E, Z(Q_1)]\le [E, W(E)]\le Z(Q_1)$. But then, $[E, Q_1\cap Q_2]\le Z(Q_1)\cap Q_2\le W(E)$, $[W(E), Q_1\cap Q_2]\le Z(S)\le [W(E), E]$ and $Q_1\cap Q_2$ centralizes $[W(E), E]$ so that $Q_1\cap Q_2$ centralizes the chain $\{1\}\normaleq [W(E), E]\normaleq W(E)\normaleq E$, a contradiction by \cref{Chain}. Hence, since $W(E)\le Q_1$, $W(E)\not\le Z(Q_1)Q_2$ and $W(E)\cap Q_2=C_{Q_2}(W(E))=Z(Q_1)\cap Q_2$. But $\Omega(Z(E))Z(Q_1)\le W(E)$ and so $Z(Q_1)\cap Q_2<\Omega(Z(E))Z(Q_1)\cap Q_2\le W(E)\cap Q_2=Z(Q_1)\cap Q_2$, a contradiction.

Hence, if $\Omega(Z(E))\not\le Z(Q_1)$ then $p$ is odd and $Z(Q_1)\cap E=Z(Q_1)\cap Q_2$. Since $E\not\le Q_1$, we have that $Z(S)=[Z(Q_1)\cap Q_2, E]\le \Phi(E)$. If $E\le Z(Q_1)Q_2$ then $[Q_1\cap Q_2, E]\le Z(S)\le \Phi(E)$ and we get that $Q_1\cap Q_2\le E$ by \cref{burnside}, a contradiction. Now, $E\cap Q_1$ centralizes $Z(E)Z(Q_1)>Z(Q_1)$ and so centralizes $Z(E)Z(Q_1)\cap Q_2>Z(Q_1)\cap Q_2$ and by \cref{SUMod}, we have that $E\cap Q_1\le Z(Q_1)(Q_1\cap Q_2)$. Then, $Z(Q_1)(Q_1\cap Q_2)$ centralizes $(E\cap Q_1)/\Phi(E)$ which has index at most $q^2$ in $E/\Phi(E)$ and by \cref{SEFF}, $Z(Q_1)(Q_1\cap Q_2)\cap E$ has index at most $q^2$ in $Z(Q_1)(Q_1\cap Q_2)$. Indeed, since $Z(Q_1)\cap E=Z(Q_1)\cap Q_2$, we deduce that $|E\cap Q_1\cap Q_2|\geq q^4$. But then, for any $e\in E\setminus Z(Q_1)Q_2$, $|[e, E\cap Q_1\cap Q_2]Z(S)|\geq q^2$ and so $\Phi(E)\cap Z(Q_1)$ has order at least $q^2$. Again, applying \cref{SEFF}, we have this time that $Z(Q_1)(Q_1\cap Q_2)\cap E$ has index at most $q$ in $Z(Q_1)(Q_1\cap Q_2)$ and $|N_S(E)/E|\leq q$. Since $Z(Q_1)\cap E=Z(Q_1)\cap Q_2$, $N_S(E)=Z(Q_1)E$ so that $N_S(E)$ centralizes $\Omega(Z(E))$. But now, $Z(Q_1)(Q_1\cap Q_2)\le N_S(E)$ and we conclude that $\Omega(Z(E))\le Z(Q_1)$, a final contradiction.
\end{proof}

\begin{lemma}\label{CenterContain2}
Suppose that $E\in\mathcal{E}(\fs)$ and $E\ne Q_1$. If $\Omega(Z(E))\ne Z(S)$ then $p$ is odd and $Z(Q_1)\cap E=Z(Q_1)\cap Q_2$.
\end{lemma}
\begin{proof}
Suppose throughout that $|\Omega(Z(E))|=q^2$. Assume first that $p=2$. Then $J(Q_1)=Z(Q_1)$ and since $\Omega(C_{Q_1}(J(Q_1)))\le J(Q_1)$, we have that $\Omega(Q_1)=Z(Q_1)$. Now, $[Z(Q_1), E]=\Omega(Z(E))$ so that by \cref{Chain}, $Z(Q_1)\le E$. Let $A\in\mathcal{A}(E)$. Then $A\cap Q_1\le \Omega(C_{Q_1}(J(Q_1)))=Z(Q_1)$. But now, $A$ centralizes a subgroup of $Z(Q_1)$ of order at least $q^3$ and by \cref{Omega4}, we have that $A\le Q_1$ and $A=Z(Q_1)$. Hence, $J(E)=Z(Q_1)$. Since $Q_1\cap Q_2$ centralizes the chain $\{1\}\normaleq J(E)\normaleq E$ by \cref{Chain}, $Q_1\cap Q_2\le E$ and $E\normaleq S$. Then $Q_1$ centralizes the chain $\{1\}\normaleq J(E)\normaleq C_E(J(E))=E\cap Q_1\normaleq E$ and $Q_1\le E$. Furthermore, $Q_1$ is a characteristic subgroup of $E$ so that $Z(Q_1)(Q_1\cap Q_2)=J(E)[E, Q_1]$ is also characteristic in $E$. Then $S$ centralizes the chain $J(E)\normaleq J(E)[E, Q_1]\normaleq Q_1\normaleq E$ and since $J(E)=\Phi(Q_1)\le \Phi(E)$, we have a contradiction by \cref{Chain}.

Hence, if $|\Omega(Z(E))|=q^2$, then $p$ is odd. Aiming for a contradiction, assume that $Z(Q_1)\cap E>Z(Q_1)\cap Q_2$. If $[Z(Q_1)\cap E, x, x]=\{1\}$ for some $x\in E$ then by \cref{Omega4} we have that $x\in Q_1$ so that $[Z(Q_1)\cap E, x]=\{1\}$. Hence, $Z(Q_1)\cap E\le W(E)$ and as $W(E)$ is elementary abelian, we deduce that $W(E)\le E\cap Q_1$. Then $Z(Q_1)$ centralizes the chain $\{1\}\normaleq W(E)\normaleq E$ and we have that $Z(Q_1)\le W(E)$. If $W(E)\not\le Z(Q_1)(Q_1\cap Q_2)$ then by \cref{SUMod}, we deduce that $W(E)\cap Q_2=C_{Q_2}(W(E))=Z(Q_1)\cap Q_2$ and since $E\cap Q_1\cap Q_2$ acts quadratically on $W(E)$, we infer that $E\cap Q_1\cap Q_2=Z(Q_1)\cap Q_2$. Then $[E, Q_1\cap Q_2]\le W(E)$ and $Q_1\cap Q_2$ centralizes $Z(Q_1)$ which has index at most $q$ in $W(E)$ and \cref{SEFF} gives a contradiction. 

Therefore, $Z(Q_1)\le W(E)\le Z(Q_1)(Q_1\cap Q_2)$ so that $Q_1\cap Q_2$ centralizes the chain $1\normaleq Z(E)\normaleq W(E)\normaleq E$, $Z(Q_1)(Q_1\cap Q_2)\le E$ and $E\normaleq S$. Now, $Z(Q_1)(Q_1\cap Q_2)$ is generated by elementary abelian subgroups of order $q^5$ so that $Z(Q_1)(Q_1\cap Q_2)\le J(E)$. Note that for any $A\in\mathcal{A}(S)$, $A\not\le Q_1$ yields that $|(A\cap Q_1)Z(Q_1)|=|A\cap Q_1||Z(Q_1)|/|A\cap Z(Q_1)|\geq |A||Z(Q_1)|/q^3>|A|$ so that $J(S)=Q_1$ and $J(E)\le E\cap Q_1$. Since $Z(Q_1)=Z(Z(Q_1)(Q_1\cap Q_2))$, we have that $Z(Q_1)=Z(J(E))$ is characteristic in $E$, $E\cap Q_1=C_{E}(Q_1)$ is characteristic in $E$, and $Q_1$ centralizes the chain $\{1\}\normaleq Z(Q_1)\normaleq E\cap Q_1\normaleq E$ so that $Q_1$ is characteristic in $E$. Then $Z(Q_1)\cap Q_2=Z(E)[Z(Q_1), E]$, $Z(S)=[Z(Q_1)\cap Q_2, E]$ and $Z(Q_1)(Q_1\cap Q_2)=[E, Q_1]Z(J(E))$ are all characteristic in $E$ and $S$ centralizes the chain $\{1\}\normaleq Z(S)\normaleq Z(Q_1)\cap Q_2\normaleq Z(Q_1)(Q_1\cap Q_2)\normaleq E$, a contradiction.
\end{proof}

\begin{proposition}\label{U5Step}
Suppose that $E\in\mathcal{E}(\fs)$. Then $\Omega(Z(E))=Z(S)$ or $E=Q_1$.
\end{proposition}
\begin{proof}
Throughout, we assume that $\Omega(Z(E))\ne Z(S)$ and $E\not\le Q_1$. By \cref{CenterContain2}, we have that$|\Omega(Z(E))|=q^2$, $p$ is odd and $Z(Q_1)\cap E=Z(Q_1)\cap Q_2$. Then $Z(Q_1)$ centralizes $E\cap Q_1$ of index $q$ in $E$ so that by \cref{SEFF}, $N_S(E)=Z(Q_1)E$ and $O^{p'}(\Out_{\fs}(E))\cong \SL_2(q)$. If $Q_1\cap Q_2\le E$, then $[Q_2, E]\le Q_1\cap Q_2\le E$ and $Q_2\le N_S(E)$. But $Q_2\le N_S(E)=EZ(Q_1)$ centralizes $\Omega(Z(E))>\Omega(Z(S))$, a contradiction. Hence, $Q_1\cap Q_2\not\le E$. Similarly, if $E\le Q_2$ then $Q_2\le N_S(E)$ and we obtain a similar contradiction. Thus, we have that $E$ is maximally essential and so, using that $E$ is receptive and applying the Alperin--Goldschmidt theorem, any $p'$-automorphisms in $N_{\Aut_{\fs}(E)}(\Aut_S(E))$ lift to morphisms in $\Aut_{\fs}(S)$. Upon restriction to $E$, all such morphisms normalize $\Aut_{Q_1\cap Q_2}(E)$ and as $O^{p'}(\Out_{\fs}(E))\cong \SL_2(q)$, $N_S(E)=E(Q_1\cap Q_2)$.

Let $\hat{t}$ be a non-trivial involution in $\Aut_{\fs}(E)$ with image in $Z(O^{p'}(\Out_{\fs}(E)))$. Then, $\hat{t}$ is the restriction to $E$ of some element $t\in\Aut_{\fs}(S)$. Since $N_S(E)=Z(Q_1)E$ centralizes $\Omega(Z(E))$, we deduce that $O^{p'}(\Aut_{\fs}(E))$ acts trivially on $\Omega(Z(E))$ and so $t$ acts trivially on $\Omega(Z(E))$. Then $[Z(Q_1), E]\Omega(Z(E))/\Omega(Z(E))=(Z(Q_1)\cap Q_2)\Omega(Z(E))/\Omega(Z(E))$ is of order $q$ inverted by $t$. For a similar reason, $E/E\cap Q_1$ is inverted by $t$ and $(E\cap Q_1)/(Z(Q_1)\cap Q_2)$ is centralized by $t$. Furthermore, $t$ centralizes $\Aut_S(E)\cong N_S(E)/E\cong Z(Q_1)/Z(Q_1)\cap E$ and we infer that $t$ centralizes $Z(Q_1)(E\cap Q_1)/Z(Q_1)\cap Q_2=(N_S(E)\cap Q_1)/Z(Q_1)\cap Q_2$. Let $\mathcal{C}$ be the preimage in $Q_1$ of $C_{Q_1/Z(Q_1)}(t)$. Then $[\mathcal{C}, t, \mathcal{C}]=\{1\}$ and by the three subgroup lemma, $t$ centralizes $[\mathcal{C}, \mathcal{C}]$. But $Q_1\cap Q_2\le \mathcal{C}$ and $[t, Z(Q_1)\cap Q_2]\ne\{1\}$ and so by \cref{SUMod}, we deduce that $\mathcal{C}\le Z(Q_1)(Q_1\cap Q_2)$ so that $Q_1/Z(Q_1)(Q_1\cap Q_2)$ is inverted by $t$. But then, $[t, Q_1\cap Q_2, Q_1]=\{1\}$, $[Q_1, Q_1\cap Q_2, t]=[Z(Q_1)\cap Q_2, t]$ has order $q$ and is normalized by $Q_1, Q_1\cap Q_2$ and $t$, and by the three subgroup lemma, $[Q_1, t, Q_1\cap Q_2]$ has order $q$. But $[Q_1, t]\not\le Z(Q_1)(Q_1\cap Q_2)$ and we have a contradiction by \cref{SUMod}.
\end{proof}

As promised, we now complete the determination of all possible essential subgroups of $\fs$.

\begin{proposition}\label{PSU5Essen}
Suppose that $E\in\mathcal{E}(\fs)$. Then $E\in\{Q_1, Q_2\}$.
\end{proposition}
\begin{proof}
Ultimately aiming for a contradiction, we assume that $Q_1\ne E\ne Q_2$. Since $\Omega(Z(E))=Z(S)$ by \cref{U5Step}, if $E\le Q_2$ then \cref{Chain} implies that $E=Q_2$ and so for the remainder of the proof, we assume that $Q_1\not\ge E\not\le Q_2$. Let $Z^E$ be the preimage in $E$ of $Z(E/\Omega(Z(E)))$, a characteristic subgroup of $E$. Then $Z(Q_1)\cap Q_2\le Z^E$ and $C_E(Z^E)\le E\cap Q_1$

Assume that $Z^E\not\le Z(Q_1)Q_2$ so that $E\cap Q_2=Z(Q_1)\cap Q_2$ by \cref{SUMod}. Since $Q_1\cap Q_2\le N_S(E)$ and $Q_1\cap Q_2$ centralizes $E\cap Z(Q_1)Q_2$ modulo $\Omega(Z(E))$, applying \cref{SEFF} we have that $S=EZ(Q_1)Q_2$. But then $\Omega(Z(E))\ge [E, Z^E]\not\le Q_2$, a contradiction.

Hence, $Z^E\le Z(Q_1)Q_2$. Then $[Z^E, Q_1\cap Q_2]\le \Omega(Z(E))$ and $Q_1\cap Q_2$ centralizes the chain $\{1\}\normaleq \Omega(Z(E))\normaleq Z^E\normaleq E$ from which we deduce that $Q_1\cap Q_2\le E$. Furthermore, $Q_2$ normalizes $E$. If $Z^E\not\le Q_2$, then $E\cap Q_2=Q_1\cap Q_2$ by \cref{SUMod} so that $|Q_2E/E|=q^2$. We apply \cref{MaxEssenEven} when $p=2$ and use $[E, Q_2, Q_2]\le \Omega(Z(E))$ and \cref{MaxEssenOdd} when $p$ is odd, so that $O^{p'}(\Out_{\fs}(E))$ is quasisimple. But now, $Z^E\cap Q_2$ has index at most $q$ in $Z^E$ and is centralized modulo $\Omega(Z(E))$ by $Q_2$ so that, by \cref{SEFF}, $O^{p'}(\Out_{\fs}(E))$ acts trivially on $Z^E$. Then, it follows that $Q_2E$ centralizes $Z^E/\Omega(Z(E))$ and \cref{SUMod} yields a contradiction.

Thus, $Z^E\le Q_2$ and since $E\not\le Q_2$, we have that $Z^E\le Q_1\cap Q_2$ so that $Z(Q_1)\cap E\le C_E(Z^E)\le E\cap Q_1$. Then $Z(Q_1)$ centralizes the chain $\{1\}\normaleq C_E(Z^E)\normaleq E$ and we deduce that $Z(Q_1)\le E$ and $E\normaleq S$. Now, if $E\le Z(Q_1)Q_2$, then $Z^E=Q_1\cap Q_2$ and $Z(Z^E)=Z(Q_1)\cap Q_2$ is characteristic in $E$. But then, $Q_1$ centralizes the chain $\{1\}\normaleq Z(Q_1)\cap Q_2\normaleq Q_1\cap Q_2\normaleq E$ so that $Q_1\le E$ by \cref{Chain}, a contradiction since $E\le Z(Q_1)Q_2$. Hence, $E\not\le Z(Q_1)Q_2$ and we conclude that $Z^E=Z(Q_1)\cap Q_2$ and $C_E(Z^E)=E\cap Q_1$. Since $Z(Q_1)(Q_1\cap Q_2)\le E\cap Q_1$ and $Z(Q_1)=Z(Z(Q_1)(Q_1\cap Q_2))$, we have that $Z(Q_1)=Z(C_E(Z^E))$ is characteristic in $E$ and $Q_1$ centralizes the chain $\{1\}\normaleq Z(Q_1)\normaleq E\cap Q_1\normaleq E$. Hence, $Q_1=E\cap Q_1$ is characteristic in $E$. Then $E\not\le Q_1$ so that $Z(Q_1)[E, Q_1]=Z(Q_1)(Q_1\cap Q_2)$ is also characteristic in $E$. Finally, $S$ centralizes the chain \[\{1\}\normaleq \Omega(Z(E))\normaleq Z^E\normaleq Z(Q_1)\normaleq Z(Q_1)[E, Q_1]\normaleq Q_1\normaleq E\] and \cref{Chain} provides a final contradiction.
\end{proof}

In the classification of saturated fusion systems supported on $S$, we apply \cref{PSUCor} using \cref{MainThm} when $Q_1$ and $Q_2$ are both essential and, as in earlier cases, we remark that this reduces to applying the main result from \cite{Greenbook}. Moreover, we are able to calculate the local actions in the fusion system without any reliance on the classification of the finite simple groups. In this way, the following theorem independent of any $\mathcal{K}$-group hypothesis.

\begin{theorem}\label{PSU5}
Let $\fs$ be a saturated fusion system on a Sylow $p$-subgroup of $\mathrm{PSU}_5(p^n)$. Then either:
\begin{enumerate}
\item $\fs=\fs_S(S: \Out_{\fs}(S))$;
\item $\fs=\fs_S(Q_1: \Out_{\fs}(Q_1))$ where $O^{p'}(\Out_{\fs}(Q_1))\cong \SL_2(p^n)$;
\item $\fs=\fs_S(Q_2: \Out_{\fs}(Q_2))$ where $O^{p'}(\Out_{\fs}(Q_2))\cong \mathrm{(P)SU}_3(p^n)$; or 
\item $\fs=\fs_S(G)$ where $F^*(G)=O^{p'}(G)=\mathrm{PSU}_5(p^n)$.
\end{enumerate}
\end{theorem}
\begin{proof}
If $\mathcal{E}(\fs)=\emptyset$ then outcome (i) holds. If $Q_1\in\mathcal{E}(\fs)$ then $S$ centralizes $Z(Q_1)(Q_1\cap Q_2)$ modulo $Z(Q_1)=\Phi(Q_1)$ and \cref{SEFF} yields that $Q_1/\Phi(Q_1)$ is a natural module for $O^{p'}(\Out_{\fs}(Q_1))\cong \SL_2(q^2)$. Then, as $\Out_S(Q_1)$ acts non-trivially on $Z(Q_1)$, we have that $O^{p'}(\Out_{\fs}(Q_1))$ acts non-trivially and comparing with \cref{SL2ModRecog}, using that $\Out_S(Q_1)$ does not act quadratically on $Z(Q_1)$, we deduce that $Z(Q_1)$ is a natural $\Omega_4^-(q)$-module for $O^{p'}(\Out_{\fs}(Q_1))\cong \SL_2(q^2)$, on which $Z(O^{p'}(\Out_{\fs}(Q_1)))$ acts trivially. If $Q_2\not\in\mathcal{E}(\fs)$ then outcome (ii) holds.

Hence we may assume that $Q_2\in\mathcal{E}(\fs)$. Suppose first that $p=2$. We apply \cref{MaxEssenEven} and use that $S/Q_2$ is isomorphic to a Sylow $2$-subgroup of $\PSU_3(q)$ to deduce that either $O^{2'}(\Out_{\fs}(Q_2))\cong\mathrm{(P)SU}_3(q)$, or $q=2$ and $\Out_S(Q_2)\cong Q_8$. In the latter case, we appeal to MAGMA for a list of subgroups of $\GL_6(2)$ which have a Sylow $2$-subgroup isomorphic to $Q_8$ which acts in the same manner as $\Out_S(Q_2)$ does on a group of order $2^6$. Then $O^{2'}(\Out_{\fs}(Q_2))\cong\SU_3(2)$ in this case too.

Suppose now that $p$ is odd. Then $Z(Q_1)$ acts quadratically on $Q_2$ and so $\Out_{Z(Q_1)}(Q_2)$ is quadratic on $Q_2$. With the aim of applying \cref{MaxEssenOdd} we need to show that $\syl_p(\Out_{\fs}(Q_2))$ is a TI-set for $\Out_{\fs}(Q_2)$. To this end, aiming for a contradiction, assume there is $1\ne s\in T\cap R$, where $T$ and $R$ are distinct Sylow $p$-subgroups of $\Out_{\fs}(Q_2)$. By \cref{SUMod}, $|C_{Q_2/\Phi(Q_2)}(s)|\in\{q^2, q^4\}$. 

If $|C_{Q_2/\Phi(Q_2)}(s)|=q^4$ then $s\in Z(T)\cap Z(R)$ by \cref{SUMod}. Then \[C_{Q_2/\Phi(Q_2)}(Z(T))=C_{Q_2/\Phi(Q_2)}(s)=C_{Q_2/\Phi(Q_2)}(Z(R))\] and \[[Q_2/\Phi(Q_2), Z(T)]=[Q_2/\Phi(Q_2), s]=[Q_2/\Phi(Q_2), Z(R)].\] Thus, by \cref{SUMod}, $\langle T, R\rangle$ centralizes the chain \[\{1\}\normaleq [Q_2/\Phi(Q_2), s]\normaleq C_{Q_2/\Phi(Q_2)}(s)\normaleq Q_2/\Phi(Q_2)\] and \cref{Chain} yields that $\langle T, R\rangle$ is a $p$-group. Since $T$ and $R$ are distinct Sylow $p$-subgroups, this is a contradiction. 

If $|C_{Q_2/\Phi(Q_2)}(s)|=q^2$ then $s\in (T\setminus Z(T))\cap (R\setminus Z(R))$ and \[[Q_2/\Phi(Q_2), T]=[Q_2/\Phi(Q_2), s]=[Q_2/\Phi(Q_2), R]\] by \cref{SUMod}. Then \[[Q_2/\Phi(Q_2), T, T]=C_{Q_2/\Phi(Q_2)}(s)=[Q_2/\Phi(Q_2), R, R]\] and $\langle T,R\rangle$ centralizes the chain \[\{1\}\normaleq C_{Q_2/\Phi(Q_2)}(s)\normaleq [Q_2/\Phi(Q_2), s]\normaleq Q_2/\Phi(Q_2)\] and \cref{Chain} yields that $\langle T, R\rangle$ is a $p$-group. Since $T$ and $R$ are distinct Sylow $p$-subgroups, this is again a contradiction. Thus, $\syl_p(\Out_{\fs}(Q_2))$ is a TI-set and as $\Out_S(Q_2)\cong S/Q_2$ is non-abelian, we conclude from \cref{MaxEssenOdd} that $O^{p'}(\Out_{\fs}(Q_2))\cong\mathrm{(P)SU}_3(q)$. If $\mathcal{E}(\fs)=\{Q_2\}$ then outcome (iii) holds.

To complete the proof, we are left with the case $\mathcal{E}(\fs)=\{Q_1, Q_2\}$. By \cite[Proposition I.4.5]{ako}, $O_p(\fs)\le Q_1\cap Q_2$. Since $O^{p'}(\Out_{\fs}(Q_2))$ acts irreducibly on $Q_2/Z(S)$, we have that $O_p(\fs)\le Z(S)$. But now, $O^{p'}(\Out_{\fs}(Q_1))$ acts irreducibly on $Z(Q_1)$, and we have that $O_p(\fs)=\{1\}$. Then, $\fs$ satisfies the hypothesis of \cref{PSUCor}, and (iv) holds, completing the proof.
\end{proof}

\section[Fusion Systems on a Sylow \texorpdfstring{$p$}{p}-subgroup of \texorpdfstring{${}^3\mathrm{D}_4(p^n)$}{3D4(pn)}]{Fusion Systems on a Sylow $p$-subgroup of ${}^3\mathrm{D}_4(p^n)$}

In this final section, we classify saturated fusion systems supported on a Sylow $p$-subgroup of ${}^3\mathrm{D}_4(p^n)$, completing the proof of \hyperlink{thm2}{Theorem A} and of the \hyperlink{MainThm}{Main Theorem}. Throughout we fix $S$ isomorphic to a Sylow $p$-subgroup of ${}^3\mathrm{D}_4(p^n)$ and write $q=p^n$.

For $G:={}^3\mathrm{D}_4(p^n)$, we identify $S\in\syl_p(G)$ and set $G_1:=N_G(C_S(Z_2(S)))$ and $G_2:=N_G(Q)$, where $Q$ is the preimage in $S$ of $J(S/Z(S))$ of order $q^9$. Let $L_i:=O^{p'}(G_i)$ for $i\in\{1,2\}$ so that $L_1$ is of shape $L_1\cong q^{2+(1+1+1)+(2+2+2)}:\SL_2(q)$ and $L_2$ is of shape $q^{1+8}:\SL_2(q^3)$. Note that $G_1$ and $G_2$ are the unique maximal parabolic subgroups of $G$ containing $N_G(S)$. We write $Q_i:=O_p(L_i)=O_p(G_i)$ so that $Q_1=C_S(Z_2(S))$ and $Q=Q_2$, characteristic subgroups of $S$.

We appeal to a combination of the techniques in \cref{PrelimSec}, and \cite{3D4Par} and \cite{Seitz} for structure of a chief series of $L_1$ and $L_2$, and the structure of certain subgroups of $S$. We highlight the main features required for this work, which may be calculated from this information.

\begin{proposition}\label{D4Basic}
The following hold:
\begin{enumerate}
\item $|S|=q^{12}$;
\item $S'=\Phi(S)=Q_1\cap Q_2$;
\item $m_p(S)=5n$;
\item $Z(S)=Z(Q_2)$ has order $q$;
\item $Z_2(S)=Z(Q_1)$ is elementary abelian of order $q^2$;
\item $\Phi(Q_1)=Z_3(S)$ is elementary abelian of order $q^5$;
\item $C_S(\Phi(Q_1)/Z(S))=Q_2$; and
\item if $p=2$ and $x\in S$ is such that $x^2=1$, then $x\in Q_1 \cup Q_2$.
\end{enumerate}
\end{proposition}

\begin{proposition}
The following hold:
\begin{enumerate}
    \item $|Q_1|=q^{11}$;
    \item $|\Phi(Q_1)|=q^5$, $\Phi(Q_1)\in\mathcal{A}(S)$ and $Q_1/\Phi(Q_1)$ is a direct sum of three natural modules for $L_1/Q_1\cong \SL_2(q)$;
    \item $|Z(Q_1)|=q^2$, $\Phi(Q_1)/Z(Q_1)$ is centralized by $L_1$ and $Z(Q_1)$ is a natural module for $L_1/Q_1\cong \SL_2(q)$;
    \item $\Phi(Q_1)=[Q_1, Q_1]$ and $Z(Q_1)=[\Phi(Q_1), Q_1]$;
    \item $Q_2$ is a semi-extraspecial group of order $q^9$; and
    \item $Q_2/Z(S)$ is a triality module for $L_2/Q_2\cong \SL_2(q^3)$.
\end{enumerate}
\end{proposition}

We first determine the structure of centralizers of certain $p$-subgroups of $S$ to facilitate later arguments.

\begin{lemma}\label{Phi1Cent}
Suppose that $A\le Q_1$ and let $b\in Q_1\setminus \Phi(Q_1)$. Then $|C_{\Phi(Q_1)}(A)|\in\{q^2, q^3, q^4, q^5\}$, $|C_{\Phi(Q_1)}(b)|=q^4$ and $|[b, \Phi(Q_1)]|=q$. Moreover, if $|C_{\Phi(Q_1)}(A)|=q^2$ then $C_{\Phi(Q_1)}(A)=Z(Q_1)$ and if $|C_{\Phi(Q_1)}(A)|=q^5$ then $A\le \Phi(Q_1)$.
\end{lemma}
\begin{proof}
Let $b\in Q_1\setminus \Phi(Q_1)$. Then there is $x\in L_1$ such that $b^x\le Q_1\cap Q_2$. Moreover, since $Z(Q_1)\le C_{\Phi(Q_1)}(b)\cap C_{\Phi(Q_1}(b^x)$ and $L_1$ acts trivially on $\Phi(Q_1)/Z(Q_1)$, we deduce that $C_{\Phi(Q_1)}(b)=C_{\Phi(Q_1)}(b^x)$. Now, the map $\phi_{b^x}: \Phi(Q_1)\to \Phi(Q_1)$ with $q\phi_{b^x}=[q, b^x]$ is a homomorphism with kernel $C_{\Phi(Q_1)}(b^x)$. Since $\Phi(Q_1)\in\mathcal{A}(S)$ and $b^x\le Q_2\setminus \Phi(Q_1)$, we have by \cref{Ultraspecial} that $\mathrm{Im}(\phi_{b^x})=[b^x, \Phi(Q_1)]=Z(S)$ has order $q$ so that $|C_{\Phi(Q_1)}(b^x)|=q^4$.

Now, let $A\le Q_1$ so that $Z(Q_1)\le C_{\Phi(Q_1)}(A)$. Indeed, if $|C_{\Phi(Q_1)}(A)|=q^2$ then $C_{\Phi(Q_1)}(A)=Z(Q_1)$. Since $\Phi(Q_1)$ is abelian and self-centralizing in $Q_1$, $C_{\Phi(Q_1)}(A)=\Phi(Q_1)$ if and only if $A\le \Phi(Q_1)$. Hence, to complete the proof, we may assume that there is $a\in A\setminus (A\cap \Phi(Q_1))$. Then $C_{\Phi(Q_1)}(A)\le C_{\Phi(Q_1)}(a)$ and $C_{\Phi(Q_1)}(a)$ has order $q^4$. We claim that for $a, a'\in A$ either $C_{\Phi(Q_1)}(a)=C_{\Phi(Q_1)}(a')$ or $|C_{\Phi(Q_1)}(a)\cap C_{\Phi(Q_1)}(a')|=q^3$. Let $a'^x$ be an $L_1$ conjugate of $a'$ in $Q_1\cap Q_2$ so that $C_{\Phi(Q_1)}(a')=C_{\Phi(Q_1)}(a'^x)$. By properties of $Q_2$, if $[C_{\Phi(Q_1)}(a), a'^x]<Z(S)$, then $[C_{\Phi(Q_1)}(a), a'^x]=\{1\}$ and $C_{\Phi(Q_1)}(a)=C_{\Phi(Q_1})(a')$. Otherwise, applying an argument using the commutation homomorphism, we deduce that $|C_{C_{\Phi(Q_1)}(a)}(a'^x)|=q^3$ and the claim holds. Hence, if $|C_{\Phi(Q_1)}(A)|<q^4$, then $|C_{\Phi(Q_1)}(A)|\leq q^3$. Taking another element $\hat{a}$ and an $L_1$-conjugate $\hat{a}^x\in Q_1\cap Q_2$ and applying the commutation homomorphism to $C_{\Phi(Q_1)}(a)\cap C_{\Phi(Q_1)}(a')$, we deduce that $|C_{\Phi(Q_1)}(a)\cap C_{\Phi(Q_1)}(a')\cap C_{\Phi(Q_1)}(\hat{a})|\in\{q^2, q^3\}$ and the result on $C_{\Phi(Q_1)}(A)$ follows.
\end{proof}

\begin{lemma}\label{CentStruct}
Let $a\in Q_1\setminus \Phi(Q_1)$ and set $D:=C_{\Phi(Q_1)}(a)$. Then $|C_{Q_1}(D)|=q^7$, $C_{Q_1}(D)/\Phi(Q_1)$ is a natural $\SL_2(q)$-module for $O^{p'}(\Out_{L_1}(Q_1))\cong \SL_2(q)$, $D=Z(C_{Q_1}(D))$ and if $p=2$ then $\Phi(C_{Q_1}(D))=Z(Q_1)$.
\end{lemma}
\begin{proof}
Take $a$ and $D$ as above. Then $[a, D]=\{1\}$ and since $D$ is normalized by $L_1$, $[\langle a^{L_1}\rangle, D]=\{1\}$. Since $Q_1/\Phi(Q_1)$ is a direct sum of three natural $\SL_2(q)$-modules for $L_1/Q_1$, we have that $\Phi(Q_1)\langle a^{L_1}\rangle$ has order $q^7$ and centralizes $D$. If $|C_{Q_1}(D)|>q^7$, then for some $b$ in $C_{Q_1}(D)\setminus \langle a^{L_1}\rangle\Phi(Q_1)$, we apply the same argument to enlarge $C_{Q_1}(D)$. It follows that $|C_{Q_1}(D)|\geq q^9$. But as argued in \cref{Phi1Cent}, we then have that $|D\cap C_{\Phi(Q_1)}(c)|\geq q^3$ for some $c$ with $\langle c^{L_1}\rangle C_{Q_1}(D)=Q_1$ so that $Z(Q_1)=D\cap C_{\Phi(Q_1)}(c)$, a contradiction. Hence, $|C_{Q_1}(D)|=q^7$ and $C_{Q_1}(D)\normaleq L_1$. Indeed, by \cref{SL2ModRecog}, $C_{Q_1}(D)/\Phi(Q_1)$ is a natural $\SL_2(q)$-module for $O^{p'}(\Out_L(Q_1))\cong \SL_2(q)$. Since $\Phi(Q_1)$ is self-centralizing in $Q_1$, we deduce that $D\le Z(C_{Q_1}(D))\le \Phi(Q_1)$ so that $D=Z(C_{Q_1}(D))$.

By properties of $Q_2$, as in \cref{Ultraspecial}, $C_{Q_1 \cap Q_2}(D)'=Z(S)$ and as $\Phi(C_{Q_1}(D))\normaleq L_1$, $Z(Q_1)\le \Phi(C_{Q_1}(D))$. Assume that $p=2$. Since $Q_2$ is irreducible for $L_2/Q_2$, we have that $\Omega(Q_2)=Q_2$ and we deduce that there are involutions in $C_{Q_1}(D)\setminus \Phi(Q_1)$. Then for any maximal subgroup of $A$ of $\Phi(Q_1)$ containing $Z(Q_1)$, we have that $C_{Q_1}(D)/A$ is either extraspecial or elementary abelian. If $C_{Q_1}(D)/A$ is elementary abelian for all such $A$, then $\Phi(C_{Q_1}(D))=Z(Q_1)$ and so we assume that there is $B$ with $C_{Q_1}(D)/B$ extraspecial. Since $L_1$ acts on $C_{Q_1}(D)/B$ non-trivially and there are elements of order $2$ in $C_{Q_1}(D)\setminus \Phi(Q_1)$, we have a contradiction by \cite[(5.13)]{Greenbook}.
\end{proof}

\begin{lemma}\label{Q2Cent}
Let $x\in S\setminus Q_2$. Then either:
\begin{enumerate}
    \item $p=2$, $|C_{Q_2/Z(S)}(x)|=q^4$, $|C_{\Phi(Q_1)/Z(S)}(x)|=q^3$ and writing $D$ for the preimage in $\Phi(Q_1)$ of $C_{\Phi(Q_1)/Z(S)}(x)$, $C_{Q_2/Z(S)}(x)\le C_{Q_1\cap Q_2}(D)/Z(S)$; or
    \item $p$ is odd, $|C_{Q_2/Z(S)}(x)|=q^3$ and $C_{Q_2/Z(S)}(x)\le \Phi(Q_1)/Z(S)$.
\end{enumerate}
\end{lemma}
\begin{proof}
By \cite[Lemma 3.14 (iii)]{parkerSymp}, we have that $|C_{Q_2/Z(S)}(x)|=q^4$ when $p=2$ and $|C_{Q_2/Z(S)}(x)|=q^3$ when $p$ is odd. Since $S=Q_1Q_2$, $x=q_1q_2$ for some $q_i\in Q_i$ where $i\in\{1,2\}$. Furthermore, as $[q_2, Q_2]\le [Q_2, Q_2]=Z(S)$, we must have that $C_{Q_2/Z(S)}(x)=C_{Q_2/Z(S)}(q_1)$ and by \cref{Phi1Cent}, $|C_{\Phi(Q_1)/Z(S)}(q_1)|\geq q^3$. Since $C_S(\Phi(Q_1)/Z(S))=Q_2$, it follows that $|C_{\Phi(Q_1)/Z(S)}(q_1)|=q^3$. Since $\Phi(Q_1)\le Q_2$, when $p$ is odd we deduce that $C_{Q_2/Z(S)}(x)=C_{\Phi(Q_1)/Z(S)}(x)\le \Phi(Q_1)/Z(S)$ and the result holds in this case. If $p=2$, then writing $D$ for the preimage in $\Phi(Q_1)$ of $C_{\Phi(Q_1)/Z(S)}(x)$, we have that $[x, C_{Q_1}(D)]\le Z(Q_1)$. In particular, $[x, C_{Q_1\cap Q_2}(D)]\le Z(Q_1)$ and since $|C_{Q_1\cap Q_2}(D)|=q^6$ and $|Z(Q_1)|=q^2$, we deduce that $|C_{C_{Q_1\cap Q_2}(D)/Z(S)}(x)|=q^4$ and since $|C_{Q_2/Z(S)}(x)|=q^4$, we deduce that $C_{Q_2/Z(S)}(x)\le C_{Q_1\cap Q_2}(D)/Z(S)$. Indeed, $C_{Q_1\cap Q_2}(D)/Z(S)=C_{Q_2/Z(S)}(x)(\Phi(Q_1)/Z(S))$.
\end{proof}

The above information will suffice in the determination of all potential essential subgroups of a saturated fusion system on $S$, which we achieve in \cref{D4Essen}. We now turn our focus to this endeavor, providing several structural features of a mooted essential subgroup of $\fs$.

\begin{proposition}\label{U5Notle}
Suppose that $\fs$ is a saturated fusion system supported on $S$ and let $E\in\mathcal{E}(\fs)$. If $E\le Q_2$, then $E=Q_2$.
\end{proposition}
\begin{proof}
Suppose that $E\le Q_2$. Since $\Omega(Z(S))=\Phi(Q_2)\le E$, we have that $E\normaleq Q_2$ and that $\Phi(E)\le \Omega(Z(S))$. Assume first that $E$ is elementary abelian. By \cref{D4Basic}, $|E|\leq q^{5}$ and $|N_{Q_2}(E)/E|\geq q^{4}$. Since $[Q_2, E]\le [Q_2, Q_2]= \Omega(Z(S))$, \cref{SEFF} applied to $E$ as an $\Out_{\fs}(E)$-module gives a contradiction. Hence, $E$ is not elementary abelian.

If $\Phi(E)<\Omega(Z(S))$ then for some maximal subgroup $A$ of $\Omega(Z(S))$ containing $\Phi(E)$, we have that $Q_2/A$ is extraspecial and $E/A$ is elementary abelian. Hence, by \cref{Ultraspecial}, $|E/A|\leq p^{1+4n}$ so that $|Q_2/E|\geq q^{4}$. Since $|[E/\Phi(E), Q_2]|\leq \Omega(Z(S))/\Phi(E)$ has order strictly less than $q$, \cref{SEFF} applied to $E/\Phi(E)$ gives a contradiction.

Hence, we have that $\Phi(E)=\Omega(Z(S))$ and $Q_2$ centralizes the chain $\{1\}\normaleq \Phi(E)\normaleq E$, so that by \cref{Chain}, $E=Q_2$.
\end{proof}

\begin{lemma}\label{D4pearls1}
Suppose that $\fs$ is a saturated fusion system supported on $S$ and let $E\in\mathcal{E}(\fs)$. If $Z(E)\not\le Q_2$ and $Z(Q_1)\le E$ then $Z(Q_1)\le \Omega(Z(E))\le E\le Q_1$, $|E\cap \Phi(Q_1)|=q^4$ and $\Omega(Z(E))(E\cap \Phi(Q_1))\in\mathcal{A}(E)\subseteq \mathcal{A}(S)$.
\end{lemma}
\begin{proof}
Suppose that $Z(E)\not\le Q_2$ and that $Z(Q_1)\le E$. Clearly, we have that $Q_1\ne E\ne Q_2$. Then $[E, \Phi(Q_1)]\le [S, \Phi(Q_1)]=Z(Q_1)\le E$ so that $\Phi(Q_1)\le N_S(E)$. Now, $Z(E)\cap Z(Q_1)\ge [Z(E), \Phi(Q_1)]Z(S)>Z(S)$ from which it follows that $E\le Q_1$ and $Z(Q_1)\le Z(E)$. Since $C_S(\Phi(Q_1))=\Phi(Q_1)$ and $Z(E)\not\le Q_2$, we deduce that $\Phi(Q_1)\not\le E$. Note that $E\cap \Phi(Q_1)\le C_{\Phi(Q_1)}(Z(E))$ and since $[E, \Phi(Q_1)]\le Z(Q_1)\le Z(E)$, $C_{\Phi(Q_1)}(Z(E))$ centralizes the chain $\{1\}\normaleq Z(E)\normaleq E$ so that $E\cap \Phi(Q_1)=C_{\Phi(Q_1)}(Z(E))$. Thus, $|E\cap \Phi(Q_1)|\in\{q^2, q^3, q^4\}$ by \cref{Phi1Cent}.

If $|E\cap \Phi(Q_1)|\leq q^3$, then $|N_S(E)/E|\geq |\Phi(Q_1)E/E|\geq q^2$. Since $|[\Phi(Q_1), E]|\leq |Z(Q_1)|=q^2$, an application of \cref{SEFF} yields that $|E\cap \Phi(Q_1)|=q^3$, $N_S(E)=\Phi(Q_1)E$ and $\Phi(E)\cap Z(Q_1)=\{1\}$. Note that $[C_E(O^{p'}(\Out_{\fs}(E))), \Phi(Q_1)]=\{1\}$ so that $C_E(O^{p'}(\Out_{\fs}(E)))\le C_S(\Phi(Q_1))=\Phi(Q_1)$. Let $e\in [E, O^{p'}(\Out_{\fs}(E))]\setminus (E\cap \Phi(Q_1))$. Indeed, we can choose $e$ to lie in the preimage of the unique non-central chief factor for $\Out_{\fs}(E)$ in $E$ with the property $[e, x]\ne\{1\}$ for any $x\in \Phi(Q_1)\setminus (E\cap \Phi(Q_1))$. But by \cref{Phi1Cent}, $|C_{\Phi(Q_1)}(e)|=q^4$, a contradiction since $|E\cap \Phi(Q_1)|=q^3$.

Hence, we have that $|E\cap \Phi(Q_1)|=q^4$. Note that $[\Phi(Q_1), E]\le Z(Q_1)\le \Omega(Z(E))$ and so $\Omega(Z(E))$ contains all non-central chief factors for $O^{p'}(\Out_{\fs}(E))$ in $E$. Now, $A:=\Omega(Z(E))(E\cap \Phi(Q_1))$ is elementary abelian and so has order at most $q^5$. But $\Phi(Q_1)$ centralizes $\Omega(Z(E))\cap \Phi(Q_1)$ and an application of \cref{SEFF} yields that $O^{p'}(\Out_{\fs}(E))\cong \SL_2(q)$ and $|\Omega(Z(E))/(\Omega(Z(E))\cap \Phi(Q_1))|=q$ so that $|A|=q^5$, $m_p(E)=m_p(S)=5n$ and $A\in\mathcal{A}(E)\cap \mathcal{A}(S)$.
\end{proof}

With the aim of demonstrating that $E\in\{Q_1, Q_2\}$, we suppose that $E$ is an essential subgroup of $\fs$ chosen maximally such that $Q_1\ne E\ne Q_2$. By \cref{U5Notle}, $E\not\le Q_2$ and so either $E$ is maximally essential, or $E\le Q_1$ and $Q_1\in\mathcal{E}(\fs)$.

\begin{lemma}\label{D4pearls} 
Suppose that $\fs$ is a saturated fusion system supported on $S$. If $Z(E)\not\le Q_2$ then $p$ is odd, $E=\Omega(Z(E))$, $E\cap Z(Q_1)=Z(S)$, $S=EQ_1$ and $N_S(E)=EZ(Q_1)$.
\end{lemma}
\begin{proof}
We suppose first that $Z(Q_1)\le E$ and let $A=\Omega(Z(E))(E\cap \Phi(Q_1))$ so that $A\in\mathcal{A}(S)$ by \cref{D4pearls1}. In addition, $Z(Q_1)\le \Omega(Z(E))\le E\le Q_1$ and $|E\cap \Phi(Q_1)|=q^4$. Assume that $p$ is odd. Note that $E\cap Q_2/Z(S)\le C_{Q_2/Z(S)}(\Omega(Z(E)))\le \Phi(Q_1)/Z(S)$ by \cref{Q2Cent}. Since $|\Omega(Z(E))\Phi(Q_1)/\Phi(Q_1)|=q$, we conclude that $C_{Q_1}(E\cap \Phi(Q_1))=\Omega(Z(E))C_{Q_1\cap Q_2}(E\cap \Phi(Q_1))$. 

If $Q_1\in\mathcal{E}(\fs)$ then $p'$-elements of $O^{p'}(\Aut_{\fs}(Q_1))$ necessarily act faithfully on $\Phi(Q_1)$ (this follows from the three subgroups lemma) and since $[S, \Phi(Q_1)]\le Z(Q_1)$, we deduce that $O^{p'}(\Aut_{\fs}(Q_1))$ acts faithfully on $Z(Q_1)$ and \cref{SEFF} implies that $O^{p'}(\Out_{\fs}(Q_1))\cong \SL_2(q)$. Then $E\cap \Phi(Q_1)$ is normalized by $O^{p'}(\Aut_{\fs}(Q_1))$ and so too is $C_{Q_1}(E\cap \Phi(Q_1))/\Phi(Q_1)$. Indeed, by $C_{Q_1}(E\cap \Phi(Q_1))/\Phi(Q_1)$ is a natural module for $O^{p'}(\Out_{\fs}(Q_1))$ by \cref{SEFF} and since $O^{p'}(\Aut_{\fs}(Q_1))$ is transitive on subgroups of order $q$ in this action by \cref{sl2p-mod}, we infer that there is $\alpha\in O^{p'}(\Aut_{\fs}(Q_1))$ with $\Omega(Z(E))\alpha\le Q_2$ but $\Omega(Z(E))\not\le \Phi(Q_1)$. Then $[E\alpha, \Omega(Z(E))\alpha]=[E, \Omega(Z(E))]\alpha=\{1\}$ and applying \cref{Q2Cent}, we conclude that $E\alpha\le Q_2$. But then $Q_2\le N_S(E\alpha)$ and since $E$ is fully normalized, we have that $N_S(E)=E\Phi(Q_1)$ has index at most $q^3$ in $S$ and is contained in $Q_1$. But $N_S(E)\cap Q_2=\Phi(Q_1)(E\cap Q_2)\le \Phi(Q_1)$ and we have a contradiction.

Hence, we continue assuming now that $Q_1\not\in\mathcal{E}(\fs)$ so that $E$ is maximally essential. For $\beta\in O^{p'}(\Aut_{\fs}(E))$ of order $q-1$ and normalizing $\Aut_S(E)$, $\beta$ lifts to $\hat{\beta}$ in $\Aut_{\fs}(S)$. Hence, upon restriction, $\beta$ normalizes $Z(S)$. Since $E\le Q_1$, we have that $\Phi(E)\le \Phi(Q_1)$ and as $N_S(E)=E\Phi(Q_1)$, we deduce that $O^p(O^{p'}(\Aut_{\fs}(E)))$ centralizes $\Phi(E)$. Moreover, $[\Phi(Q_1), E]\le Z(Q_1)\le \Omega(Z(E))$ and writing $P_E$ for the preimage in $E$ of $[E/\Phi(E), O^p(O^{p'}(\Aut_{\fs}(E)))]$ and $C_E$ for the preimage in $E$ of $C_{E/\Phi(E)}(O^p(O^{p'}(\Aut_{\fs}(E))))$, we have that $P_E\le \Omega(Z(E))\Phi(E)$ and $P_E\cap C_E\le \Phi(E)$. Indeed, since $[E, \Phi(Q_1)]\le Z(Q_1)$ and as $\beta$ also normalizes $Z(Q_1)\cap C_E$, we infer that $[\Phi(Q_1), C_E]\le Z(Q_1)\cap C_E\le Z(S)$ and $C_E\le Q_2$. Then $C_E\le E\cap Q_2\le \Phi(Q_1)$ so that $E=P_EC_E=\Omega(Z(E))(E\cap \Phi(Q_1))$ is elementary abelian and $E=\Omega(Z(E))$. Thus, $E=P_E\times C_E$, $|P_E|=q^2$ and $|C_E|=q^3$. Again, since $\beta$ normalizes $Z(S)$, we calculate that either $Z(S)\le [E, O^{p'}(\Aut_{\fs}(E))]$ or $Z(S)\le C_E$. Since $E\not\le Q_2$, $[E, \Phi(Q_1)]\ne Z(S)$ and we have that $Z(S)\le C_E$. Now, $Z(S)\le C_E\le \Phi(Q_1)$ so that $C_E\normaleq Q_2$. Moreover, $N_S(E)=E\Phi(Q_1)$, $N_S(E)/C_E\cong T\in\syl_p(\PSL_3(q))$ and $Z(N_S(E)/C_E)=Z(Q_1)C_E/C_E$. Now, $C_{Q_1\cap Q_2}(N_S(E)/Z(Q_1)C_E)$ normalizes $E$ and so $C_{Q_1\cap Q_2}(N_S(E)/Z(Q_1)C_E)\le \Phi(Q_1)$. Then $Q_1\cap Q_2/\Phi(Q_1)$ has order $q^3$ and embeds in the automorphism group of $N_S(E)/C_E$. By \cite[Proposition 5.3]{parkerBN}, we have a contradiction.

Assume now that $p=2$. Since $Z(E)\not\le Q_2$, we have by \cref{Q2Cent} that $E\cap Q_2\le C_{Q_1\cap Q_2}(E\cap \Phi(Q_1))$. Indeed, $C_{Q_1}(E\cap \Phi(Q_1))=\Omega(Z(E))C_{Q_1\cap Q_2}(E\cap \Phi(Q_1))$. Let $B$ be an $L_1$-conjugate of $A$ contained in $Q_1\cap Q_2$. Since $A\cap \Phi(Q_1)$ is normalized by $L_1$, we deduce that $B\cap \Phi(Q_1)=A\cap \Phi(Q_1)$ has order $q^4$. Now, $[A, B]\le Z(Q_1)$ by \cref{CentStruct}, so that $[A, B, E]=\{1\}$. Moreover, $[A, E, B]=[E\cap \Phi(Q_1), E, B]\le [Z(Q_1), B]=\{1\}$ and the three subgroups lemma gives that $[B, E, A]=\{1\}$. Since $[B, E]\le \Phi(Q_1)$, we have that $[B, E]\le C_{\Phi(Q_1)}(\Omega(Z(E)))=E\cap \Phi(Q_1)$ so that $B\le N_S(E)$. But now, $B\Phi(Q_1)\le N_S(E)$ and since $|N_S(E)/E|=q$ and $|\Phi(Q_1)/E\cap \Phi(Q_1)|=q$, we must have that $|E\cap B\Phi(Q_1)|=q^5$ and $E\cap B\Phi(Q_1)\le Q_2$. But $E\cap B\Phi(Q_1)$ centralizes $A$ so that $\Omega(E\cap B\Phi(Q_1))\le C_{B\Phi(Q_1)}(A)=E\cap \Phi(Q_1)$. Then $\{1\}\ne \mho^1(E\cap B\Phi(Q_1))\le Z(S)\cap \mho^1(E)$. If $E\not\le C_{Q_1}(E\cap \Phi(Q_1))$, then $|\Phi(E)\cap Z(Q_1)|\geq |[E, E\cap \Phi(Q_1)]\mho^1(E)\cap Z(Q_1)|>q$ and as $|\Phi(Q_1)E/E|=q$, \cref{SEFF} provides a contradiction. Hence, $E\le C_{Q_1}(E\cap \Phi(Q_1))$ so that $A=\Omega(Z(E))$. From this, we deduce that $E=A$. But $A\cap B\Phi(Q_1)=E\cap \Phi(Q_1)$ has order $q^4$, a contradiction.

Therefore, if $Z(E)\not\le Q_2$ then $Z(Q_1)\not\le E$. If there is $x\in (Z(Q_1)\cap E)\setminus Z(S)$ then $Z(E)\le Q_1$ so that $[Z(E), Z(Q_1)]=\{1\}$ and $Z(Q_1)$ centralizes the chain $\{1\}\normaleq Z(E)\normaleq E$, a contradiction. Thus, $Z(Q_1)\cap E=Z(S)$ and since $E$ is $S$-centric, we deduce that $E\not\le Q_1$. Furthermore, $[E, Z(Q_1)]=Z(S)\le \Omega(Z(E))$ and so we must have that $\Omega(Z(E))\not\le Q_1$. But now, $|Z(Q_1)E/E|=q$ and $C_{\Omega(Z(E))}(Z(Q_1))=\Omega(Z(E))\cap Q_1$ has index $|\Omega(Z(E))Q_1/Q_1|$ in $\Omega(Z(E))$. By \cref{SEFF}, we have that $|\Omega(Z(E))Q_1/Q_1|=q$ and $S=\Omega(Z(E))Q_1$. Then $E=\Omega(Z(E))(E\cap Q_1)$. Furthermore, since $\Omega(Z(E))\not\le Q_1$, $E\cap Q_1\le E\cap Q_1\cap Q_2$ so that $\Phi(E)=\Phi(\Omega(Z(E))(E\cap Q_1\cap Q_2))=\Phi(E\cap Q_1\cap Q_2)\le Z(S)$. Since $[Z(Q_1), \Omega(Z(E))]=Z(S)$ we have by \cref{SEFF} that $\Phi(E)=\{1\}$ and $E$ is elementary abelian. Again, since $E\not\le Q_1$, $E\cap Q_1\le E\cap Q_1\cap Q_2$. If $p=2$, then $E\le (E\cap Q_1)(E\cap Q_2)\le Q_2$, a contradiction by \cref{U5Notle}. Hence, $p$ is odd.
\end{proof}

Thus, whenever $Z(E)\not\le Q_2$, we see that $Q_1\not\ge E\not\le Q_2$ and $E$ is maximally essential in $\fs$. 

\begin{proposition}\label{3D4pearl1}
Suppose that $\fs$ is a saturated fusion system supported on $S$. Then $Z(E)\le Q_2$.
\end{proposition}
\begin{proof}
Aiming for a contradiction, we assume that $Z(E)\not\le Q_2$ so that by \cref{D4pearls}, $E=\Omega(Z(E))\not\le Q_2$, $S=EQ_1$, $N_S(E)=EZ(Q_1)$ and $p$ is odd. Since $E\not\le Q_2$, we have that $EQ_2\cap Q_1\not\le Q_1\cap Q_2$. However, $[EQ_2\cap Q_1, E\cap \Phi(Q_1)]\le Z(S)$ and we deduce that $EQ_2\le Q_2C_{Q_1}(E\cap \Phi(Q_1))$. Then $[Z(C_{Q_1}(E\cap \Phi(Q_1))), EQ_2]=Z(S)\le E$ so that $Z(Q_1)(E\cap \Phi(Q_1))=Z(C_{Q_1}(E\cap \Phi(Q_1)))$ and $|(E\cap \Phi(Q_1))Z(Q_1)|\in\{q^2, q^3, q^4\}$. But now, $\Phi(Q_1)\le N_S(N_S(E))$. In particular, $\Phi(Q_1)$ normalizes $EZ(Q_1)/C_E(O^{p'}(\Aut_{\fs}(E)))\cong T\in\syl_p(\PSL_3(q))$. Moreover, if $x\in \Phi(Q_1)$ has $[x, EZ(Q_1)]\le Z(S)C_E(O^{p'}(\Aut_{\fs}(E)))$ then $x\in N_S(E)\cap \Phi(Q_1)=Z(Q_1)(E\cap \Phi(Q_1))$. Applying \cite[Proposition 5.3]{parkerBN}, $|\Phi(Q_1)/Z(Q_1)(E\cap \Phi(Q_1)|<q^2$ and we deduce that $|Z(Q_1)(E\cap \Phi(Q_1))|=q^4$ and $|N_S(N_S(E))/N_S(E)|<q^2$.

Let $t$ be a non-trivial element in $Z(O^{p'}(\Aut_{\fs}(E)))$. Since we have chosen $E$ to be maximally essential, applying the Alperin--Goldschmidt theorem and using that $E$ is receptive, $t$ lifts to $\hat{t}\in\Aut_{\fs}(S)$ and normalizes $Q_2$. Note that $\hat{t}$ centralizes $Z(Q_1)(E\cap Q_1)/Z(S)$. Write $C_t$ to be the preimage in $Q_2$ of $C_{Q_2/Z(S)}(\hat{t})$. Then $[C_t, \hat{t}, C_t]=\{1\}$ and since $Z(S)$ is inverted by $\hat{t}$, we deduce that $C_t$ is abelian. Indeed, we must have that $|[Q_1\cap Q_2/\Phi(Q_1), \hat{t}]|\geq q^2$. Moreover, we have that $\Phi(Q_1)/Z(Q_1)(E\cap \Phi(Q_1))$ is inverted by $\hat{t}$, for otherwise $[E, t, C_{\Phi(Q_1)}(\hat{t})]\le Z(S)$ and it follows that $E\le Q_2$. Therefore, $E[Q_1\cap Q_2, \hat{t}]/(E\cap \Phi(Q_1))Z(Q_1)$ is inverted by $\hat{t}$ is so is abelian. But then $\Phi(Q_1)[Q_1\cap Q_2, \hat{t}]\le N_S(N_S(E))$. But then $|N_S(N_S(E))/N_S(E)|\geq q^3$, a contradiction. Hence, $Z(E)\le Q_2$, as desired.
\end{proof}

\begin{proposition}\label{3D4pearl}
Suppose that $\fs$ is a saturated fusion system supported on $S$. Then $Z(E)\le \Phi(Q_1)$, $Z(Q_1)\le E$ and $\Phi(Q_1)\le N_S(E)$.
\end{proposition}
\begin{proof}
Since $Z(E)\le Q_2$ by \cref{3D4pearl1} it follows immediately from \cref{Q2Cent} that $Z(E)\le \Phi(Q_1)$ when $p$ is odd. When $p=2$, we have that $Z(E)\le Q_1\cap Q_2$. In either case, since $Z(Q_1)$ centralizes the chain $\{1\}\normaleq Z(E)\normaleq E$, we deduce that $Z(Q_1)\le E$. Then $[\Phi(Q_1), E]\le [\Phi(Q_1), S]=Z(Q_1)\le E$ and $\Phi(Q_1)\le N_S(E)$. It remains to prove that $Z(E)\le \Phi(Q_1)$ whenever $p=2$. Aiming for a contradiction, assume otherwise, noting that $Z(E)\le Q_1\cap Q_2$. We have that $|E\cap \Phi(Q_1)|\leq q^4$ since $Z(E)\not\le \Phi(Q_1)$ and as $E\not\le Q_2$, \cref{Q2Cent} yields that $Z(E)\cap \Phi(Q_1)$ has index at most $q$ in $Z(E)$. 

Suppose first that $Z(E)$ contains no non-central chief factors for $O^{2'}(\Aut_{\fs}(E))$ so that $Z(S)$ is normalized by $O^{2'}(\Aut_{\fs}(E))$. Note that $E\not\le Q_1$ for otherwise $Z(Q_1)\le Z(E)$ and since $\Phi(Q_1)$ centralizes $E/Z(Q_1)$, we would have a contradiction. Set $Z_E$ to be the preimage in $E$ of $Z(E/Z(S))$ so that $Z_E$ is normalized by $O^{2'}(\Aut_{\fs}(E))$ and contains $Z(Q_1)$. Now, since $L_1/Q_1$ acts on $Q_1/\Phi(Q_1)$ as a direct sum of three natural $\SL_2(q)$-modules, we either have that $E\cap Q_1\le Q_1\cap Q_2$ so that $\Omega(E)\le Q_2$; or $Z_E\le Q_1\cap Q_2$. In either case, since $[\Phi(Q_1), E]\le Z(Q_1)\le Z_E\cap \Omega(E)$ and $[\Phi(Q_1), Q_2]\le Z(S)\le Z(E)$, $\Phi(Q_1)$ centralizes a chain and we have a contradiction.

Hence, $Z(E)$ contains a non-central chief factor for $O^{2'}(\Aut_{\fs}(E))$ and applying \cref{SEFF} we infer that $|E\cap \Phi(Q_1)|=q^4$, $|Z(E)/Z(E)\cap \Phi(Q_1)|=q$ and $N_S(E)=E\Phi(Q_1)$. In particular, $A:=Z(E)(E\cap \Phi(Q_1))$ is elementary abelian of order $q^5$. If $E\not\le Q_1$, then $Z(S)=[E, Z(Q_1)]\le \Phi(E)$. Note that if $Z(Q_1)\le \Phi(E)$ then $\Phi(Q_1)$ centralizes $E/\Phi(E)$ and so we must have that $[E, E\cap \Phi(Q_1)]\le Z(S)$. Now, for $e\in \Omega(E)\cap Q_1$ an involution, we have that $[e, E\cap \Phi(Q_1)]\le Z(S)$ and so either $e\in E\cap Q_2$, or $[e, E\cap \Phi(Q_1)]=\{1\}$. In the latter case, $[e, A]=\{1\}$ so that $e\in A\le Q_2$. Hence, $Z(Q_1)\le \Omega(E)\le Q_2$ and as $Z(S)\le \Phi(E)\le \Phi(E)\Omega(E)\le Q_2$, $\Phi(Q_1)$ centralizes the chain $\Phi(E)\normaleq \Phi(E)\Omega(E)\normaleq E$, a contradiction.

Finally, we deduce that $Z(E)$ contains a non-central chief factor for $O^{2'}(\Aut_{\fs}(E))$ and $E\le Q_1$. In particular, $Z(Q_1)\le Z(E)\le Q_1\cap Q_2$ and so $Z(S)\cap \Phi(E)=\{1\}$ by \cref{SEFF}. Thus, $[E\cap Q_2, E\cap Q_2]=\{1\}$ and since $A\le E\cap Q_2$ is maximal abelian, we infer that $A=E\cap Q_2$. Moreover, any $e\in E\setminus A$ is not contained in $Q_2$ and centralizes $Z(E)$ and \cref{Q2Cent} implies that $E\le C_{Q_1}(E\cap \Phi(Q_1))$. Hence, $A=\Omega(E)=Z(E)$. But $[Q_2, A]\le Z(S)\le A$ and we calculate that $E=C_S(A)$ so that $E$ is normalized by $Q_2$. Since $N_S(E)=E\Phi(Q_1)$, this is a contradiction.
\end{proof}

\begin{proposition}\label{Phi1Contain}
Suppose that $\fs$ is a saturated fusion system supported on $S$. Then $\Phi(Q_1)\le E$, $Q_1\cap Q_2\le N_S(E)$ and $m_p(E)=m_p(S)$.
\end{proposition}
\begin{proof}
Aiming for a contradiction, assume throughout that $\Phi(Q_1)\not\le E$. Since $Z(S)\le Z(E)$, $[\Phi(Q_1), E/Z(E)]\le Z(Q_1)Z(E)/Z(E)$ and $[\Phi(Q_1), Z(E)]=\{1\}$ we must have by \cref{SEFF} that $|E\cap \Phi(Q_1)|\geq q^4$ and $Z(E)\cap Z(Q_1)=Z(S)$. Another application of \cref{SEFF} yields that $\Phi(E)\cap Z(Q_1)=Z(S)$. In particular, $[E,E\cap \Phi(Q_1)]=Z(S)$ and we have that $|E\cap \Phi(Q_1)|=q^4$.

Now, since $[\Phi(Q_1), E/Z(E)]=Z(Q_1)Z(E)/Z(E)$, \cref{SEFF} implies that $O^{p'}(\Out_{\fs}(E))\cong \SL_2(q)$, $N_S(E)=E\Phi(Q_1)$ and there is a unique non-central chief factor for $O^{p'}(\Out_{\fs}(E))$ inside $E$. Indeed, $O^{p'}(\Aut_{\fs}(E))$ acts trivially on $Z(S)$ and normalizes $D$, the preimage in $E$ of $Z(E/Z(S))$.

Assume that $\Omega(D)\not\le Q_2$. In particular, $[D, \Phi(Q_1)]\not\le Z(S)$ and so we must have that $[\Omega(D), \Phi(Q_1)]Z(S)=Z(Q_1)$. Hence, $[\Phi(Q_1), E]\le Z(Q_1)\le \Omega(D)$ and $\Omega(D)/Z(E)$ contains the unique non-central chief factor for $O^{p'}(\Out_{\fs}(E))$ inside $E$. If $p=2$, then $\Omega(D)=(\Omega(D)\cap Q_1)(\Omega(D)\cap Q_2)$ by \cref{D4Basic} and $(\Omega(D)\cap Q_1)\not\le Q_2$. Note that if $\Omega(D)\cap Q_2\not\le Q_1\cap Q_2$, then $\Phi(Q_1)\not\ge [\Omega(D)\cap Q_1, \Omega(D)\cap Q_2]\le Z(S)$, a contradiction. Hence, $\Omega(D)\cap Q_2\le Q_1$ and $\Omega(D)\le Q_1$. But then $Z(Q_1)\le \Omega(Z(\Omega(D)))$ so that $[\Phi(Q_1), E]=Z(Q_1)\le \Omega(Z(\Omega(D)))$ and $\Omega(Z(\Omega(D)))$ contains the unique non-central chief factor for $O^{2'}(\Out_{\fs}(E))$ inside $E$. But then for $r$ an odd order element of $O^{2'}(\Aut_{\fs}(E))$, $[r, E, D]=\{1\}$ and $[E, D, r]=\{1\}$ and by the three subgroup lemma, $[r, D]\le Z(E)$ so that $[D, r]=[D, r, r]\le [Z(E), r]=\{1\}$, a contradiction since $D$ contains a non-central chief factor.

If $p$ is odd, then $C_{(E\cap Q_2)/Z(S)}(D)\le C_{Q_2/Z(S)}(D)\cap E/Z(S)\le (E\cap \Phi(Q_1))/Z(S)$ by \cref{Q2Cent}. Hence, $E\cap Q_2\le \Phi(Q_1)$. Since $[E, Z(Q_1)]\ne\{1\}$, $E\not\le Q_1$ and since $[E, E\cap Q_1]\le \Phi(E)\le E\cap Q_2\le \Phi(Q_1)$, we deduce that $E\cap Q_1\le E\cap Q_1\cap Q_2\le \Phi(Q_1)$. But now, since $|N_S(E)/E|=q$ and $D$ contains the unique non-central chief factor within $E$, we deduce that $S=DQ_1$ and $E=D(E\cap \Phi(Q_1))$ is non-abelian. Hence, $[E, E]=Z(S)$. Note that $C_{E}(O^p(\Aut_{\fs}(E)))\le \Phi(Q_1)$ from which it follows that $C_{E}(O^p(\Aut_{\fs}(E)))=Z(E)$ is of order $q^3$. Furthermore, we have that $\Phi(Q_1)E/Z(E)\cong T\in\syl_p(\PSL_3(q))$ and $Z(\Phi(Q_1)E/Z(E))=Z(Q_1)Z(E)/Z(E)$. Now, $Q_1\cap Q_2$ normalizes $Z(E)$ and $\Phi(Q_1)E$ and so acts on $\Phi(Q_1)E/Z(E)$. Moreover, for $x\in Q_1\cap Q_2$, if $[x, \Phi(Q_1)E]\le Z(E)Z(Q_1)$ then $x\in N_{Q_2}(E)=\Phi(Q_1)$. Applying \cite[Proposition 5.3]{parkerBN}, $|(Q_1\cap Q_2)E/\Phi(E)E|<q^2$, a contradiction.

Therefore, $\Omega(D)\le Q_2$ and $[\Omega(D), \Phi(Q_1)]=Z(S)\le Z(E) $. Then, as $[E, \Phi(Q_1)]\le Z(Q_1)\le \Omega(D)$, $\Phi(Q_1)$ centralizes the chain $\{1\}\normaleq Z(E)\normaleq \Omega(D)\normaleq E$, a contradiction. Hence, $\Phi(Q_1)\le E$. Since $\Phi(Q_1)\in\mathcal{A}(S)$, we have that $m_p(S)=m_p(E)=5n$ and $\Phi(Q_1)\in\mathcal{A}(E)$. Since $[E, Q_1\cap Q_2]\le [S, Q_1\cap Q_2]=\Phi(Q_1)\le E$, we have that $Q_1\cap Q_2\le N_S(E)$.
\end{proof}

Since $\Phi(Q_1)\le E$ and $\Phi(Q_1)\in\mathcal{A}_{\normaleq}(S)$, $\Phi(Q_1)\in\mathcal{A}_{\normaleq}(E)\ne \emptyset$ and so $J_{\normaleq}(E)$ is defined. We will use this group and its properties frequently in the following lemmas and propositions.

\begin{proposition}\label{NormalThompson1}
Suppose that $\fs$ is a saturated fusion system supported on $S$ and let $E\in\mathcal{E}(\fs)$. Then either $J_{\normaleq}(E)\le Q_1$ or $J_{\normaleq}(E)\le Q_2$.
\end{proposition}
\begin{proof}
Aiming for a contradiction, we suppose throughout that there is $A,C\in\mathcal{A}_{\normaleq}(E)$ with $A\not\le Q_1$ and $C\not\le Q_2$. Since $\Phi(Q_1)$ is a normal elementary abelian subgroup of $E$ and $m_p(S)=m_p(E)=5n$, we have that $|A|=|C|=q^5$. Moreover, since $A\not\le Q_1$, $A\cap Z(Q_1)=C_{Z(Q_1)}(A)=Z(S)$. Then $B:=(A\cap Q_1)Z(Q_1)$ is elementary abelian of order $|A\cap Q_1|q$ and we conclude that $|A\cap Q_1|=q^4$ and $S=AQ_1$. Since $A\not\le Q_1$ and $A$ centralizes $A\cap Q_1$, we have that $A\cap Q_1\le Q_1\cap Q_2$ and $B\le Q_1\cap Q_2$. If $p=2$ then by \cref{D4Basic}, we have that $A\le Q_2$. If $p$ is odd, then $[A, B]=[A, Z(Q_1)]=Z(S)$ and by \cref{Q2Cent}, since $|B/Z(S)|=q^4$, we have again that $A\le Q_2$. In particular, $Q_2=A(Q_1\cap Q_2)$.

Since $Q_2=C_S(\Phi(E)/Z(S))$, we have that $Z(Q_1)=[C, \Phi(E)]Z(S)\le C$ so that $C\le Q_1$. Let $c\in C\setminus (C\cap Q_2)$. Then by \cref{Phi1Cent}, for $F_c:=C_{\Phi(Q_1)}(c)$, $|F_c|=q^4$, $|C_{Q_1}(F_c)|=q^7$ and $|C_{Q_1}(F_c)\cap Q_2|=q^6$. Note that $[c, A]\Phi(Q_1)=C_{Q_1}(F_c)\cap Q_2$ and as $A$ normalizes $C$, we deduce that $|(C\cap C_{Q_1}(F_c))\Phi(Q_1)|>q^6$. But then, $(C\cap C_{Q_1}(F_c))F_c$ is elementary abelian of order strictly larger than $q^5$, a contradiction. Hence, either $J_{\normaleq}(E)\le Q_1$ or $J_{\normaleq}(E)\le Q_2$. 
\end{proof}

\begin{proposition}\label{Q2Contain}
Suppose that $\fs$ is a saturated fusion system supported on $S$ and let $E\in\mathcal{E}(\fs)$. If $Q_2\le E$, then $E=Q_2$.
\end{proposition}
\begin{proof}
Since $Q_2\le E$, we ascertain immediately that $E\normaleq S$, $Z(E)=Z(S)$ and $Q_2$ is the preimage in $E$ of $J(E/Z(E))$. Thus, $Q_2$ is characteristic in $E$. Aiming for a contradiction, we assume throughout that $Q_2<E$ and so there is $x\in (E\cap Q_1)\setminus (Q_1\cap Q_2)$. In particular, $Z(S)[Q_2, x]\le \Phi(E)$. We remark that $\Phi(E)<Q_1\cap Q_2$, else $S$ centralizes $E/\Phi(E)$.

Suppose first that $p$ is odd. By \cref{Q2Cent}, we have that $|C_{Q_2/Z(S)}(x)|=q^3$ and so we deduce that $|Z(S)[Q_2, x]|=q^6$. Indeed, it follows that $|Z(S)[Q_1\cap Q_2, x]|=q^5$ and since $x\in Q_1$, $[x, Q_1\cap Q_2]\le \Phi(Q_1)$ and we deduce that $\Phi(Q_1)=Z(S)[Q_1\cap Q_2, x]\le \Phi(E)$. Since $[S, E]\le Q_1\cap Q_2\le Q_2$ and $[S, Q_1\cap Q_2]\le \Phi(Q_1)$, we deduce that $O^{p'}(\Aut_{\fs}(E))$ centralizes $E/Q_2$. Moreover, \cref{SEFF} reveals that $|\Phi(E)|=q^7$, $|S/E|=q$ and $Q_2/\Phi(E)$ is a natural module for $O^{p'}(\Out_{\fs}(E))\cong \SL_2(q)$. Now, $E/\Phi(E)$ splits by coprime action, and we can arrange that $x$ lies in the preimage in $E$ of $C_{E/\Phi(E)}(O^{p'}(\Out_{\fs}(E)))$. Hence, $E_x:=\langle x\rangle\Phi(E)$ is normalized by $O^{p'}(\Aut_{\fs}(E))$. Note that the preimage of $Z(E_x/Z(S))$ is equal to $C_{\Phi(Q_1)}(x)$ of order $q^4$ and normalized by $O^{p'}(\Aut_{\fs}(E))$. Note we can also choose $y$ in the preimage in $E$ of $C_{E/\Phi(E)}(O^{p'}(\Out_{\fs}(E)))$ with $y\in (E\cap Q_1)\setminus (Q_1\cap Q_2)$ and $C_{\Phi(Q_1)}(x)\ne C_{\Phi(Q_1}(y)$, and we may form $E_y$ normalized by $O^{p'}(\Aut_{\fs}(E))$. Then $Z_2(E)=C_{\Phi(Q_1)}(x)\cap C_{\Phi(Q_1)}(y)$ and $\Phi(Q_1)=C_{\Phi(Q_1)}(x)C_{\Phi(Q_1)}(y)$ is normalized by $O^{p'}(\Aut_{\fs}(E))$. But then $Z(Q_1)=[E, \Phi(Q_1)]$ is normalized by $O^{p'}(\Aut_{\fs}(E))$ and as $Q_1$ centralizes the chain $\{1\}\normaleq Z(Q_1)\normaleq \Phi(Q_1)$, we have that all $p'$-elements in $O^{p'}(\Aut_{\fs}(E))$ centralize $\Phi(Q_1)$. Since $\Phi(Q_1)$ is self-centralizing in $E$, the three subgroups lemma implies that all $p'$-elements in $O^{p'}(\Aut_{\fs}(E))$ centralize $E$, a contradiction.

Hence, we have that $p=2$. We employ the fusion systems package in MAGMA \cite{Comp1} when $q=2$ to deduce that $E=Q_2$. Hence, we may assume that $q>2$. As in the odd prime case, we recognize that $O^{2'}(\Aut_{\fs}(E))$ centralizes $E/Q_2$. Let $x\in (E\cap Q_1)\setminus Q_2$ and write $C_x$ for the preimage in $Q_2$ of $C_{Q_2/Z(S)}(x)$. Then, applying \cref{Q2Cent}, $C_x=[x, Q_2]Z(S)$ is of order $q^5$ and is not contained in $\Phi(Q_1)$. Indeed, $C_x\cap \Phi(Q_1)=C_{\Phi(Q_1)}(x)$ has order $q^4$. We suppose first that $C_x$ is the preimage in $Q_2$ of $C_{Q_2/Z(S)}(E)$ so that $|S/E|\geq q^2$. Note that as $E_x:=\langle x\rangle Q_2$ is normalized by $O^{p'}(\Aut_{\fs}(Q_2))$, so too is $C_x=[E_x, Q_2]$. Indeed, $|Q_2/C_x|=|C_x/Z(S)|=q^4$ and at least one of these quotients contains a non-central chief factor for $O^{2'}(\Out_{\fs}(E))$. Note that $C_x=[x, Q_2]Z(S)\le \Phi(E)$ and so $Q_2/C_x$ contains a non-central chief factor. Indeed, applying \cref{MaxEssenEven} and \cref{SL2ModRecog} we have that $C_x=\Phi(E)$ and $Q_2/\Phi(E)$ is irreducible for $O^{2'}(\Out_{\fs}(E))\cong \SL_2(q^2)$. We assume now that there is $x,y\in E\setminus Q_2$ with $C_x\cap \Phi(Q_1)\ne C_y\cap \Phi(Q_1)$. Indeed, applying \cref{Q2Cent} we deduce that $|C_x\cap C_y|=q^3$ and $|C_xC_y|=q^7$. Since $C_xC_y\le \Phi(E)$, we deduce that $C_xC_y=\Phi(E)$ for otherwise, using the structure of $Q_1/\Phi(Q_1)$ as a $L_1/Q_1$-module, we would have that $\Phi(E)=Q_1\cap Q_2$ and $S$ would centralize $E/\Phi(E)$. Hence, $|S/E|\geq q$ and $O^{2'}(\Out_{\fs}(E))$ acts non-trivially on $Q_2/C_xC_y$. Applying \cref{MaxEssenEven} and \cref{SL2ModRecog}, we have that $Q_2/C_xC_y$ is a natural module for $O^{2'}(\Out_{\fs}(E))\cong \SL_2(q)$.

Therefore, we have that $p=2$, $q>2$ and $O^{2'}(\Out_{\fs}(E))\cong \SL_2(q^a)$ for $a\in\{1,2\}$. Set $K_E$ the group of order $q^a-1$ obtained by lifting a Hall $2'$-subgroup of $N_{O^{2'}(\Aut_{\fs}(E))}(\Aut_S(E))$ to $\Aut(S)$, set $K_2$ a Hall $2'$-subgroup of $\Aut_{L_2}(S)$ and set $H:=\langle K_E, K_2\rangle\le \Aut(S)$. Let $\hat{K_E}$ be the subgroup of $\Aut(S/Q_2)$ induced by $K_E$ and define $\hat{K_2}$ and $\hat{H}$ similarly. In particular, $K_E\cong \hat{K_E}$, $K_2\cong \hat{K_2}$ and $\bar{H}:=\langle \hat{K_2}, \hat{K_E}\rangle\le \hat{H}$. Since $\hat{K_2}$ is transitive on involutions in $S/Q_2$, and $\hat{K_E}$ fixes a subgroup of $S/Q_2$ of order $q^{3-a}$, we deduce that $\hat{K_2}\cap \hat{K_E}=\{1\}$. Indeed, $|\bar{H}|\geq |\hat{K_E}||\hat{K_2}|=(q^3-1)(q^a-1)$. Now, $\hat{K_2}$ is a Singer cycle of $\Aut(S/Q_2)\cong \GL_{3n}(2)$ and \cite{KantorSinger} yields that $\GL_{3n/s}(2^s)\normaleq \bar{H}\le \Gamma L_{3n/s}(2^s)$ for some $s$ which divides $3n$. If $s=3n$ then $\bar{H}$ lies in the normalizer of a Singer cycle. However, the normalizer of a Singer cycle in $\GL_{3n}(2)$ has order $3n(q^3-1)$ by \cite[7.3 Satz]{Huppert} which is only divisible by $\hat{H}$ if $a=1$ and $n=2$. In this case, we compute directly in $\GL_6(2)$ that $\hat{K_2}\normaleq \bar{H}$ and that no element of order $3$ in $\bar{H}$ fixes a subgroup of $S/Q_2$ of order $2^4$, a contradiction since $\hat{K_E}\le \bar{H}$.

Hence, $s<3n$ and $\bar{H}$ is not solvable. We note that $H$ is trivial on $Z(S)$. Since $C_H(Z(Q_1))\normaleq H$ and $|Z(Q_1)/Z(S)|=q$ we deduce that $C_H(Z(Q_1))$ has a composition factor isomorphic to $\PSL_{3n/s}(2^s)$. Moreover, the three subgroups lemma yields that $C_H(Z(Q_1))$ centralizes $S/Q_1$ so odd order elements of $C_H(Z(Q_1))$ acts faithfully on $Q_1/Q_1\cap Q_2$. We note by the three subgroups lemma that $C_{\Aut(Q_1)}(Z(Q_1))\cap C_{\Aut(Q_1)}(\Phi(Q_1)/Z(Q_1))$ centralizes $Q_1/\Phi(Q_1)$ and $L_1/Q_1$ embeds faithfully in $C_{\Aut(Q_1)}(\Phi(Q_1)/Z(Q_1))$. In particular, the subgroups generated by induced actions of $L_1$ and $C_H(Z(Q_1))$ on $Q_1/\Phi(Q_1)$ commute and intersect trivially in $\Aut(Q_1/\Phi(Q_1))$. Set $T$ to be the subgroup of $\Aut(Q_1/Q_1\cap Q_2)$ induced by a Hall $2'$-subgroup of $N_{L_1}(S)$ so that $T$ has order $q-1$, and write $C$ the subgroup of $\Aut(Q_1/Q_1\cap Q_2)$ induced by $C_H(Z(Q_1))$. Hence, we have a subgroup of shape $C\times T$ contained in $\Aut(Q_1/Q_1\cap Q_2)\cong \GL_{3n}(2)$, and $C$ has normal subgroup isomorphic to $\SL_{3n/s}(2^s)$ which we denote $C^*$. Indeed, $C^*$ is also normalized by the subgroup of $\Aut(Q_1/Q_1\cap Q_2)$ induced by $H$, which we denote $\wt H$. Since $K_2$ restricts faithfully to $Q_1/Q_1\cap Q_2$, we have that $X:=\langle \wt H, T\rangle$ contains a Singer cycle of $\Aut(Q_1/Q_1\cap Q_2)$. Applying \cite{KantorSinger} again, and using that $C^*\normaleq X$, we must have that $\GL_{3n/s}(2^s)\normaleq X\normaleq \Gamma L_{3n/s}(2^s)$ and as $T$ centralizes $C^*$ and intersects $C^*$ trivially, we have a contradiction. This completes the proof.
\end{proof}

\begin{proposition}\label{NormalThompson}
Suppose that $\fs$ is a saturated fusion system supported on $S$ and let $E\in\mathcal{E}(\fs)$. Then either $E=Q_2$ or $\Phi(Q_1)\le J_{\normaleq}(E)\le Q_1$.
\end{proposition}
\begin{proof}
We suppose that $E\not\le Q_2$, $J_{\normaleq}(E)\le Q_2$ and there is $A\in\mathcal{A}_{\normaleq}(E)$ with $A\not\le Q_1$. Indeed, it follows that $Q_2=A(Q_1\cap Q_2)$. Note, if $Q_1\cap Q_2\not\le E$ then, $[E, Q_1\cap Q_2]\le \Phi(Q_1)\le J_{\normaleq}(E)$. But now, $Z(S)=[A, \Phi(E)]\le [J_{\normaleq}(E), \Phi(Q_1)]\le \Phi(J_{\normaleq}(E))\le Z(S)$. Hence, $Q_1\cap Q_2$ centralizes the chain $\{1\}\normaleq \Phi(J_{\normaleq}(E))\normaleq J_{\normaleq}(E)\normaleq E$, a contradiction. Thus, $Q_2=A(Q_1\cap Q_2)\le E$ and by \cref{Q2Contain}, $E=Q_2$, a contradiction.
\end{proof}

\begin{lemma}\label{EcontainQ1}
If $\Phi(Q_1)\le J_{\normaleq}(E)\not\le Q_1\cap Q_2$, then $E\le Q_1$ and $E\cap Q_2>\Phi(Q_1)$.
\end{lemma}
\begin{proof}
By \cref{NormalThompson}, $Z(Q_1)\le Z(J_{\normaleq}(E))\le\Phi(Q_1)\le J_{\normaleq}(E)\le E\cap Q_1$. Aiming for a contradiction, we assume that that there is $A\in\mathcal{A}_{\normaleq}(E)$ with $A\not\le Q_2$. If $E\not\le Q_1$, then for $a\in A\setminus (A\cap Q_2)$, we have that for $D_a:=C_{\Phi(Q_1)}(a)$, $|D_a|=q^4$ and $L:=C_{Q_1\cap Q_2}(D_a)=[E, a]\Phi(Q_1)$. Indeed, we have that $|(L\cap A)D_a|\geq q^5$ so that $\langle a\rangle(L\cap A)D_a$ is elementary abelian of order strictly larger than $q^5$, a contradiction. Hence, $E\le Q_1$. 

Since $\Phi(Q_1)<E< Q_1$ we have that $Z(Q_1)\le Z(E)<\Phi(Q_1)$. Aiming for a contradiction, suppose that $E\cap Q_2=\Phi(Q_1)$. Then $Q_1\cap Q_2$ normalizes $E$, $|E(Q_1\cap Q_2)/E|=q^3$ and $\Phi(Q_1)$ has index at most $q^3$ in $E$. Since $Z(E)\le \Phi(Q_1)$ and $[Q_1\cap Q_2, \Phi(Q_1)]=Z(S)$, an application of \cref{SEFF} to $E/Z(E)$ and $Z(E)$ yields that $Z(E)=Z(Q_1)\ge \Phi(E)$, $|E/\Phi(Q_1)|=q^3$, $N_S(E)=Q_1=E(Q_1\cap Q_2)$ and $O^{p'}(\Out_{\fs}(E))\cong \SL_2(q^3)$. Since $E$ is maximally chosen, the only other essential of $\fs$ that could contain $E$ is $Q_1$.

If $Q_1$ is essential in $\fs$, then as $[S, \Phi(Q_1)]\le Z(Q_1)$, we deduce that $\Aut_{\fs}(Q_1)$ acts trivially on $\Phi(Q_1)/Z(Q_1)$ and so normalizes the subspaces described in \cref{Phi1Cent}. In particular, for $D\le \Phi(Q_1)$ of order $q^4$ with $C_{Q_1}(D)$ of order $q^7$, we have that $\Aut_{\fs}(Q_1)$ normalizes $C_{Q_1}(D)$. By \cref{SEFF}, $\Out_{\fs}(Q_1)$ induces a natural $\SL_2(q)$-action on $C_{Q_1}(D)/\Phi(Q_1)$. Choose $\alpha\in \Aut_{\fs}(Q_1)$ with $(E\cap C_{Q_1}(D))\alpha=C_{Q_1\cap Q_2}(D)$. Since $Q_1=N_S(E)$, $E$ is fully $\fs$-normalized and $E\le Q_1$, it follows that $E\alpha$ is also essential in $\fs$. Moreover, $Z(E)\alpha=Z(E\alpha)=Z(Q_1)$ and since $E$ is fully normalized, $E\alpha\ne Q_1\cap Q_2$. But now, for all $e\in E\alpha\setminus C_{Q_1\cap Q_2}(D)$, $[e, E\alpha\cap Q_2]\le [E\alpha, E\alpha]=[E, E]\alpha=Z(Q_1)$. Then $|C_{E\alpha\cap Q_2/Z(S)}(e)|\geq q^4$ and \cref{Q2Cent} provides a contradiction. Hence, $Q_1$ is not essential in $\fs$ and $E$ is maximally essential.

Let $r$ be an element of order $q-1$ in $O^{p'}(\Aut_{\fs}(E))$ which normalizes $\Aut_S(E)=\Aut_{Q_1}(E)$. Using the receptiveness of $E$ and the Alperin--Goldschmidt theorem, $r$ lifts an automorphism $\hat{r}\in\Aut_{\fs}(S)$ such that $\hat{r}|_E=r$. Since $r$ centralizes $Z(Q_1)$, we have that $[\hat{r}, S]\le C_S(Z(Q_1))=Q_1$ by the three subgroups lemma. Since $\hat{r}|_E=r$ and $Q_1=N_S(E)$, we have that $[\hat{r}, Q_1]\le E$. Hence, $[\hat{r}, S]\le E$. Since $Q_2$ is characteristic in $S$, $\hat{r}$ normalizes $Q_2$ so that $[\hat{r}, Q_2]\le [\hat{r}, S]\cap Q_2=E\cap Q_2=\Phi(Q_1)$. But then $[\hat{r}, Q_2, \Phi(Q_1)]=\{1\}$, $[Q_2, \Phi(Q_1), \hat{r}]=\{1\}$ so that $[\hat{r}, \Phi(Q_1), Q_2]=\{1\}$ by the three subgroups lemma. But $C_{\Phi(Q_1)}(Q_2)=Z(S)$ and $\Phi(Q_1)=[\Phi(Q_1), r]Z(Q_1)$, and we have a contradiction.
\end{proof}

\begin{lemma}\label{EinQ_1}
If $\Phi(Q_1)\le J_{\normaleq}(E)\not\le Q_1\cap Q_2$ then $p=2$, $Z(E)=Z(Q_1)$, $Q_1=N_S(E)$ and $O^{2'}(\Aut_{\fs}(E))$ acts trivially on $Z(Q_1)$.
\end{lemma}
\begin{proof}
We have that $E\cap Q_2>\Phi(Q_1)$ by \cref{EcontainQ1} so that $Z(Q_1)=[E, \Phi(Q_1)]\le \Phi(E)\le \Phi(Q_1)$. By \cref{Phi1Cent}, we deduce that $|Z(E)|\in\{q^2, q^3, q^4\}$. Note that $C_{Q_1}(E/Z(E))$ centralizes $E/Z(E)\Phi(E)$ and $Z(E)\Phi(E)/\Phi(E)$ and so we deduce by \cref{burnside} that $C_{Q_1}(E/Z(E))\le E$ and $E/Z(E)$ contains a non-central chief factor for $O^{p'}(\Aut_{\fs}(E))$. Applying \cref{SEFF}, since $[Q_1, Z(E)\Phi(E)]\le Z(Q_1)\le \Phi(E)$, we have that $|Q_1/E|\leq |[E, Q_1]Z(E)/Z(E)|$ and we deduce that $|Z(E)|\in\{q^2, q^3\}$. Moreover, if $|Z(E)|=q^3$, then \cref{SEFF} implies that $E=C_{Q_1}(Z(E))$ and $Q_1=N_S(E)$, a contradiction since in this instance, $E\normaleq S$. Hence, $Z(E)=Z(Q_1)$.

If $p$ is odd, then for $a\in E\setminus E\cap Q_2$, we have that $|[a, Q_2/Z(Q_1)|=q^4$. We deduce that $|C_{Q_2/Z(Q_1)}(a)|=q^3$ and as $\Phi(Q_1)/Z(Q_1)\le C_{Q_2/Z(Q_1)}(a)$, we have $\Phi(Q_1)/Z(Q_1)=C_{Q_2/Z(Q_1)}(a)$. It follows that $Z(E/Z(Q_1))=\Phi(Q_1)/Z(Q_1)$ and that $Q_1$ centralizes the chain $\{1\}\normaleq Z(Q_1)\normaleq \Phi(Q_1)\normaleq E$, a contradiction since $E<Q_1$. Hence, $p=2$. Clearly, $N_S(E)\in\{S, Q_1\}$. Assume that $E\normaleq S$. Since $Z(Q_1)=Z(E)$ and $Q_1/\Phi(Q_1)$ is a direct of natural $\SL_2(q)$-modules for $L_1/Q_1$, we infer that $C_{\Phi(Q_1)}(Q_1\cap Q_2\cap E)=Z(Q_1)$. Indeed, a calculation as in the odd case yields that either $Z(E/Z(Q_1))=\Phi(Q_1)/Z(Q_1)$ and we obtain a contradiction as before, or $E\le C_{Q_1}(C_{\Phi(Q_1)}(a))(Q_1\cap Q_2)$ and the preimage of $Z(E/Z(Q_1))$ is equal to $C_{Q_1\cap Q_2}(C_{\Phi(Q_1)}(a))$. But now, $[Q_1, E]\le C_{Q_1\cap Q_2}(C_{\Phi(Q_1)}(a))$, $[Q_1, \Phi(Q_1)]=Z(Q_1)=Z(E)$, $\Phi(Q_1)$ has index $q$ in $C_{Q_1\cap Q_2}(C_{\Phi(Q_1)}(a))$ and $|Q_1/E|\geq q^2$. Hence, \cref{SEFF} provides a contradiction. Therefore, $p=2$ and $Q_1=N_S(E)$. Indeed, since $Q_1$ centralizes $Z(Q_1)$, we deduce that $O^{2'}(\Aut_{\fs}(Q_1))$ acts trivially on $Z(Q_1)$.
\end{proof}

\begin{lemma}\label{NotEssen}
We have that $\Phi(Q_1)\le J_{\normaleq}(E)\le Q_1\cap Q_2$.
\end{lemma}
\begin{proof}
Note first that $(E\cap Q_2)/Z(S)$ is elementary abelian and $E\le Q_1$ by \cref{EinQ_1}. Assume that there is $F\normaleq E$ with $F\not\le Q_2$, $F/Z(S)$ elementary abelian and $|F|\geq |E\cap Q_2|$. Let $f\in F\setminus Q_2$. Then $(\Phi(Q_1)\cap F)/Z(S)\le C_{\Phi(Q_1)/Z(S)}(f)=C_{\Phi(Q_1)}(f)/Z(S)$ by \cref{Phi1Cent}. Moreover, by \cref{Q2Cent}, we deduce that $|F/F\cap Q_2|\leq q^5/|F\cap \Phi(Q_1)|$. Hence, \[|F|\leq q^5|F\cap Q_2|/|\Phi(Q_1)\cap F|=|\Phi(Q_1)||F\cap Q_2|/|\Phi(Q_1)\cap F|=|\Phi(Q_1)(F\cap Q_2)|\leq |E\cap Q_2|\leq |F|\] and we infer that every inequality is an equality. Therefore, $E\cap Q_2=(F\cap Q_2)\Phi(Q_1)$. Using that $E\cap Q_2>\Phi(Q_1)$, for $e\in (F\cap Q_2)\setminus \Phi(Q_1)$, we have that $F/Z(S)\le C_{Q_1/Z(S)}(e)$. By \cref{Q2Cent}, we have that $C_{Q_1/Z(S)}(e)=C_{Q_1}(C_{\Phi(Q_1)}(e))(Q_1\cap Q_2)/Z(S)$. But $(F\cap Q_2)/Z(S)\le C_{(Q_1\cap Q_2)/Z(S)}(f)$ and another application of \cref{Q2Cent} implies that $F\cap Q_2\le C_{Q_1\cap Q_2}(C_{\Phi(Q_1)}(e))$ and we deduce that $F\le C_{Q_1}(C_{\Phi(Q_1)}(e))$. Then $F\cap \Phi(Q_1)=C_{\Phi(Q_1)}(f)=C_{\Phi(Q_1)}(e)=Z(F(E\cap Q_2))$ is of order $q^4$. Indeed, $F\cap \Phi(Q_1)=Z(E\cap Q_2)$ does not depend on $F$.

Since $Z(S)$ is normalized by $O^{2'}(\Aut_{\fs}(E))$ so too is $Z(E\cap Q_2)$ and hence, so too is $C_E(Z(E\cap Q_2))$. Now, $[Q_1\cap Q_2, E]\le \Phi(Q_1)\le C_E(Z(E\cap Q_2))$ and $[Q_1\cap Q_2, Z(E\cap Q_2)]\le Z(S)$ and so each odd order element of $\langle \Aut_{Q_1\cap Q_2}(E)^{\Aut_{\fs}(E)}\rangle$ acts faithfully on $C_E(Z(E\cap Q_2))/Z(E\cap Q_2)$. But $|(Q_1\cap Q_2)E/E|\geq q^2$ and $E\cap Q_2$ has index $q$ in $C_E(Z(E\cap Q_2))$ and is centralized modulo $Z(E\cap Q_2)$ by $Q_1\cap Q_2$. An application of \cref{SEFF} provides a contradiction.

Hence, $E\cap Q_2/Z(S)=J_{\normaleq}(E/Z(S))$ and since $Z(S)$ is $O^{2'}(\Aut_{\fs}(E))$-invariant, so too is $E\cap Q_2$. Then $Q_1\cap Q_2$ centralizes the chain $\{1\}\normaleq Z(S)\normaleq E\cap Q_2\normaleq E$ so that $Q_1\cap Q_2\le E$ by \cref{Chain} and $Z(Q_1)=Z(E)$. Note that for $a\in Q_1\setminus (Q_1\cap Q_2)$ since $|C_{Q_1\cap Q_2/Z(S)}(a)|\leq q^4$, we deduce that $|[a, Q_1\cap Q_2]Z(S)|\geq q^4$ so that $|\Phi(E)|\geq q^4$. Letting $b\in Q_1\setminus (Q_1\cap Q_2)$ with $C_{Q_2/Z(S)}(a)\ne C_{Q_2/Z(S)}(b)$, we have that $[a, Q_1\cap Q_2][b, Q_1\cap Q_2]Z(S)=\Phi(Q_1)$ so that $\Phi(E)=\Phi(Q_1)$. In this scenario, $Q_1$ centralizes the chain $\{1\}\normaleq Z(Q_1)\normaleq \Phi(Q_1)\normaleq E$, a contradiction. Letting $A\in\mathcal{A}_{\normaleq}(E)$ with $A\not\le Q_1\cap Q_2$, it follows that for $a\in A\setminus (A\cap Q_2)$, $E\le C_{Q_1}(C_{\Phi(Q_1)}(a))(Q_1\cap Q_2)$ and $|Q_1/E|\geq q^2$. But now, $[Q_1, E]=\Phi(Q_1)$ and since $|Q_1/E|=q^2>q=|\Phi(Q_1)/\Phi(E)|$, an application of \cref{SEFF} gives a contradiction. Thus, we have shown that $\Phi(Q_1)\le J_{\normaleq}(E)\le Q_1\cap Q_2$.
\end{proof}

As promised, we finally demonstrate that $\mathcal{E}(\fs)\subseteq \{Q_1, Q_2\}$.

\begin{proposition}\label{D4Essen}
Suppose that $\fs$ is a saturated fusion system supported on $S$ and let $E\in\mathcal{E}(\fs)$. Then $E\in\{Q_1, Q_2\}$.
\end{proposition}
\begin{proof}
We suppose, ultimately for a contradiction, that $E$ is an essential subgroup chosen maximally such that $Q_1\ne E\ne Q_2$. By \cref{NotEssen}, $\Phi(Q_1)\le J_{\normaleq}(E)\le Q_1\cap Q_2$. Assume first that $\Phi(Q_1)=J_{\normaleq}(E)$ and $E\cap Q_2>\Phi(Q_1)$. Since $E\not\le Q_2$, it follows that $[E, \Phi(Q_1)]=Z(Q_1)$. If $Q_1\not\le E$, then  $\langle \Aut_{Q_1}(E)^{\Aut_{\fs}(E)}\rangle$ contains elements of order coprime to $p$ and we may choose such an element $r$. Since $Q_1$ centralizes the chain $\{1\}\normaleq Z(Q_1)\normaleq \Phi(Q_1)$, we deduce that $r$ centralizes $\Phi(Q_1)$. But then, by the three subgroups lemma, $[r, E]\le C_E(\Phi(Q_1))\le C_S(\Phi(Q_1))=\Phi(Q_1)$ so that $r$ acts trivially on $E$, a contradiction. Hence, $Q_1<E$ and $Z(E)=Z(S)$. But then $\Phi(S)=Q_1\cap Q_2=[E, Q_1]\le \Phi(E)$ and $S$ centralizes $E/\Phi(E)$, a contradiction. Hence, if $\Phi(Q_1)=J_{\normaleq}(E)$ then $E\cap Q_2=\Phi(Q_1)$. But now, $[E, Q_1\cap Q_2]\le \Phi(Q_1)$ and $[Q_1\cap Q_2, \Phi(Q_1)]=Z(S)$ and since $|Z(S)|=q<|(Q_1\cap Q_2)E/E|=q^3$, applying \cref{SEFF} yields a contradiction.

Therefore, we have that $\Phi(Q_1)<J_{\normaleq}(E)\le Q_1\cap Q_2$. Then $\Phi(J_{\normaleq}(E))=Z(S)$ is characteristic in $E$. Since $[E, Q_1\cap Q_2]\le \Phi(Q_1)\le J_{\normaleq}(E)$, and $[Q_1\cap Q_2, J_{\normaleq}(E)]\le \Phi(Q_1\cap Q_2)=Z(S)=\Phi(J_{\normaleq}(E))$, $Q_1\cap Q_2$ centralizes the chain $\{1\}\normaleq \Phi(J_{\normaleq}(E))\normaleq J_{\normaleq}(E)\normaleq E$ so that $Q_1\cap Q_2\le E$ and $E\normaleq S$. Now, if $Q_2\not\le E$ then $\langle \Aut_{Q_2}(E)^{\Aut_{\fs}(E)}\rangle$ contains elements of order coprime to $p$. Choose such an element $r$. Since $[Q_2, J_{\normaleq}(E)]\le \Phi(Q_2)=Z(S)=\Phi(J_{\normaleq}(E))$, it follows that $[r, J_{\normaleq}(E)]=\{1\}$. But then, by the three subgroups lemma, $[r, E]\le C_{E}(J_{\normaleq}(E))\le J_{\normaleq}(E)$ and $r$ acts trivially on $E$, a contradiction. Hence, $Q_2\le E$ and \cref{Q2Contain} provides a contradiction.
\end{proof}

Finally, we classify all saturated fusion systems supported on a Sylow $p$-subgroup of ${}^3\mathrm{D}_4(p^n)$. As remarked in the Introduction and after \cref{MainThm}, we need to employ a $\mathcal{K}$-group hypothesis when $p$ is an odd prime to deduce $O^{p'}(\Out_{\fs}(Q_2))\cong \SL_2(q^3)$ acting on $Q_2/Z(S)$ as a triality module.

\begin{theorem}\label{3D4}
Let $\fs$ be a saturated fusion system on a Sylow $p$-subgroup of ${}^3\mathrm{D}_4(p^n)$ and assume that $\Aut_{\fs}(Q_2)$ is a $\mathcal{K}$-group if $p$ is odd. Then either:
\begin{enumerate}
\item $\fs=\fs_S(S:\Out_{\fs}(S))$;
\item $\fs=\fs_S(Q_1: \Out_{\fs}(Q_1))$ where $O^{p'}(\Out_{\fs}(Q_1))\cong \SL_2(p^n)$;
\item $\fs=\fs_S(Q_2: \Out_{\fs}(Q_2))$ where $O^{p'}(\Out_{\fs}(Q_2))\cong \SL_2(p^{3n})$;
\item $\fs=\fs_S(G)$ where $F^*(G)=O^{p'}(G)={}^3\mathrm{D}_4(p^n)$.
\end{enumerate}
\end{theorem}
\begin{proof}
If $\mathcal{E}(\fs)=\emptyset$ then outcome (i) holds. Assume that $Q_1$ is essential. Since $\Phi(Q_1)$ is self-centralizing, it follows from the three subgroups lemma that $p'$-elements of $\Aut_{\fs}(Q_1)$ act faithfully $\Phi(Q_1)$. Since $S$ centralizes $\Phi(Q_1)/Z(Q_1)$, we deduce that $p'$-elements of $O^{p'}(\Aut_{\fs}(Q_1))$ act faithfully on $Z(Q_1)$ and since $C_S(Z(Q_1))=Q_1$, \cref{SEFF} yields that $Z(Q_1)$ is a natural module for $O^{p'}(\Out_{\fs}(Q_1))\cong \SL_2(q)$. If $Q_1$ is the only essential then outcome (ii) holds.

Assume that $Q_2$ is essential. Since $\Out_{\fs}(Q_2)$ is a $\mathcal{K}$-group when $p$ is odd and $\Out_S(Q_2)$ is elementary abelian of order $q^3$, we have by \cref{MaxEssen}, \cref{MaxEssenEven} and \cite[(2.5), (3.3)]{parkerSE} (which in turn uses \cite[Theorem 7.6.1]{GLS3}) that $O^{p'}(\Out_{\fs}(Q_2))\cong \mathrm{(P)}SL_2(q^3)$. Applying \cref{MaxEssen} and \cref{SL2ModRecog}, there is a unique non-trivial composition factor in $Q_2/Z(S)$ for $O^{p'}(\Out_{\fs}(Q_2))$. Set $U:=\langle Z(Q_1)/Z(S)^{\Aut_{\fs}(Q_2)}\rangle$. Then $C_U(\Out_S(Q_2))=Z(Q_1)/Z(S)$ has order $q$ and comparing with \cref{SL2ModRecog}, either $U=Z(Q_1)/Z(S)$ and $O^{p'}(\Out_{\fs}(Q_2))$ is trivial on $Z(Q_1)/Z(S)$; or $U=Q_2/Z(S)$ is a triality module for $O^{p'}(\Out_{\fs}(Q_2))\cong \SL_2(q^3)$. With the aim of forcing a contradiction, we assume that $U=Z(Q_1)/Z(S)$. Then for $r$ of $p'$-order in $O^{p'}(\Out_{\fs}(Q_2))$, we have by coprime action that $[r, Z(Q_1)]=\{1\}$ and the three subgroups lemma yields that $[r, Q_2]\le Q_1\cap Q_2$. Since $Q_1\cap Q_2/Z(Q_1)$ is of order $q^6$, applying \cref{SL2ModRecog} we have that $Q_1\cap Q_2/Z(Q_1)$ is the unique non-trivial composition factor in $Q_2/Z(S)$ for $O^{p'}(\Out_{\fs}(Q_2))$. Since $\Out_S(Q_2)$ is quadratic on $Q_1\cap Q_2/Z(Q_1)$ we deduce that $O^{p'}(\Out_{\fs}(Q_2))\cong \SL_2(q^3)$ and $Q_1\cap Q_2/Z(Q_1)$ has the structure of a natural module. If $p=2$ then for all $x\in \Out_S(Q_2)$ we have that $|C_{Q_1\cap Q_2/Z(S)}(x)|=q^4$ and by properties of natural $\SL_2(q)$-modules, we conclude that $C_{Q_1\cap Q_2/Z(S)}(x)=\Phi(Q_1)/Z(S)$, a contradiction since $[\Phi(Q_1), \Out_S(Q_2)]=Z(Q_1)$. If $p$ is odd, then $Q_1\cap Q_2/Z(S)$ splits into a direct sum of a natural module and $Z(Q_1)/Z(S)=C_{Q_1\cap Q_2/Z(S)}(O^{p'}(\Out_{\fs}(Q_2)))$ and so we deduce again that $C_{Q_1\cap Q_2/Z(S)}(\Out_S(Q_2))=\Phi(Q_1)/Z(S)$, again a contradiction. Hence, $Q_2/Z(S)$ is a triality module for $O^{p'}(\Out_{\fs}(Q_2))\cong \SL_2(q^3)$. If $Q_1$ is not essential, then outcome (iii) holds.

Hence, we are left with the case where $\mathcal{E}(\fs)=\{Q_1, Q_2\}$. If $O_p(\fs)\ne\{1\}$, then $O_p(\fs)\cap Z(S)\ne\{1\}$ and since $Z(Q_1)$ is a natural $\SL_2(q)$-module for $O^{p'}(\Out_{\fs}(Q_1))$, we have that $Z(Q_1)\le O_p(\fs)$. Then, as $\Out_{\fs}(Q_2)$ is irreducible on $Q_2/Z(S)$, we deduce that $Q_2\le O_p(\fs)\not\le Q_1$, a contradiction by \cite[Proposition I.4.5]{ako}. Hence, $O_p(\fs)=\{1\}$ and $\fs$ satisfies outcome (iv) by \cref{3D4Cor}.
\end{proof}

\printbibliography

\end{document}